\documentclass[letterpaper, 11pt,  reqno]{amsart}

\usepackage{amsmath,amssymb,amscd,amsthm,amsxtra, esint}

\usepackage{hyperref} 
\usepackage[margin=1.1in,marginparwidth=1.5cm, marginparsep=0.5cm]{geometry}
\usepackage{scalerel} 
\allowdisplaybreaks[2]

\newtheorem{theorem}{Theorem} [section]

\newtheorem{lemma}[theorem]{Lemma}
\newtheorem{proposition}[theorem]{Proposition}
\newtheorem{remark}[theorem]{Remark}

\newtheorem{corollary}[theorem]{Corollary}

\newcommand{\1}{\hspace{0.5mm}\text{I}\hspace{0.5mm}}
\newcommand{\II}{\text{I \hspace{-2.8mm} I} }
\newcommand{\III}{\text{I \hspace{-2.9mm} I \hspace{-2.9mm} I}}

\newcommand{\noi}{\noindent}
\newcommand{\Z}{\mathbb{Z}}
\newcommand{\R}{\mathbb{R}}

\newcommand{\T}{\mathbb{T}}

\let\Re=\undefined\DeclareMathOperator*{\Re}{Re}
\let\Im=\undefined\DeclareMathOperator*{\Im}{Im}

\newcommand{\RR}{\mathcal{R}}

\newcommand{\NN}{\mathcal{N}}

\newcommand{\F}{\mathcal{F}}

\def\norm#1{\|#1\|}

\newcommand{\al}{\alpha}
\newcommand{\be}{\beta}
\newcommand{\dl}{\delta}

\newcommand{\nb}{\nabla}

\newcommand{\eps}{\varepsilon}

\newcommand{\g}{\gamma}
\newcommand{\G}{\Gamma}
\newcommand{\ld}{\lambda}
\newcommand{\Ld}{\Lambda}
\newcommand{\s}{\sigma}

\newcommand{\ft}{\widehat}

\newcommand{\cj}{\overline}
\newcommand{\dx}{\partial_x}
\newcommand{\dt}{\partial_t}

\newcommand{\embeds}{\hookrightarrow}

\renewcommand{\t}{\tau}

\renewcommand{\l}{\ell}

\renewcommand{\t}{\tau}

\newcommand{\les}{\lesssim}
\newcommand{\ges}{\gtrsim}

\newcommand{\jb}[1]
{\langle #1 \rangle}

\newcommand{\ind}{\mathbf 1}

\newcommand{\M}{\mathcal{M}}
\def\e{\varepsilon}

\newcommand{\N}{\mathbb{N}}

\newcommand{\J}{\mathcal{J}}

\newcommand{\nm}{n_{\text{max}}}

\usepackage{tikz}

\usetikzlibrary{shapes.misc}
\usetikzlibrary{shapes.symbols}
\usetikzlibrary{shapes.geometric}
\tikzset{
	dot/.style={circle,fill=black,draw=black,inner sep=0pt,minimum size=0.5mm},
	>=stealth,
	}
\tikzset{
	ddot/.style={circle,fill=white,draw=black,inner sep=0pt,minimum size=0.8mm},
	>=stealth,
	}

\tikzset{decision/.style={ 
        draw,
        diamond,
        aspect=1.5
    }}

\tikzset{dia2/.style
={diamond,fill=white,draw=black,inner sep=0pt,minimum size=1mm},
	>=stealth,
	}

\tikzset{dia/.style
={star,fill=black,draw=black,inner sep=0pt,minimum size=1mm},
	>=stealth,
	}

\makeatletter
\def\DeclareSymbol#1#2#3{\expandafter\gdef\csname MH@symb@#1\endcsname{\tikz[baseline=#2,scale=0.15]{#3}}}
\def\<#1>{\csname MH@symb@#1\endcsname}
\makeatother

\DeclareSymbol{X}{-2.4}{\node[dot] {};}
\DeclareSymbol{1}{0}{\draw[white] (-.4,0) -- (.4,0); \draw (0,0)  -- (0,1.2) node[dot] {};}
\DeclareSymbol{2}{0}{\draw (-0.5,1.2) node[dot] {} -- (0,0) -- (0.5,1.2) node[dot] {};}

\DeclareSymbol{1}{-2.7}
 {
  \draw (0,0) node[dot]{};}

\DeclareSymbol{1'}{-2.7}
 {\draw (0,0) node[ddot]{};}

\DeclareSymbol{1''}{-2.7}
 {\draw (0,0) node[dia]{};}

\DeclareSymbol{3}{0}
 {\draw (0,0) node[dot]{} -- (0,1.2) node[ddot] {}; 
 \draw (-1.2,0) node[dot] {} -- (0,1.2)node[ddot] {} -- (1.2,0) node[dot] {};}

\DeclareSymbol{3'}{0}
 {\draw (0,0) node[dia]{} -- (0,1.2) node[ddot] {}; 
 \draw (-1.2,0) node[dia] {} -- (0,1.2)node[ddot] {} -- (1.2,0) node[dia] {};}

\DeclareSymbol{31}{-3}{
\draw (0,-1)node[dot] {} -- (0,0) node[ddot] {}-- (0.9, -1) node[dot] {}; 
\draw (0,-1)node[dot] {} -- (0,0) node[ddot] {}-- (-0.9, -1) node[dot] {}; 
\draw (0,0)node[ddot] {} -- (1.3,1) node[ddot] {}-- (1.3, 0) node[dot] {}; 
\draw  (1.3,1) node[ddot] {}-- (2.6, 0) node[dot] {}; 
}

\DeclareSymbol{32}{-3}{
\draw (0,-1)node[dot] {} -- (0,0) node[ddot] {}-- (0.9, -1) node[dot] {}; 
\draw (0,-1)node[dot] {} -- (0,0) node[ddot] {}-- (-0.9, -1) node[dot] {}; 
\draw (0,0)node[ddot] {} -- (0,1) node[ddot] {}-- (0.9, 0) node[dot] {}; 
\draw (0,0)node[ddot] {} -- (0,1) node[ddot] {}-- (-0.9, 0) node[dot] {}; 
}

\DeclareSymbol{33}{-3}{
\draw (2.6,-1)node[dot] {} -- (2.6,0) node[ddot] {}-- (3.5, -1) node[dot] {}; 
\draw (2.6,-1)node[dot] {} -- (2.6,0) node[ddot] {}-- (1.7, -1) node[dot] {}; 
\draw (0,0)node[dot] {} -- (1.3,1) node[ddot] {}-- (1.3, 0) node[dot] {}; 
\draw  (1.3,1) node[ddot] {}-- (2.6, 0) node[ddot] {}; 
}

\DeclareSymbol{3''}{0}
 {\draw (0,0) node[dia]{} -- (0,1.2) node[ddot] {}; 
 \draw (-1.2,0) node[dot] {} -- (0,1.2)node[ddot] {} -- (1.2,0) node[dia] {};}

\DeclareSymbol{3'''}{0}
 {\draw (0,0) node[dot]{} -- (0,1.2) node[ddot] {}; 
 \draw (-1.2,0) node[dot] {} -- (0,1.2)node[ddot] {} -- (1.2,0) node[dia] {};}

\DeclareSymbol{31'}{-3}{
\draw (0,-1.2)node[dia] {} -- (0,0) node[ddot] {}-- (0.9, -1.2) node[dia] {}; 
\draw (0,-1.2)node[dia] {} -- (0,0) node[ddot] {}-- (-0.9, -1.2) node[dia] {}; 
\draw (0,0)node[ddot] {} -- (1.3,1.2) node[ddot] {}-- (1.3, 0) node[dia] {}; 
\draw  (1.3,1.2) node[ddot] {}-- (2.6, 0) node[dia] {}; 
}

\DeclareSymbol{32'}{-3}{
\draw (0,-1.2)node[dia] {} -- (0,0) node[ddot] {}-- (0.9, -1.2) node[dia] {}; 
\draw (0,-1.2)node[dia] {} -- (0,0) node[ddot] {}-- (-0.9, -1.2) node[dia] {}; 
\draw (0,0)node[ddot] {} -- (0,1.2) node[ddot] {}-- (0.9, 0) node[dia] {}; 
\draw (0,0)node[ddot] {} -- (0,1.2) node[ddot] {}-- (-0.9, 0) node[dot] {}; 
}

\DeclareSymbol{33'}{-3}{
\draw (2.6,-1.2)node[dia] {} -- (2.6,0) node[ddot] {}-- (3.5, -1.2) node[dia] {}; 
\draw (2.6,-1.2)node[dia] {} -- (2.6,0) node[ddot] {}-- (1.7, -1.2) node[dia] {}; 
\draw (0,0)node[dot] {} -- (1.3,1.2) node[ddot] {}-- (1.3, 0) node[dot] {}; 
\draw  (1.3,1.2) node[ddot] {}-- (2.6, 0) node[ddot] {}; 
}

\newtheorem*{ackno}{Acknowledgments}

\numberwithin{equation}{section}
\numberwithin{theorem}{section}

\begin{document}

\title[GWP of 1D cubic fractional NLS in negative Sobolev spaces]
{Global well-posedness of one-dimensional cubic fractional nonlinear Schr\"odinger equations in negative Sobolev spaces}

\author[E.~Brun, G.~Li, R.~Liu, and Y.~Zine]
{Enguerrand~Brun, Guopeng~Li, Ruoyuan~Liu, and  Younes~Zine}

\address{
Enguerrand Brun\\
Ecole Normale Sup\'erieure de Lyon\\
 46, all\'ee d'Italie\\
 69364-Lyon Cedex 07\\
 France} 
\email{enguerrand.brun@ens-lyon.fr}

\address{
Guopeng Li\\
School of Mathematics\\
The University of Edinburgh\\
and The Maxwell Institute for the Mathematical Sciences\\
James Clerk Maxwell Building\\
The King's Buildings\\
Peter Guthrie Tait Road\\
Edinburgh\\ 
EH9 3FD\\
United Kingdom} 
\email{guopeng.li@ed.ac.uk}

\address{
Ruoyuan Liu\\
School of Mathematics\\
The University of Edinburgh\\
and The Maxwell Institute for the Mathematical Sciences\\
James Clerk Maxwell Building\\
The King's Buildings\\
Peter Guthrie Tait Road\\
Edinburgh\\ 
EH9 3FD\\
United Kingdom}
\email{ruoyuan.liu@ed.ac.uk}

\address{
Younes Zine\\
School of Mathematics\\
The University of Edinburgh\\
and The Maxwell Institute for the Mathematical Sciences\\
James Clerk Maxwell Building\\
The King's Buildings\\
Peter Guthrie Tait Road\\
Edinburgh\\ 
EH9 3FD\\
United Kingdom\\
and Chair of Probability and PDEs\\
EPFL SB MATH\\
PROPDE\\
MA C2 647\\
CH-1015 Lausanne\\
Switzerland}
\email{younes.zine@epfl.ch}

\begin{abstract}
We study the Cauchy problem for the cubic fractional nonlinear Schr\"odinger equation (fNLS) on the real line and on the circle. In particular, we prove global well-posedness of the cubic fNLS with all orders of dispersion higher than the usual Schr\"odinger equation in negative Sobolev spaces. On the real line, our well-posedness result is sharp in the sense that a contraction argument does not work below the threshold regularity. On the circle, due to ill-posedness of the cubic fNLS in negative Sobolev spaces, we study the renormalized cubic fNLS. In order to overcome the failure of local uniform continuity of the solution map in negative Sobolev spaces, by applying a gauge transform and partially iterating the Duhamel formulation, we study the resulting equation with a cubic-quintic nonlinearity. In proving uniqueness, we present full details justifying the use of the normal form reduction for rough solutions, which seem to be missing from the existing literature. Our well-posedness result on the circle extends those in Miyaji-Tsutsumi (2018) and Oh-Wang (2018) to the endpoint regularity.
\end{abstract}



\maketitle

\tableofcontents

\section{Introduction}

\subsection{Cubic fractional nonlinear Schr\"odinger equations}
\label{SUB:intro}
In this paper, 
we consider the Cauchy problem for the cubic fractional nonlinear Schr\"odinger equation (fNLS)
on the real line~$\R$ and on the circle $\T = \R/(2\pi \Z)$:
\begin{align}
\begin{cases}
i \dt u   = D^{\al} u  + | u |  ^2 u  \\
u|_{t = 0} = u_0, 
\end{cases}
\ (t, x) \in \R \times \M, 
\label{fnls1}
\end{align}

\noi
where $\M = \R$ or $\T$, $ \al > 0 $, $ u  $ is a complex-valued function. Here, $D^\al = ( - \partial_x^2 )^{\frac{\al}{2}}$ denotes the Fourier multiplier operator defined by
\[
 \mathcal{F} \big( D^\al f \big) (\xi) =   | \xi | ^{\al} \ft{f}(\xi) \qquad \text{for $ \xi \in \ft{\M} $} ,
\]

\noi
where $\mathcal{F}$ and $\ft{\cdot}$ denote the Fourier transform and
\begin{align*}
\ft \M = \begin{cases}
 \R & \text{if } \M = \R, \\
 \Z & \text{if } \M = \T.
\end{cases}
\end{align*}

\noi
The equation is called defocusing if the sign in front of the nonlinearity is a plus (as in \eqref{fnls1})
and is called focusing if the sign is a minus. In this paper, this sign is irrelevant in our analysis, and so we only work with the defocusing case.


The fractional Schr\"odinger equations were first introduced in the theory of the fractional quantum mechanics, where the Feynmann path integrals approach is generalized to $\al$-stable L\'evy process \cite{La}. 
Different choices of $\al$ represent different physical contexts. For instance,
when $\al=2$, 
the equation \eqref{fnls1} is the cubic nonlinear Schr\"odinger equation which appears as an important model in the study of nonlinear optics, fluids and plasma physics; see \cite{SS} for a general survey.
When $\al = 3$, the equation \eqref{fnls1} becomes the third-order NLS (3NLS) equation, 
which appears as a mathematical model for nonlinear pulse propagation phenomena in various fields of physics, especially in nonlinear optics; see \cite{HK, Ag}.
For $\al = 4$,
the equation \eqref{fnls1} corresponds to the cubic fourth-order NLS (4NLS) and has applications in the study of solitons in magnetic materials; see \cite{IK, Tur}.
We also remark that fNLS for $0<\al < 2$ also arises in various physical settings; see, for example, \cite{ES, FL, KLS, IP}.
In this paper, we only discuss the large dispersion regime of \eqref{fnls1}, i.e.~$\al>2$.

Note that the cubic fNLS \eqref{fnls1} enjoys a scaling symmetry,
which states that if $u(t, x)$ is a solution to \eqref{fnls1} on $ \R$, then $u^\ld$ defined by
\begin{align}
u^{\lambda}(t,x):=\lambda^{-\frac{\al}{2}}u( {\lambda^{-\al}}t, {\lambda}^{-1}x)
\label{scaling}
\end{align}

\noi
is also a solution to \eqref{fnls1} for any $\ld > 0$. 
Associated to this scaling symmetry, one defines the scaling-critical Sobolev index $s_c$
such that the
homogeneous $\dot H^{s_c}$-norm is invariant under the dilation symmetry. 
A simple calculation yields that
\[
s_c=\frac{1-\al}{2}.
\]

\noi
On the one hand, the scaling argument provides heuristics indicating that the cubic fNLS is well-posed in $H^s$ for $s > s_c$ and is ill-posed in $H^s$ for $s < s_c$. 
For instance, for the cubic NLS
(i.e.~\eqref{fnls1} with $\al = 2$) on $\M = \R$ or $\T$, Christ-Colliander-Tao \cite{CCT0} and Kishimoto \cite{Ki} proved
ill-posedness of \eqref{fnls1} in the sense of norm inflation in $H^s (\M)$ for $s \leq -\frac12 = s_c$; see also \cite{Oh, OWill}. 
On the other hand, these heuristics are known to often fail in negative Sobolev spaces.
Indeed, in \cite{CP}, Choffrut-Pocovnicu proved norm inflation of \eqref{fnls1} on $\M = \R$ or $\T$ above $s_c$: when $\al > 2$, norm inflation of \eqref{fnls1} holds in $H^s(\M)$ for $s<\frac{1-2\al}{6}$, which is above $s_c = \frac{1 - \al}{2}$ given $\al > 2$.

There are also results on other types of ill-posedness of \eqref{fnls1} above the
scaling critical regularity.
For instance, for the cubic NLS (i.e.~$\al = 2$) on $\T$, Christ-Colliander-Tao \cite{CCT1}
and Molinet~\cite{Mo} showed the failure of continuity of the solution map in 
$H^s(\T)$ for $s < 0$.
Moreover, Guo-Oh \cite{GO} 
and Oh-Wang \cite{OW}
proved
non-existence of solutions of the cubic NLS ($\al = 2$) and 4NLS ($\al = 4$)
on $\T$ if the initial data lies in $H^s(\T) \setminus  L^2(\T)$,
where $s \in (-\frac18 , 0)$ and $s \in (-\frac{9}{20} , 0)$, respectively.
We point out that these ill-posedness results no longer hold if we remove a
certain resonant term from the nonlinearity. See Subsection \ref{SUBSEC:introT} below for this removal procedure.
However, 
the norm inflation in
\cite{CP} mentioned above still holds even after we remove this resonant term.

Our
main interest in this paper is to establish well-posedness theory for the cubic fNLS \eqref{fnls1} on $\M = \R$ and $\T$. Let us briefly go over previous well-posedness results of \eqref{fnls1}, starting from~\eqref{fnls1} on the real line $\R$. For $\al = 2$, Tsutsumi \cite{Tsu} proved global well-posedness of~\eqref{fnls1} in $L^2 (\R)$. When $1 < \al < 2$, Cho-Hwang-Kwon-Lee \cite{CHKL} showed local well-posedness of~\eqref{fnls1} in $H^s (\R)$ for $s \geq \frac{2 - \al}{4}$. In the larger dispersion regime (i.e.~$\al > 2$), one can study well-posedness theory in negative Sobolev spaces. Indeed, Seong \cite{kihoon} showed that when $\al = 4$, the equation \eqref{fnls1} is globally well-posed in $H^s (\R)$ for $s \geq -\frac 12$. In this paper, we also establish global well-posedness of \eqref{fnls1}, but with a general $\al > 2$, in negative Sobolev spaces.

The cubic fNLS \eqref{fnls1} on the circle $\T$ has also been studied. For $\al = 2$, Bourgain \cite{BO1} proved global well-posedness for \eqref{fnls1} in $L^2 (\T)$. When $1 < \al < 2$, Demirbas-Erdo\v{g}an-Tzirakis \cite{DET} proved local well-posedness of \eqref{fnls1} in $H^s (\T)$ for $s > \frac{2 - \al}{4}$ and global well-posedness of \eqref{fnls1} for $s > \frac{5\al + 1}{12}$. In the same range of $\al$, the authors in \cite{CHKL} then showed local well-posedness of \eqref{fnls1} at the endpoint $s = \frac{2 - \al}{4}$. The larger dispersion regime (i.e.~$\al > 2$) of \eqref{fnls1} was also studied in some previous works, including Miyaji-Tsutsumi \cite{MT1, MT} with a third order dispersion (i.e.~$\al = 3$; see Remark \ref{RMK:3NLS} below) and Oh-Wang \cite{OW} with 4NLS (i.e.~$\al = 4$). These works will also be mentioned in more details below. In this paper, we establish global well-posedness of the cubci fNLS \eqref{fnls1} with a general $\al > 2$ in negative Sobolev spaces.


\subsection{fNLS on the real line}
\label{SUBSEC:fNLS_R}
In this section, we consider well-posedness for the cubic fNLS~\eqref{fnls1} on the real line $\R$ in negative Sobolev spaces.
We begin by stating our local well-posedness result for the cubic fNLS \eqref{fnls1}.

\begin{theorem}[Local well-posedness on $\R$]
\label{TM:LWP1}
Let $\al>2$ and $\frac{2-\al}{4} \leq s<0$.
Then, the
cubic fNLS \eqref{fnls1}
is locally well-posed in $H^s(\R )$. More precisely, given any $u_0 \in H^s (\R)$, there exists $T = T(\| u_0 \|_{H^s}) > 0$ and a unique solution $u \in C([-T, T]; H^s (\R))$ to \eqref{fnls1} with $u|_{t = 0} = u_0$.
\end{theorem}

Our proof of Theorem \ref{TM:LWP1} is based on the Fourier restriction norm method, and so the uniqueness statement in Theorem \ref{TM:LWP1} holds only in (the local-in-time version of) the relevant $X^{s, b}$-space; see Subsection \ref{SUBSEC:spaces} below.
The main task in proving local well-posedness in Theorem \ref{TM:LWP1} is to establish a trilinear estimate; see Proposition \ref{PROP:triR} below. 

We point out that the regularity $s = \frac{2 - \al}{4}$ in Theorem \ref{TM:LWP1} is a threshold below which local well-posedness of \eqref{fnls1} on $\R$ cannot be established by a contraction argument; see Remark \ref{RMK:low} below. 
Note that in \cite{kihoon}, the author showed that the 4NLS (i.e.~\eqref{fnls1} with $\al=4$) is locally well-posed in $H^s(\R)$ for $s \geq -\frac12$, which agrees with Theorem \ref{TM:LWP1} when $\al=4$.

\begin{remark} \rm
\label{RMK:low}
Our local well-posedness result in Theorem \ref{TM:LWP1} is sharp in the sense that the solution map fails to be locally uniformly continuous in $H^s (\R)$ for $\frac{(2 - 3 \al)(1 - \al)}{4 (1 - 2\al)} < s < \frac{2 - \al}{4}$, where $\al > 2$. This can be shown from the same way as in \cite{CHKL}, where the authors established failure of local uniform continuity of \eqref{fnls1} for $\al < 2$. In particular, this failure of local uniform continuity means that well-posedness of \eqref{fnls1} in $H^s (\R)$ for $s < \frac{2 - \al}{4}$ cannot be handled by a fixed point argument.

As mentioned above, the authors in \cite{CP} proved norm inflation of \eqref{fnls1} in $H^s (\R)$ for $s < \frac{1 - 2\al}{6}$ with $\al > 2$. Hence, for \eqref{fnls1} with $\al > 2$, there is a gap $\frac{1 - 2\al}{6} \leq s < \frac{2 - \al}{4}$ between local well-posedness and ill-posedness.
\end{remark}

Next, we would like to extend Theorem \ref{TM:LWP1} globally-in-time. Our global well-posedness result for the cubic fNLS \eqref{fnls1} reads as follows. 


\begin{theorem}[Global well-posedness on $\R$]
\label{TM:GWP1}
Let $\al>2$ and $\frac{2-\al}{4} \leq s<0$.
 Then, the
cubic fNLS \eqref{fnls1}
 is globally well-posed in $H^s(\R )$. More precisely, given any $u_0 \in H^s (\R)$ and any $T > 0$, there exists a unique solution $u \in C([-T, T]; H^s (\R))$ to \eqref{fnls1} with $u|_{t = 0} = u_0$.
\end{theorem}

In Theorem \ref{TM:GWP1}, the uniqueness holds in the sense that for any $t_0 \in \R$, there exists a time interval $I(t_0) \ni t_0$ such that the solution $u$ to \eqref{fnls1} is unique in the relevant $X^{s, b}$-space restricted to the time interval $I(t_0)$; see Subsection \ref{SUBSEC:spaces} below. Note that as in \cite{kihoon}, we are able to cover global well-posedness of \eqref{fnls1} for the full range of $s < 0$ where local well-posedness holds.

To prove Theorem \ref{TM:GWP1}, we need to obtain an a priori control of the $H^s$-norm of the solution and then iterate the local-in-time argument. Note that smooth solutions to the cubic fNLS~\eqref{fnls1} conserve the mass
\begin{align*}
M(u (t)) := \int_\R |u(t, x)|^2 dx,
\end{align*}

\noi
which provides an a priori control of the $L^2$-norm of the solution. However, if $u$ is a solution to~\eqref{fnls1} in $H^s(\R)$ for $s < 0$, we have in general $M(u) = \infty$. To deal with this issue,
we use
the $I$-method (also known as the method of almost conservation laws) as in \cite{kihoon}, which was
first introduced by
Colliander-Keel-Staffilani-Takaoka-Tao \cite{CKSTT0, CKSTT1}.

Let us briefly go over the main idea of the $I$-method, which relies on a smoothing operator
$I = I_N$ (known as the $I$-operator) mapping from $H^s(\R)$ to $L^2(\R)$.
See \eqref{Iop1} for the precise definition and Section \ref{SEC:Imd} for the discussion.
Thanks to the smoothing properties of the $I$-operator, the modified mass
\[
M(I u)=\int_{\R} | I u|^2 dx
\]

\noi
is finite for $u\in H^s(\R)$.
More importantly, the modified mass $M(I u)$ provides an a priori control of $\| u\|_{H^s}^2$.
Thus, the main task becomes controlling the growth of the modified mass $M(I u)$. See Subsection \ref{SEC:Imd} for further details.


\begin{remark} \rm
From the proof of global well-posedness of \eqref{fnls1}, we can obtain a polynomial growth bound of the Sobolev norm of the solution $u$ as in \cite{kihoon}. In fact, compared to \cite{kihoon}, our estimate is more optimal and we obtain a better growth bound of $\| u \|_{H^s}$ than that in \cite{kihoon}. See Remark \ref{RMK:sob} for more details.
\end{remark}

\subsection{fNLS on the circle}
\label{SUBSEC:introT}
In this subsection, we consider well-posedness  for the cubic fNLS~\eqref{fnls1} on the circle $\T$ in negative Sobolev spaces.

For $\al > 2$,
it is not difficult to show that the Cauchy problem of cubic fNLS \eqref{fnls1} is globally well-posed in $L^2(\T)$, 
which follows from the Strichartz estimate (see Lemma \ref{LEM:L4str} below) and a contraction argument together with the mass conservation. 
Since the scaling critical Sobolev index is negative,
one can expect that well-posedness for \eqref{fnls1} holds in some negative Sobolev spaces. 
However, 
by adapting the argument in \cite{Mo} to \eqref{fnls1}, we see that the solution map of \eqref{fnls1} fails to be continuous in $H^s (\T)$ for $s < 0$. In fact, in this range of $s$, we can show non-existence of solutions of \eqref{fnls1}; see Corollary \ref{COR:nonexist} below.
For this reason, we follow the idea in \cite{OS} and consider instead the following renormalized cubic fNLS:
\begin{align}
\begin{cases}
i\dt v = D^{\al} v + (|v|^2 - 2 \fint |v|^2) v \\
u|_{t= 0} = u_0,
\end{cases}
\quad 	(t, x)  \in \R \times \T,
\label{fnls2}
\end{align}

\noi
where $\fint f := \frac{1}{2\pi} \int_\T f(x)dx$. If $u_0 \in L^2 (\T)$, the equations \eqref{fnls1} and \eqref{fnls2} are equivalent under the following gauge transform
\begin{align}
\begin{split}
v(t) &= \mathcal{G} [u] (t) := e^{2 i t \fint |u(t)|^2} u(t), \\
u(t) &= \mathcal{G}^{-1} [v] (t) := e^{-2 i t \fint |v(t)|^2} v(t).
\end{split}
\label{gauge}
\end{align}

\noi 
Note that the additional term in \eqref{fnls2}
removes the resonant interactions caused by
 $n_2 = n_1$ or $n_2 = n_3$.\footnote{Such a modification does not seem to have a significant effect on $\R$,
since $\xi_2 = \xi_1$ or $\xi_2 = \xi_3$ is a set of measure zero in the hyperplane $\xi = \xi_1 - \xi_2 + \xi_3$ for a fixed $\xi$.}
More preseisly,
the nonlinearity on the right-hand side of \eqref{fnls2} can be written as
\begin{align*}  
\bigg( |v|^2 - 2 \fint  |v|^2 \bigg) v     =:  \NN_1 (v) - \RR_1 (v),
\end{align*}

\noi
where the non-resonant part $\NN_1 (v) = \NN_1 (v, v, v)$ and the resonant part $\RR_1 (v) = \RR_1 (v, v, v)$ are defined by
\begin{align*}
\NN_1(v_1, v_2, v_3)(t, x)&:=\sum_{n\in \Z} e^{inx} \sum_{\G(n)}\ft{v_1} (t, n_1) \cj{\ft{v_2} (t, n_2)} \ft{v_3} (t, n_3), \\
\RR_1(v_1,v_2,v_3)(t, x)
&:=\sum_{n \in \Z} e^{inx} \ft{v_1} (t, n) \cj{\ft{v_2} (t, n)} \ft{v_3} (t, n).
\end{align*}

\noi
with $\G(n)$ denoting the hyperplane:
\begin{equation}
\G(n) :=\{(n_1, n_2, n_3)\in  \Z^3 : n = n_1 - n_2 + n_3, n_1 \neq n_2, \text{and } n_2 \neq n_3 \}.
\label{Gam0}
\end{equation}

We now focus on the equation \eqref{fnls2} for $v$. By using the same arguments as in \cite{BGT, CCT2}, one can see that the solution map of \eqref{fnls1} fails to be locally uniformly continuous; see also \cite[Lemma 6.16]{OT}. This implies that the contraction argument does not work for proving local well-posedness of \eqref{fnls2}, and so one may construct a solution via a compactness argument. A similar situation happened in the modified KdV equation, where Takaoka-Tsutsumi \cite{TT} introduced modifications on the $X^{s, b}$-norm (see the $Y^{s, b}$-norm in Subsection \ref{SEC:Xsb}) to further weaken the effect of the resonant part $\RR_1 (v)$; see also \cite{NTT, MT}. In this paper, we avoid relying too much on the modified $X^{s, b}$-norm (this is essential in our global well-posedness part explained later) and follow Oh-Tzvetkov-Wang \cite{OTzW} by using a second gauge transform:
\begin{align}
w(t, x) = \J [v](t, x) := \sum_{n\in \Z} e^{inx-it|\ft{u_0} (n)|^2} \ft v (t, n),
\label{gauge2}
\end{align}

\noi
and  $\mathcal J^{-1}[v] $ is defined by replacing the minus sign with a plus sign.
One can easily see that $v$ satisfies \eqref{fnls2} if and only if $w$ satisfies the following equation:
\begin{equation}
\label{fnls3} 
\begin{cases}
i \dt w = D^\al w + \NN_2(w) - \RR_{2}(w) \\
w|_{t= 0} = u_0,
\end{cases} 
\quad (t, x) \in \R \times \T,
\end{equation}

\noi
where the non-resonant nonlinearity $\NN_2(w) = \NN_2 (w, w, w)$ is defined by
\begin{equation}\label{NN1O}
\NN_2(w_1, w_2, w_3) (t, x) : = \sum_{n\in \Z}  e^{inx} \sum_{\G(n)} 
e^{it\Psi(\cj n)} \ft{w_1} (t, n_1) \cj{\ft{w_2}} (t, n_2)  \ft{w_3} (t, n_3)
\end{equation}

\noi
with $\G(n)$ defined in \eqref{Gam0}
and $\Psi(\cj n)$ defined by
\begin{align}
\label{Psi}
\Psi (\cj n) &\hphantom{:}=  \Psi (n_1,n_2,n_3, n) \notag \\
&:= 
|\ft{u_0}(n_1)|^2 - |\ft{u_0}(n_2)|^2 + |\ft{u_0}(n_3)|^2 - |\ft{u_0}(n)|^2,
\end{align}

\noi
and the resonant nonlinearity $\RR_{2}(w) $ is defined by 
\begin{equation}\label{NN2O}
\RR_{2}(w) (t, x) : =  \sum_{n\in \Z}  e^{inx}  \big(|\ft w (t, n) |^2 - |\ft{u_0}(n)|^2\big) \ft w (t, n)\,.
\end{equation}

\noi
We point out that $\RR_{2}(w) $ has a smoothing effect which allows us to obtain a good estimate; see the discussion between Proposition \ref{Prop:triR} and Proposition \ref{Prop:quinc} below.

In the rest of the paper, we choose to study the gauge-transformed cubic fNLS \eqref{fnls3} for~$w$. Our local well-posedness result for \eqref{fnls3} reads as follows.
\begin{theorem}[Local well-posedness on $\T$]
\label{TM:LWP2}
Let $\al>2$ and $\frac{2-\al}{6} \leq s<0$.
Then, the
cubic fNLS \eqref{fnls3}
is locally well-posed in $H^s(\T )$. More precisely, given any $u_0 \in H^s (\T)$, there exists $T = T(\| u_0 \|_{H^s}) > 0$ and a unique solution $w \in C([-T, T]; H^s (\T))$ to \eqref{fnls3} with $u|_{t = 0} = u_0$.
\end{theorem}

Since the second gauge transform $\mathcal{J}$ is well defined in $H^s (\T)$ for any $s \in \R$, our local well-posedness result in Theorem \ref{TM:LWP2} can be carried over to the renormalized cubic fNLS \eqref{fnls2} for $v$. As in the non-periodic setting, our proof of Theorem \ref{TM:LWP1} is based on the Fourier restriction norm method, and so the uniqueness statement in Theorem \ref{TM:LWP2} holds only in (the local-in-time version of) the relevant $X^{s, b}$-space; see Subsection \ref{SUBSEC:spaces} below. Our local well-posedness result is able to cover the endpoint case in Miyaji-Tsutsumi \cite{MT}; see Remark \ref{RMK:3NLS} below.

As mentioned above, we prove Theorem \ref{TM:LWP2} by using a compactness argument. Namely, we construct the solution to \eqref{fnls3} by using smooth approximation of solutions and establishing an energy estimate. 
In particular, the proof of the existence part of local well-posedness involves the use of continuity of the local-in-time version of the $X^{s, b}$-space with respect to the time restriction. We point out that this continuity property has been used in many works but has not been proved anywhere. In this paper, we provide a proof of this continuity property; see Lemma \ref{LEM:Xsb_cty} below.

For the uniqueness part of local well-posedness, our main task is to justify the normal form reduction step (integration by parts in time) for rough solutions. We point out that in the framework of Takaoka-Tsutsumi \cite{TT} (also \cite{NTT, MT}), this justification for the uniqueness part does not seem to be written in an explicit manner in any existing literature. Also, in a recent work \cite{kwak} on 4NLS (i.e.~\eqref{fnls3} with $\al = 4$) using the same framework, the author claimed uniqueness, but the proof is not correct; see \cite[Proposition 3.8]{kwak}, where the author worked with the difference between two smooth solutions instead of rough solutions. In this paper, we give for the first time a rigorous justification for this normal form reduction step for rough solutions; see Remark \ref{RMK:fub_ibp} for more details.



\medskip
As a corollary to the proof of our local well-posedness for the renormalized cubic fNLS~\eqref{fnls2} in Theorem \ref{TM:LWP2}, we obtain the following nonexistence result for the original cubic fNLS~\eqref{fnls1} mentioned above.
\begin{corollary}
\label{COR:nonexist}
Let $\al > 2$, $\frac{2 - \al}{4} \leq s < 0$, and $u_0 \in H^s (\T) \setminus L^2 (\T)$. Then, for any $T > 0$, there is no distributional solution $u \in C([-T, T]; H^s (\T))$ to the cubic fNLS \eqref{fnls1} on $\T$ such that

\begin{enumerate}
\item[(i)] $u|_{t = 0} = u_0$;

\smallskip
\item[(ii)] There exist smooth global solutions $\{ u_n \}_{n \in \N}$ to \eqref{fnls1} such that $u_n \to u$ in $C([-T, T]; \mathcal{D}' (\T))$ as $n \to \infty$.
\end{enumerate}
\end{corollary}

To prove Corollary \ref{COR:nonexist}, once we have an a priori estimate for solutions to the renormalized cubic fNLS \eqref{fnls2}, we can follow the idea by Guo-Oh \cite{GO} to exploit the bad behavior of the gauge transform $\mathcal{G}$ below $L^2 (\T)$; see \cite[Section 9]{GO} for more details. In the proof of local well-posedness in Theorem \ref{TM:LWP2}, we show an a priori estimate for the solution $w$ to the equation~\eqref{fnls3}. Since the second gauge transform $\mathcal{J}$ preserves the $H^s$-norm, we can obtain an a priori estimate for the solution $v$ to the renormalized cubic fNLS \eqref{fnls2}. This allows us to achieve the nonexistence result in Corollary \ref{COR:nonexist}.

\begin{remark} \rm
In \cite{CP}, Choffrut-Pocovnicu showed a norm inflation result for the cubic fNLS~\eqref{fnls1} in $H^s (\T)$ for $s < \frac{1 - 2\al}{6}$. As mentioned in \cite{CP}, their norm inflation result can be carried over to the renormalized cubic fNLS \eqref{fnls2}, since the gauge transform $\mathcal{G}$ preserves the $H^s$-norm. 
As a result, there is a gap $\frac{1 - 2\al}{6} \leq s < \frac{2 - \al}{6}$ between local well-posedness and ill-posedness result for the renormalized cubic fNLS \eqref{fnls2}. 
\end{remark}

Lastly, we would like to extend Theorem \ref{TM:LWP2} globally-in-time. Let us state our global well-posedness result for the cubic fNLS \eqref{fnls3} in the following.

\begin{theorem}[Global well-posedness on $\T$]
\label{TM:GWP2}
Let $\al>2$ and $\frac{2-\al}{6} \leq s<0$.
 Then, the
cubic fNLS \eqref{fnls3}
 is globally well-posed in $H^s(\T )$. More precisely, given any $u_0 \in H^s (\T)$ and any $T > 0$, there exists a unique solution $u \in C([-T, T]; H^s (\T))$ to \eqref{fnls3} with $u|_{t = 0} = u_0$.
\end{theorem}

As before, the uniqueness statement in Theorem \ref{TM:GWP2} holds in the sense that for any $t_0 \in \R$, there exists a time interval $I(t_0) \ni t_0$ such that the solution $u$ to \eqref{fnls3} is unique in the relevant $X^{s, b}$-space restricted to the time interval $I(t_0)$; see Subsection \ref{SUBSEC:spaces} below. Also, we are able to cover global well-posedness of \eqref{fnls3} for the full range of $s < 0$ where local well-posedness holds. Our global well-posedness result covers the missing endpoint case in Oh-Wang \cite{OW} in some suitable sense; see Remark \ref{RMK:uni} below.

To prove Theorem \ref{TM:GWP2}, we follow an analogous argument as in the non-periodic setting.
While there is no dilation symmetry
in the periodic setting, one can still use the $I$-method as mentioned in Subsection \ref{SUBSEC:fNLS_R}.
Indeed, the authors in \cite{SPST} successfully applied the $I$-method to the period setting by doing analysis on a dilated torus $\T^d_{\ld}=\R^d/\ld\Z^d$. 
In this paper, instead of applying $I$-method on the dilated circle,
we use the following $H_M^s$-norm adapted to the parameter $M \geq 1$ defined by
\[ \| f \|_{H^s_M} = \big\|(M^2 + n^2)^\frac{s}{2} \ft f(n)\big\|_{\l^2_n}.\]

\noi
While the $H_M^s$-norm is equivalent to the standard $H^s$-norm, 
we have the following decay property when $s < 0$:
\[
\lim_{M\to \infty} \|f\|_{H^s_M}=0
\]

\noi
for any $f \in H^s(\T)$. 
This allows us to reduce the problem to a small data setting in some appropriate sense. The idea of using the $H^s_M$-norm comes from \cite{CHT, OW}; see also \cite{CFLOP, KLV}.
Due to the change of the modulation function in the $H^s_M$-norm, we adjust the definition of $I$-operator to this $H^s_M$-norm; see \eqref{iop2}.

By using the $H^s_M$-norm and the modified $I$-operator, we establish an a priori control on the growth of $\|u\|_{H^s(\T)}$, which then allows us to iterate the local well-posedness argument to achieve global well-posedness; see Subsection \ref{SEC:AClaw2} for details. In the proof, we need to invoke the $L^6$-Strichartz estimate established by the $\l^2$ decoupling theorem by Bourgain-Demeter \cite{BD}; see Lemma \ref{LEM:L6str} below. Thus, it is important that we work with the original $X^{s, b}$-norm instead of the modified $X^{s, b}$-norm as in \cite{TT, NTT, MT}. We also point out that in this paper, we use the $I$-method combined with a non-contraction argument, which is a different situation than the usual application of the $I$-method in existing literature.

\begin{remark} \rm
\label{RMK:3NLS}
In \cite{MT1, MT}, Miyaji-Tsutsumi studied the following NLS with a third order dispersion:
\begin{align}
\dt u - \partial_x^3 u + i \g \partial_x^2 u + i |u|^2 u = 0,
\label{3NLS}
\end{align}

\noi
where $\g \in \R$ is a constant with $2\g / 3 \notin \Z$. Specifically, the authors showed that \eqref{3NLS} is locally well-posed in $H^s (\T)$ for $s > -\frac 16$ and globally well-posed in $L^2 (\T)$. We remark that the same proof in this paper for our equation \eqref{fnls2} with $\al = 3$ applies to the equation \eqref{3NLS}, and we are able to cover global well-posedness of \eqref{3NLS} in $H^s (\T)$ for the range $s \geq - \frac 16$ including the endpoint $-\frac 16$.
\end{remark}

\begin{remark} \rm
\label{RMK:uni}
In \cite{OW}, in the case of the renormalized 4NLS (i.e.~\eqref{fnls2} with $\al = 4$), Oh-Wang showed an enhanced uniqueness result for the global solution in the following sense: for $-\frac 13 < s < 0$, the solution to 4NLS is unique among all solutions in $C(\R; H^s (\T))$ with the same initial data equipped with smooth approximating solutions. This was proved by using an infinite iteration of normal form reductions, which was justified in a different manner than that in this paper. 

In this paper, although our uniqueness holds only in (a local-in-time version of) the relevant $X^{s, b}$-space, we do not need the assumption that the solution admits a sequence of smooth approximating solutions. Moreover, our uniqueness covers the endpoint case $s= - \frac 13$.
We also remark that one may use an infinite iteration of normal form reductions to prove an enhanced uniqueness result as in \cite{OW} for the cubic fNLS \eqref{fnls2} for some range of $s < 0$ depending on the value of $\al$.
\end{remark}

We conclude the introduction by stating several remarks on some other related aspects of the cubic fNLS \eqref{fnls1}.
\begin{remark} \rm
It is possible to improve our global well-posedness result for the cubic fNLS~\eqref{fnls1} by using the short time Fourier restriction norm method. Indeed, in \cite{OW}, Oh-Wang used the short time Fourier restriction norm method to show a global existence result for the 4NLS (i.e.~\eqref{fnls2} with $\al = 4$) in $H^s (\T)$ for $s > -\frac{9}{20}$, which covers a larger range of~$s$ than that in our Theorem \ref{TM:GWP2}. However, uniqueness might be an issue in this approach.
\end{remark}

\begin{remark} \rm
It would be of interest to study the cubic fNLS \eqref{fnls1} in Fourier-Lebesgue spaces $\mathcal{FL}^{s, p} (\M)$ given by the norm
\begin{align*}
\| f \|_{\mathcal{FL}^{s, p} (\M)} := \big\| \jb{\xi}^s \ft{f} (\xi) \big\|_{L_\xi^p (\ft{\M})},
\end{align*} 

\noi
where $s \in \R$, $1 \leq p \leq \infty$, and $\M = \R$ or $\T$. Using a similar scaling analysis as in Subsection~\ref{SUB:intro}, we see that (a homogeneous version of) the Fourier-Lebesgue space $\mathcal{FL}^{s, p} (\R)$ is invariant under the scaling symmetry \eqref{scaling} when $s = s_{\text{crit}} (p) = 1 - \frac 1p - \frac{\al}{2}$.

The Fourier-Lebesgue spaces have been used mainly in the study of NLS (i.e.~$\al = 2$); see \cite{Grun, Chr, GH, OW21, FO}. It may be possible to use Fourier-Lebesgue spaces to cover local well-posedness for the cubic fNLS \eqref{fnls1} on $\R$ and $\T$ in the full subcritical regime. In this paper, however, we do not pursue this issue.
\end{remark}

\begin{remark} \rm
\label{RMK:rand}
One can also study almost sure well-posedness of the cubic fNLS \eqref{fnls1} using random initial data. Since Bourgain's seminal work \cite{BO96} on almost sure local well-posedness of the cubic NLS on $\T^2$, there has been a significant development in this direction for NLS on $\T^d$ with random initial data of the form
\begin{align}
u_0^\omega = \sum_{n \in \Z^d} e^{in \cdot x} \frac{g_n (\omega)}{\jb{n}^\gamma},
\label{rand_init}
\end{align}

\noi
where $\g \in \R$ and $\{g_n\}_{n \in \Z^d}$ are i.i.d standard complex Gaussian random variables; see \cite{BO97, CO, DNY2, FOSW, DNY3, Liu}. By considering random initial data \eqref{rand_init} for the cubic fNLS \eqref{fnls1} on the circle $\T$, one may be able to improve our global well-posedness result in Theorem \ref{TM:GWP2} and even go beyond the deterministic well-posedness threshold for \eqref{fnls1}.

An interesting case for the cubic fNLS \eqref{fnls1} is the random initial data \eqref{rand_init} with $\g = \frac{\al}{2}$. This initial data induce a Gaussian measure that can be used to construct the invariant Gibbs measure associated to \eqref{fnls1} and prove almost sure global well-posedness of \eqref{fnls1}; see \cite{SunTz, LW1, LW2, LLW}.

Another interesting case is the white noise initial data (i.e.~\eqref{rand_init} with $\g = 0$), which lies almost surely in $H^{-\frac 12 - \eps} (\T) \setminus H^{-\frac 12} (\T)$ for any $\eps > 0$; see \cite[Lemma B.1]{BTz08}. This white noise initial data induce the white noise measure, which is formally written as
\begin{align}
\text{``}Z^{-1} e^{-\frac 12 \| u \|_{L^2 (\T)}^2} du\text{''}.
\label{white}
\end{align} 

\noi
In view of the mass conservation of the cubic fNLS \eqref{fnls1}, the white noise measure \eqref{white} is expected to be invariant under the dynamics of \eqref{fnls1}. Indeed, in \cite{OTzW}, Oh-Tzvetkov-Wang showed almost sure global well-posedness and invariance of the white noise measure in the case of 4NLS (i.e.~\eqref{fnls1} with $\al = 4$). In a general setting, as a corollary of our global well-posedness result in Theorem \ref{TM:GWP2}, invariance of the white noise under the dynamics of the cubic fNLS~\eqref{fnls1} (more precisely, the renormalized cubic fNLS \eqref{fnls2}) follows for $\al > 5$. When $\al \leq 5$, our local well-posedness argument does not cover well-posedness in the support of the white noise and thus it would be of interest to study probabilistic well-posedness for $2 < \al \leq 5$ and $\al \neq 4$. In the case of NLS (i.e.~\eqref{fnls1} with $\al = 2$), this problem is critical; see \cite{FOW, DNY2}.

Although random initial data of the form \eqref{rand_init} are restricted to periodic domains, one can also consider probabilistic well-posedness of the cubic fNLS \eqref{fnls1} on Euclidean domains via the Wiener randomization of the initial data; see, for example, \cite{BOP1, BOP2}.
\end{remark}

\subsection{Organization of the paper} 
This paper is organized as follows. In Section \ref{SEC:2}, we go over some notations and useful function spaces.
In Section \ref{SEC:WPreal}, after establishing some Strichartz estimates, we prove local well-posedness (Theorem \ref{TM:LWP1}) and then global well-posedness (Theorem~\ref{TM:GWP1}) of the cubic fNLS \eqref{fnls1} on $\R$.
In Section \ref{SEC:WPT}, after showing some Strichartz estimates and multilinear estimates, we prove local well-posedness (Theorem \ref{TM:LWP2}) and then global well-posedness (Theorem \ref{TM:GWP2}) of the cubic fNLS \eqref{fnls3} on $\T$.

\section{Notations, function spaces, and preliminary estimates}
\label{SEC:2}

In this section, we discuss notations, function spaces, and some relevant preliminary estimates.

\subsection{Notations} 
\label{SUBSEC:notations}
For $a, b > 0$, we use $a\les b$ to denote that
there exists $C>0$ such that $a \leq Cb$.
By $a\sim b$, we mean that $a\les b$ and $b \les a$. 
We also use $a+$ (and $a-$) to denote $a + \eps$ (and $a - \eps$, respectively) for arbitrarily small $0 < \eps \ll 1$.

We use $\eta \in C_c^\infty(\R)$ to denote a smooth cutoff function supported on $[-2,2]$ with $\eta \equiv 1$ on $[-1, 1]$ and let $\eta_{T} (t) = \eta(t / T)$ for any $T > 0$.

Throughout this paper, we denote $\mathcal{M}$ by either $\R$ or $\T$. Given a function $u$ on $\M\times \R$, we
use $\ft{u}$ and $\F (u)$
to denote the space-time Fourier transform of $u$ given by
\[ \ft{u}(\tau, \xi) = \int_{\R \times \M} e^{- i t \tau - i x \xi} u(t, x) dt dx.\]

\noi
When there is no confusion,
we may simply use $\ft{u}$ or $\F(u)$
to denote
the spatial, temporal, or space-time Fourier transform
of $u$, depending on the context. We also denote $\F^{-1}$ as the inverse Fourier transform.
Note that the derivative $\dx$ corresponds to the multiplication by $i\xi$ on the Fourier side.

For a dyadic number $N\geq 1$, we define the dyadic interval $I_N$ by
\begin{align*}
I_1 =[-1,1] 
\qquad
\text{and}
\qquad
 I_N =\left[-N,- \tfrac N2\right) \cup \left(\tfrac N2,N\right],
\end{align*}

\noi
$N \geq 2$.
 Then,  
we use $\pi_N$ to
denote the projection operator defined by
$\ft{\pi_N f}(\xi) = \ind_{I_N}(\xi) \ft{f}(\xi)$.
We also set 
\[ \pi_{\leq N } = \sum_{\substack{1 \leq L \leq N \\ L \text{ dyadic}}} \pi_L
\qquad \text{and}\qquad \pi_{> N } = \sum_{\substack{L > N \\ L \text{ dyadic}}} \pi_L.\]

In the following, we use $S(t)= e^{-i t D^\al}$
to denote the solution operator to
the linear fractional Schr\"odinger equation: $i \dt u = D^\al u $.
Namely, we have
\begin{align*}
S(t) f = \int_{\ft{ \M}} e^{- i t |\xi|^\al + i x \xi} \ft{f}(\xi) d\xi.
\end{align*}

We also recall some notations from \cite{ckstt2}, which will be useful for the proof of our global well-posedness results.
If $d\geq 2$ is an even integer, we denote by $M_d = M_d(\xi_1, \dots, \xi_d)$ a general (spatial) multiplier of
order $d$ on the
hyperplane
\[ \G_d := \{ (\xi_1, \dots, \xi_d) \in \M: \xi_1 + \cdots +
\xi_d = 0 \},\]

\noi
which we endow with the Dirac delta measure $\delta(\xi_1 + \cdots +
\xi_d)$. If $M_d$ is a multiplier of order $d$ and $f_1, \ldots, f_d$ are
functions on $\R$ or $\T$, we define $\Ld_d(M_d;f_1, \dots,
f_d)$ by
\[ 
\Ld_d(M_d;f_1, \dots, f_d)
:= \int_{\G_d} M_d(\xi_1, \dots, \xi_d) \prod_{j=1}^d \ft{f_j} (\xi_j) d\xi_1 \cdots d\xi_d.
\]

\noi
We also use the notation
\[ \Lambda_d(M_d;f) := \Lambda_d(M_d; f, \cj{f}, \dots,
f, \cj{f}).\]

\noi
Observe that $\Lambda_d(M_d;f)$ is invariant under permutations of
the odd (even) indices.

If $M_d$ is a multiplier of order $d$, $j$ is an index with $1 \leq j \leq d$,
and $k \geq 1$ is an even integer, we define the elongation
$\mathbf X^k_j(M_d)$ of $M_d$ to be the multiplier of order $d+k$ given by
$$ \mathbf X^k_j(M_d)(\xi_1, \dots, \xi_{d+k})
:= M_d(\xi_1, \dots, \xi_{j-1}, \xi_j + \cdots + \xi_{j+k},
\xi_{j+k+1}, \dots, \xi_{d+k}).$$

\noi
In other words, $\mathbf X^k_j$ is the multiplier obtained by replacing
$\xi_j$ by $\xi_j + \cdots + \xi_{j+k}$ and increasing all the indices
after $\xi_j$ by $k$. We shall often write $\xi_{ij}$ for $\xi_i + \xi_j$, $\xi_{ijk}$ for
$\xi_i + \xi_j + \xi_k$, etc.

\subsection{Preliminary lemmas}
In this subsection, we recall some useful lemmas. We first record the following lower bound on the resonant function. For a proof, see \cite[Lemma 2.3]{FT}.
\begin{lemma}
\label{LEM:factor}
Let $\al > 1$ and $\xi_1, \dots, \xi_4 \in \R$ be such that $\xi_1 + \xi_2 + \xi_3 + \xi_4 = 0$. Then, we have
\begin{align*}
\big| |\xi_1|^\al - |\xi_2|^\al + |\xi_3|^\al - |\xi_4|^\al \big| \ges |\xi_1 + \xi_2| |\xi_2 + \xi_3| |\xi_{\text{max}}|^{\al - 2},
\end{align*}

\noi
where $\xi_{\text{max}} = \max \{ |\xi_1|, |\xi_2|, |\xi_3|, |\xi_4| \}$.
\end{lemma}

We now recall the following counting lemma, which will be useful for well-posedness of the cubic fNLS \eqref{fnls1} on the circle. For a proof, see \cite[Lemma 2]{STz}.
\begin{lemma}
\label{LEM:count}
Let $I, J \subset \R$ be two intervals and $g$ be a $C^1$ function on $\R$. Then, we have
\begin{align*}
\# \{ k \in \J \cap \Z: g(k) \in I \} \leq \frac{|I|}{\inf_{x \in J} |g'(x)|} + 1.
\end{align*}
\end{lemma}

\subsection{Sobolev spaces}
\label{SUBSEC:spaces}

We denote $\jb{\nb}$ as the Fourier multiplier with symbol $\jb{\xi} = (1+|\xi|^2)^\frac{1}{2}$. We recall the definition of the standard Sobolev space $H^s(\M)$:
\[ \| f \|_{H^s} :=\|\jb{\nb}^s f\|_{L^2(\M)} = \big\| \jb{\xi}^s \ft f(\xi) \big\|_{L^2(\ft{\M})}.\]

\noi
Given $M \geq1$, we define the  $H^s_M$-norm adapted to the parameter $M\geq 1$
by 
\begin{align}
\| f \|_{H^s_M} := \big\|(M^2 + \xi^2)^\frac{s}{2} \ft f(\xi)\big\|_{L^2(\M)}.
\label{SobM}
\end{align}

\noi
Clearly, 
the $H^s_M$-norm is equivalent to the standard $H^s$-norm.
When $s< 0$, however, it follows from the dominated convergence theorem that
\begin{align*}
\lim_{M\rightarrow \infty} \| f \|_{ H_M^s} = 0
\end{align*}

\noi
for any $f \in H^s(\M)$. 

We record the follow lemma. For a proof, see \cite[Lemma 3.4]{GKO}.
\begin{lemma}
\label{LEM:prod}
Let $0 \leq s \leq 1$ and $1 < p, q, r < \infty$ satisfying $\frac{1}{p} + \frac{1}{q} = \frac{1}{r} + s$. Then, we have
\begin{align*}
\| \jb{\nb}^{-s} (fg) \|_{L^r (\T)} \les \| \jb{\nb}^{-s} f \|_{L^p (\T)} \| \jb{\nb}^s g \|_{L^q (\T)}.
\end{align*}
\end{lemma}

\subsection{Fourier restriction norm method}
\label{SEC:Xsb}
In this subsection, we recall the Fourier restriction norm method introduced by Bourgain \cite{BO1}. Specifically, we recall the $X^{s, b}$-space defined by the norm
\begin{align}
\label{Xsb1}
\begin{split}
\| u\|_{X^{s, b}} &:= 
\|\jb{\xi}^s \jb{\tau + |\xi|^\al}^b \ft {u}(\tau, \xi)\|_{L_{\tau}^2 L^2_\xi (\R \times \ft\M)} \\
&\hphantom{:}= \| \jb{\dx}^s \jb{\dt}^b S(-t) u \|_{L_{t}^2 L_x^2 (\R \times \M)}.
\end{split}
\end{align}


\noi
Also, in view of \eqref{Xsb1} and \eqref{gauge2}, we introduced the following $Y^{s,b}$-norm on the circle (i.e.~$\M = \T$): 
\begin{align*}
\| u\|_{Y^{s, b}} := 
\|\jb{n}^s \jb{\tau + \mu(n)}^b \ft {u}(\tau, n)\|_{L^2_\tau \l_n^2 (\R \times \Z)},
\end{align*}

\noi
where
$\mu(n)=|n|^\al - |\ft u_0(n)|^2$.
Moreover,
for any time interval $I\subset \R$, we define the following restricted spaces:
\begin{align*}
\| u \|_{X_I^{s, b}} &:= \inf \{ \| v \|_{X^{s, b}}: v|_{I} = u \}, \\
\| u \|_{Y_I^{s, b}} &:= \inf \{ \| v \|_{Y^{s, b}}: v|_{I} = u \}.
\end{align*}

\noi
If $I = [-T, T]$ for some $T > 0$, we simply write $X_T^{s, b} = X_{[-T, T]}^{s, b}$ and $Y_T^{s, b} = Y_{[-T, T]}^{s, b}$. Given any $s \in \R$ and $b > \frac 12$, we have $X_T^{s, b} \embeds C([-T, T]; H^s (\R))$ and $Y_T^{s, b} \embeds C([-T, T]; H^s (\R))$.

The following two lemmas provide some basic information for us to do analysis by using the Fourier restriction norm method. See, for example, \cite{Tao} for proofs of these lemmas.
The first lemma states the embedding properties of the $X^{s,b}$-spaces.
\begin{lemma}
\label{LEM:emb}
Let $s \in \R$, $p > 2$, $b \geq \frac 12 - \frac 1p$, and $I \subset \R$ be an interval. 
Then, we have 
\begin{align*}
\Vert u \Vert_{L_I^p H_x^s} \lesssim \Vert u \Vert_{X_I^{s,b}}.
\end{align*}

\noi
Moreover, if $b > \frac 12$, we have
\begin{align*}
\Vert u \Vert_{C_I H_x^s} \lesssim \Vert u \Vert_{X_I^{s,b}}.
\end{align*}

\noi
The same embeddings hold if the $X_T^{s, b}$-norm is replaced by the $Y_T^{s, b}$-norm.
\end{lemma}

We also recall the following linear estimates of the $X^{s,b}$-norm.
\begin{lemma}
\label{LEM:lin}
Let $s \in \R$, $b > \frac 12$, $-\frac 12 < b_1 < b_2 < \frac 12$, and $T > 0$.

\smallskip \noi
\textup{(i)} We have the following homogeneous linear estimate:
\begin{align*}
\| S(t) f \|_{X_T^{s, b}} \les \| f \|_{H^s}.
\end{align*}

\smallskip \noi
\textup{(ii)} We have the following inhomogeneous linear estimate:
\begin{align*}
\bigg\| \int_0^t S(t - t') F(t') dt' \bigg\|_{X_T^{s, b}} \les \| F \|_{X_T^{s, b - 1}}.
\end{align*}

\smallskip \noi
\textup{(iii)} We have the following time localization estimate:
\begin{align*}
\| u \|_{X_T^{s, b_1}} \les T^{b_2 - b_1} \| u \|_{X_T^{s, b_2}}.
\end{align*}
\end{lemma}

We also need the following time localization estimate with a sharp time cutoff. For a proof, see \cite[Lemma 4.4]{Bring2}.
\begin{lemma}
\label{LEM:time}
Let $s \in \R$, $-\frac 12 < b' < \frac 12$. Then, for any $t_1 < 0 < t_2$, we have
\begin{align*}
\| u \ind_{[t_1, 0]} \|_{X^{s, b'}} &\les \| u \|_{X^{s, b'}}, \\
\| u \ind_{[0, t_2]} \|_{X^{s, b'}} &\les \| u \|_{X^{s, b'}}.
\end{align*}
\end{lemma}

Finally, we establish the following lemma on continuity of the $X^{s, b}_T$-norm with respect to $T$.

\begin{lemma}
\label{LEM:Xsb_cty}
Let $s \in \R$, $b \in \R$, and $w \in X^{s, b}$. Then, the function
\begin{align*}
T \mapsto \| w \|_{X_T^{s, b}}
\end{align*}

\noi
is continuous on $\R_+$.
\end{lemma}

\begin{proof}
The proof is similar to \cite[Lemma 8.1]{GO}, but we choose to present full details here.

Let us fix $T_0 > 0$ and show continuity of $\| w \|_{X_T^{s, b}}$ at $T_0$. Note that for all $T > 0$, by the dominated convergence theorem, we have
\begin{align*}
\| \pi_{> K} w \|_{X_T^{s, b}} \leq \| \pi_{> K} w \|_{X^{s, b}} \longrightarrow 0
\end{align*}

\noi
as $K \to \infty$. Thus, we only need to show continuity of $\| \pi_{\leq K} w \|_{X_T^{s, b}}$ at $T_0$ for some fixed large $K \in \N$. Below, we write $w$ for $\pi_{\leq K} w$ for simplicity.

For $\ld > 0$, we define
\begin{align*}
D(\ld) (w) (t, x) := w (\ld^{-1} t, x).
\end{align*}

\noi
Let $r = \frac{T}{T_0}$ be close to 1. Note that if $|n| \leq K$, we have
\begin{align}
\jb{r^{-1} \tau + |n|^\al}^b \les \jb{r^{-1} (\tau + |n|^\al)}^b \jb{(r^{-1} - 1) |n|^\al}^{|b|} \les K^{\al |b|} \jb{\tau + |n|^\al}^b.
\label{domi_bdd}
\end{align}

\noi
Thus, by the dominated convergence theorem, we have
\begin{align}
\big| \| D(r) (w) \|_{X_T^{s, b}} - \| w \|_{X_T^{s, b}} \big| \les \| D(r) (w) - w \|_{X_T^{s, b}} \longrightarrow 0
\label{liminf0}
\end{align}

\noi
as $r \to 1$. Hence, it suffices to show 
\begin{align*}
\lim_{r \to 1} \| D(r) (w) \|_{X_{r T_0}^{s, b}} = \| w \|_{X_{T_0}^{s, b}},
\end{align*}

\noi
which will follow from
\begin{align}
\liminf_{r \to 1} \| D(r) (w) \|_{X_{r T_0}^{s, b}} \geq \| w \|_{X_{T_0}^{s, b}},  \label{liminf} \\
\limsup_{r \to 1} \| D(r) (w) \|_{X_{r T_0}^{s, b}} \leq \| w \|_{X_{T_0}^{s, b}}.  \label{limsup}
\end{align}

We only prove \eqref{liminf}, and the proof for \eqref{limsup} will follow from a similar manner. For any $\eps > 0$, we let $w_r$ be an extension of $D(r) (w)$ on $[-rT_0, rT_0]$ such that
\begin{align}
\| w_r \|_{X^{s, b}}^2 \leq \| D (r) (w) \|_{X_{r T_0}^{s, b}}^2 + \eps.
\label{liminf1}
\end{align}

\noi
Note that $D(r^{-1}) (w_r)$ coincides with $w$ on $[-T_0, T_0]$, so that
\begin{align}
\| w \|_{X_{T_0}^{s, b}}^2 \leq \| D (r^{-1}) (w_r) \|_{X^{s, b}}^2.
\label{liminf2}
\end{align}

\noi
Also, we have
\begin{align*}
\| D (r^{-1}) (w_r) \|_{X^{s, b}}^2 &= \frac{1}{r} \int_\R \sum_{n \in \Z} \jb{r \tau + |n|^\al}^{2b} \jb{n}^{2s} |w_r (\tau, n)|^2 d\tau \\
&= \frac{1}{r} \| w_r \|_{X^{s, b}}^2 + \frac{1}{r} \int_\R \sum_{n \in \Z} \big( \jb{r \tau + |n|^\al}^{2b} - \jb{\tau + |n|^\al}^{2b} \big) \jb{n}^{2s} |w_r (\tau, n)|^2 d\tau.
\end{align*}

\noi
Thus, by \eqref{domi_bdd}, \eqref{liminf1}, and \eqref{liminf0}, we can use the dominated convergence theorem to obtain
\begin{align}
\| D (r^{-1}) (w_r) \|_{X^{s, b}}^2 \leq \frac{1}{r} \| D(r) (w) \|_{X_{rT_0}^{s, b}}^2 + \frac{2\eps}{r}.
\label{liminf3}
\end{align}

\noi
Combining \eqref{liminf2} and \eqref{liminf3}, we obtain the desired estimate \eqref{liminf}.

\end{proof}

\section{Well-posedness on the real line}
\label{SEC:WPreal}

In this section, we study low regularity well-posedness of the cubic fNLS \eqref{fnls1} on the real line (i.e.~$ \M = \R $).

\subsection{Strichartz estimates on the real line}
In this subsection, we show some Strichartz estimates on $\R$ that will be useful for proving our well-posedness result on $\R$.
We recall that $S(t) = e^{- i t D^\al}$.

We first consider the following linear Strichartz estimate.
\begin{lemma}
\label{LEM:str}
Let $\al > 2$. Let $0 \leq \g \leq \frac{\al}{2} - 1$, $r \geq 2$, and $q \geq \frac{2 \al}{1 + \g}$ be such that
\begin{align*}
\frac{\al}{q} + \frac{1 + \g}{r} = \frac{1 + \g}{2}.
\end{align*}

\noi
Then, we have
\begin{align*}
\big\| D^{\frac{\g}{2} (1 - \frac{2}{r})} S(t) f \big\|_{L_t^q L_x^r (\R \times \R)} \les \| f \|_{L^2 (\R)}.
\end{align*}
\end{lemma}
\begin{proof}
By following the proof of \cite[Lemma 2.1]{kihoon}, we obtain the following dispersive estimate:
\begin{align}
\| D^\g S(t) f \|_{L^\infty (\R)} \les |t|^{- \frac{1 + \g}{\al}} \| f \|_{L^1 (\R)}.
\label{disp}
\end{align}

\noi
Then, by using the $TT^*$ argument along with \eqref{disp} as in the proof of \cite[Lemma 2.5]{kihoon}, we obtain the desired Strichartz estimate.
\end{proof}

We will also need the following bilinear Strichartz estimate in a high-low interaction analysis.
\begin{lemma}
\label{LEM:blin}
Let $\al>2$ and $N_1,N_2 \geq 1$ to be dyadic numbers satisfying $N_1\leq \frac{N_2}{8}$. Then, we have
\begin{align*}
\|   S(t)  P_{N_1} \phi S(t) P_{N_2} \phi \|_{L_{t,x}^2 ( \mathbb{R}\times \mathbb{R} )} \les N_2^{\frac{1-\al}{2}} \| P_{N_1} \phi \|_{L^2 ( \mathbb{R} )} \| P_{N_2} \phi \|_{L^2 ( \mathbb{R} )}.
\end{align*}
\end{lemma}

\begin{proof}
The proof of the bilinear Strichartz estimate follows in the same way as \cite[Lemma~2.7]{kihoon}.
\end{proof}

\begin{remark} \rm 
\label{RMK:str}
By Lemma \ref{LEM:str}, we have the following estimates:
\begin{align*}
\big\| D^{\frac{\al-2}{4}} S(t) f \big\|_{L^4_t L^\infty_x(\R\times \R)}
\les
\|f\|_{L^2(\R)}, \\
\| S(t) f \|_{L^{2\al}_t L^\infty_x(\R\times \R)}
\les
\|f\|_{L^2(\R)}, \\
\| S(t) f \|_{L^\infty_t L^2_x (\R \times \R)} \les \| f \|_{L^2 (\R)}.
\end{align*}

\noi
By using the above three estimates and the transference principle as in \cite[Lemma 2.9]{Tao}, we have that for $b > \frac 12$, 
\begin{align}
\| D^{\frac{\al-2}{4}} u \|_{L^4_t L^\infty_x(\R\times \R)}
\les
\| u \|_{X^{0,b}},
\label{str1}\\
\| u \|_{L^{2\al}_t L^\infty_x(\R\times \R)}
\les
\| u \|_{X^{0,b}},
\notag \\
\| u \|_{L^\infty_t L^2_x (\R \times \R)} \les \| u \|_{X^{0, b}}. \notag
\end{align}

Also, by Lemma \ref{LEM:blin} and the transference principle as in \cite[Lemma 2.9]{Tao}, we have
\begin{align}
\| P_{N_1} u_1 P_{N_2} u_2 \|_{L_{t, x}^2 (\R \times \R)} \les N_2^{\frac{1 - \al}{2}} \| P_{N_1} u_1 \|_{X^{0, b}} \| P_{N_2} u_2 \|_{X^{0, b}}.
\label{bistr}
\end{align}
\end{remark}

\subsection{Local well-posedness on the real line}
\label{SEC:LWP1}

In this subsection, we prove Theorem \ref{TM:LWP1}, local well-posedness of the cubic fNLS \eqref{fnls1} on $\R$. 
In order to show local well-posedness, we need the following crucial trilinear estimate:

\begin{proposition}
\label{PROP:triR}
 Let $\al>2$, $ \frac{2-\al}{4} \leq s<0$, and $T > 0$. Then, for $ 0 < \e \ll 1 $ small enough, we have
\begin{align*}
\| u_1 \overline{u_2} u_3 \|_{ X_T^{s, - \frac 12 + 2\eps }} \les \prod_{j=1}^3 \| u_j \|_{ X_T^{s,\frac 12 + \eps }}.
\end{align*}
\end{proposition}

Let us first assume the trilinear estimate in Proposition \ref{PROP:triR} and prove Theorem \ref{TM:LWP1}. For this purpose, we write \eqref{fnls1} in the following Duhamel formulation:
\begin{align}
u(t) = S(t)u_0 \mp i\int_0^t S(t-t') (|u|^2 u)(t') dt',
\label{Duh}
\end{align}

\noi
where $S(t)=e^{-it D^\al}$.

\begin{proof}[Proof of Theorem \ref{TM:LWP1}]
We define the map
\begin{align*}
\G [u] (t) := S(t) u_0 \mp i \int_0^t S (t - t') (|u|^2 u) (t') dt'.
\end{align*}

\noi
Let $T > 0$ and $b = \frac 12 + \eps$ for $\eps > 0$. By Lemma \ref{LEM:lin} and Proposition \ref{PROP:triR}, we have
\begin{align*}
\| \G [u] \|_{X_T^{s, b}} \les \| u_0 \|_{H^s} + T^\eps \| u \|_{X_T^{s, b}}^3
\end{align*}

\noi
for $\eps > 0$ sufficiently small. Using similar steps, we also have the following difference estimate:
\begin{align*}
\| \G [u_1] - \G [u_2] \|_{X_T^{s, b}} \les T^\eps \Big( \| u_1 \|_{X_T^{s, b}}^2 + \| u_2 \|_{X_T^{s, b}}^2 \Big) \| u_1 - u_2 \|_{X_T^{s, b}}. 
\end{align*}

\noi
Thus, by choosing $T > 0$ sufficiently small, we obtain that $\G$ is a contraction on a ball in the $X_T^{s, b}$-space with radius $\sim \| u_0 \|_{H^s}$. This implies local well-posedness of \eqref{Duh} and thus finishes the proof of Theorem \ref{TM:LWP1}.
\end{proof}

We now show the proof of the trilinear estimate in Proposition \ref{PROP:triR}.

\begin{proof}[Proof of Proposition \ref{PROP:triR}]
The proof follows similarly as \cite[Proposition 3.5]{kihoon}, and so we will be brief in the proof below and point out the main differences in our case.

By working with extensions of $u_1$, $u_2$, and $u_3$ on the whole time line, we can ignore the subscript $T$ for the $X^{s, b}$-norms. By duality and Plancherel's theorem, we have
\begin{align*}
\| &u_1 \cj{u_2} u_3 \|_{X^{s, -\frac 12 + 2 \eps}} \\ 
&= \sup_{\| v \|_{X^{s, \frac 12 - 2 \eps}} \leq 1} \bigg| \int_{\R \times \R} u_1 \cj{u_2} u_3 \cj{v} dx dt \bigg| \\
&= \sup_{\| v \|_{X^{s, \frac 12 - 2 \eps}} \leq 1} \bigg| \int_{\substack{\tau_1 + \tau_2 + \tau_3 + \tau_4 = 0 \\ \xi_1 + \xi_2 + \xi_3 + \xi_4 = 0}} \ft{u_1} (\tau_1, \xi_1) \ft{\cj{u_2}} (\tau_2, \xi_2) \ft{u_3} (\tau_3, \xi_3) \ft{\cj{v}} (\tau_4, \xi_4) d \tau_1 d \tau_2 d \tau_3 d \xi_1 d \xi_2 d \xi_3 \bigg|
\end{align*}
By defining
\begin{align*}
f_1 (\tau_1, \xi_1) &= \jb{\xi_1}^s \jb{\tau_1 + |\xi_1|^\al}^{\frac 12 + \eps} |\ft{u_1} (\tau_1, \xi_1)|, \\
f_2 (\tau_2, \xi_2) &= \jb{\xi_2}^s \jb{\tau_2 - |\xi_2|^\al}^{\frac 12 + \eps} |\ft{u_2} (\tau_2, \xi_2)|, \\
f_3 (\tau_3, \xi_3) &= \jb{\xi_3}^s \jb{\tau_3 + |\xi_3|^\al}^{\frac 12 + \eps} |\ft{u_3} (\tau_3, \xi_3)|, \\
f_4 (\tau_4, \xi_4) &= \jb{\xi_4}^{-s} \jb{\tau_4 - |\xi_4|^\al}^{\frac 12 - 2\eps} |\ft{v} (\tau_4, \xi_4)|,
\end{align*}

\noi
we only need to show the following estimate
\begin{align}
\int_{\substack{\tau_1 + \tau_2 + \tau_3 + \tau_4 = 0 \\ \xi_1 + \xi_2 + \xi_3 + \xi_4 = 0}} m_1 (\cj{\tau}, \cj{\xi}) \prod_{j=1}^4 f_j(\tau_j, \xi_j) d \tau_1 d \tau_2 d \tau_3 d \xi_1 d \xi_2 d \xi_3 \les \prod_{j=1}^4 \| f_j \|_{L^2_{\tau_j, \xi_j}},
\label{tri2}
\end{align}

\noi
where the multiplier $m_1 (\cj{\tau}, \cj{\xi})$ is defined by
\begin{align*}
m_1 ( \cj{\tau}, \cj{\xi} ) &\hphantom{:}= m_1 (\tau_1,\dots,\tau_4, \xi_1,\dots,\xi_4) \\
&:= \frac{ \langle \t_4 - | \xi_4 | ^{\al} \rangle ^{-	 \frac 12 + 2 \e} \langle \xi_1 \rangle^{-s}  \langle  \xi_ 2 \rangle^{-s}  \langle  \xi _3\rangle^{-s}       }{  \langle \t_1 + | \xi_1 | ^{\al} \rangle^{\frac 12 + \e}     \langle \t_2 - | \xi_2 | ^{\al} \rangle   ^{\frac 12 + \e} \langle \t_3 + | \xi_3 | ^{\al} \rangle ^{\frac 12 + \e}    \langle  \xi_4 \rangle^{-s}   }, 
\end{align*}

\noi
To prove \eqref{tri2}, we perform the same case-by-case analysis as in \cite[Proposition 3.5]{kihoon}. The main difference in our case is that, instead of using the fractorization \cite[(3.16)]{kihoon},
we use Lemma \ref{LEM:factor} to obtain
\begin{align*}
\max_{j = 1, 2, 3, 4} \big\{ |\tau_j + |\xi_j|^\al| \big\}\ges |\xi_1+\xi_2 | | \xi_2+\xi_3 |   | \xi_{\text{max}} |^{\al-2}
\end{align*}

\noi
under the conditions $\tau_1 + \tau_2 + \tau_3 + \tau_4 = 0$ and $\xi_1 + \xi_2 + \xi_3 + \xi_4 = 0$, where $\xi_{\text{max}} = \max \{ |\xi_1|, |\xi_2|, |\xi_3|, |\xi_4| \}$.
Furthermore, the uses of Strichartz estimates in \cite[Proposition 3.5]{kihoon} need to be replaced by our Strichartz estimates in Remark \ref{RMK:str}. The rest of the argument is similar and so we omit details.
\end{proof}

\begin{remark} \rm
\label{RM:cteg}
We remark that Proposition \ref{PROP:triR} fails when $s<\frac{2-\al}{4}$. This failure is due to the resonant interaction of high-high-high into
high frequencies.
A counterexample similar to that in \cite[Theorem 1.7]{KPV} can be constructed. 
In fact, the same construction in \cite[Section 4]{CHKL} also works in the $\al > 2$ setting.
\end{remark}

\subsection{Global well-posedness on the real line}
\label{SEC:Imd}
In this subsection, we prove Theorem \ref{TM:GWP1}, global well-posedness of the cubic fNLS \eqref{fnls1} on $\R$. As mentioned in Subsection \ref{SUBSEC:fNLS_R}, our goal is to bound the $H^s$-norm of the solution $u$ via the $I$-method, where $\frac{2 - \al}{4} \leq s < 0$.
Specifically, let us consider a smooth even monotone function $m : \R \to [0, 1]$ such that
\begin{align}
m(\xi):= 
\begin{cases}
1   & |\xi|<N \\ 
\big( \frac{|\xi|}{N} \big)^s  & |\xi|>2N.
\end{cases}
\label{Iop1}
\end{align}

\noi
Here, $N$ is a large parameter to be determined later. 
We define the Fourier multiplier
$I = I_N : H^s (\R) \to L^2 (\R)$
such that $\ft{Iu}(\xi) = m(\xi)\ft{u}(\xi)$.
The operator $I$ acts as an identity operator on low frequencies and as a smoothing operator on high frequencies.
Note that we have the following bounds:
\begin{equation}
\label{i-smoothing}
\| u\|_{H^s} \leq \| Iu \|_{L^2}\lesssim N^{-s} \| u\|_{H^s}.   
\end{equation}

We define our modified mass by
$$  M(Iu)=\|Iu\|^2_{L^2}.$$
Note that by \eqref{i-smoothing}, this modified mass makes sense if $u$ lies in $H^s$.
In general, the quantity $M(Iu)$ is not conserved in time, 
but we will show that its increments can be controlled in terms of $N$.



Let us first apply the $I$-operator to our cubic fNLS \eqref{fnls1}:
\begin{align}
\begin{cases}
i\partial_t Iu=I D^\al u + I(| u |^2 u) \\
Iu|_{t=0} = I u_0.
\end{cases}
\label{Isys1}
\end{align}

\noi
In the following proposition, we show a local well-posedness result for the $I$-system \eqref{Isys1} as a variant of local well-posedness of \eqref{fnls1}.
\begin{proposition}
\label{PROP:ILWP1}
Let $\al>2$ and $\frac{2-\al}{4} \leq s < 0$. 
Then, given any $u_0 \in H^s (\R)$, there exists $T = T(\| I u_0 \|_{L^2}) > 0$ and a unique solution $Iu \in C ([-T, T]; L^2 (\R))$ to the $I$-system \eqref{Isys1}. 
Furthermore, there exists $\eta> 0$ sufficiently small such that for any $t_0 \in \R$ and 
 $u (t_0) \in H^s(\R)$ satisfying
\begin{align*}
M(I u(t_0)) \leq \eta,
\end{align*} 

\noi
the solution $Iu$ to the $I$-system \eqref{Isys1}
exists on the interval $J = [t_0 - 1, t_0 + 1]$ with the bound
\begin{align*}
\| Iu\|_{X_J^{0,\frac 12 + \eps}} \leq 2 \eta
\end{align*}

\noi
for some $\eps > 0$ sufficiently small, uniformly in $ N \geq 1$.
\end{proposition}

\begin{proof} 
Following the proof of local well-posedness of the cubic fNLS \eqref{fnls1} as in Subsection~\ref{SEC:LWP1}, all we need to show is the following trilinear estimate:
\begin{align*}
\| I(u_1 \cj{u_2} u_3) \|_{X_T^{0,-\frac{1}{2} + 2 \eps}} \les \prod_{j = 1}^3 \| Iu_j \|_{X_T^{0,\frac{1}{2} + \eps}}.
\end{align*}	

\noi
By using the same reduction as in the proof of Proposition \ref{PROP:triR}, we only need to show
\begin{align*}
\int_{\substack{\tau_1 + \tau_2 + \tau_3 + \tau_4 = 0 \\ \xi_1 + \xi_2 + \xi_3 + \xi_4 = 0}} m_2 ( \cj{\tau}, \cj{\xi} ) \prod_{j=1}^4 f_j(\tau_j, \xi_j) d \tau_1 d \tau_2 d \tau_3 d \xi_1 d \xi_2 d \xi_3 \les \prod_{j=1}^4 \| f_j \|_{L^2_{\tau_j, \xi_j}},
\end{align*}

\noi
where $f_j$'s are nonnegative functions and the multiplier $m_2 (\cj{\tau}, \cj{\xi})$ is defined by
\begin{align*}
m_2 ( \cj{\tau}, \cj{\xi} ) &\hphantom{:}= m_2 (\tau_1,\dots,\tau_4, \xi_1,\dots,\xi_4) \\
&:= \frac{ \langle \t_4 - | \xi_4 | ^{\al} \rangle ^{-	 \frac 12 + 2 \e} m(\xi_4)  }{  \langle \t_1 + | \xi_1 | ^{\al} \rangle^{\frac 12 + \e}     \langle \t_2 - | \xi_2 | ^{\al} \rangle   ^{\frac 12 + \e} \langle \t_3 + | \xi_3 | ^{\al} \rangle ^{\frac 12 + \e} m(\xi_1) m(\xi_2) m(\xi_3)   }.
\end{align*}

\noi
Since the largest two frequencies are comparable, we can assume without loss of generality that $|\xi_1| \ges |\xi_4|$. Also, we have $m(\xi) \jb{\xi}^{-s} \ges 1$ for all $\xi \in \R$. Thus, we have
\begin{align*}
\frac{m(\xi_4) \jb{\xi_4}^{-s}}{m(\xi_1) \jb{\xi_1}^{-s} m(\xi_2) \jb{\xi_2}^{-s} m(\xi_3) \jb{\xi_3}^{-s}} \les \frac{1}{m(\xi_2) \jb{\xi_2}^{-s} m(\xi_3) \jb{\xi_3}^{-s}} \les 1,
\end{align*}

\noi
so that
\begin{align*}
m_2 ( \cj{\tau}, \cj{\xi} ) \les \frac{ \langle \t_4 - | \xi_4 | ^{\al} \rangle ^{-\frac 12 + 2 \e} \langle \xi_1 \rangle^{-s}  \langle  \xi_ 2 \rangle^{-s}  \langle  \xi _3\rangle^{-s}       }{  \langle \t_1 + | \xi_1 | ^{\al} \rangle^{\frac 12 + \e}     \langle \t_2 - | \xi_2 | ^{\al} \rangle   ^{\frac 12 + \e} \langle \t_3 + | \xi_3 | ^{\al} \rangle ^{\frac 12 + \e}    \langle  \xi_4 \rangle^{-s}   }.
\end{align*}

\noi
The rest of the steps follow from the same way as Proposition \ref{PROP:triR} (more precisely, \cite[Proposition~3.5]{kihoon}).

\end{proof}

Before proving Theorem \ref{TM:GWP1}, we need to further introduce a correction term and present some useful estimates. 

Let $u$ be a smooth solution to the fraction NLS
\begin{align}
\dt u = - i D^\al u - i |u|^2 u.
\label{smooth}
\end{align}

\noi
Let $d \geq 2$ be an even integer and let $M_d = M_d (\xi_1, \dots, \xi_d)$ be a multiplier of order $d$. Recalling the notations in Subsection \ref{SUBSEC:notations}, by taking the Fourier transform of \eqref{smooth}, we obtain the following differentiation law:
\begin{align}
\dt \Ld_d (M_d; u(t))
=
i \Ld_d \bigg( M_d \sum_{j=1}^d (-1)^{j} |\xi_j|^\al; u(t) \bigg) + i \Ld_{d+2} \bigg( \sum_{j=1}^d (-1)^j \mathbf X^2_j(M_d) ; u(t) \bigg).
\label{diff}
\end{align}

By Plancherel's theorem, we get
\begin{align*}
M(Iu(t)) &= \| Iu (t) \|_{L_x^2}^2 \\
&=\int_{\xi_1+\xi_2=0} m(\xi_1) m(\xi_2) \ft{u}(t, \xi_1) \ft{\cj{u}}(t, \xi_2) d \xi_1 \\
&=\Lambda_2 \big( m (\xi_1) m(\xi_2);u (t) \big).
\end{align*}

\noi
By differentiating $M(Iu)$ in time and using \eqref{diff}, we obtain
\begin{equation}
\begin{split}
\frac{d}{dt} M(Iu(t)) &= \frac{d}{dt}\Ld_2 \big( m(\xi_1)m(\xi_2);u(t) \big)  \\
&= -i \Ld_2 \big( m(\xi_1)m(\xi_2) (|\xi_1|^\al-|\xi_2|^\al); u(t) \big) \\
&\quad - i\Ld_4 \big( m(\xi_{123})m(\xi_4)-m(\xi_1)m(\xi_{234}); u(t) \big)
\end{split}
\label{dtM}
\end{equation}

\noi
Since $\xi_1+\xi_2=0$ in $\Ld_2$,
the first term on the right-hand-side of \eqref{dtM} vanishes. For the second term on the right-hand-side of \eqref{dtM}, by using the condition $\xi_1 + \xi_2 + \xi_3 + \xi_4 = 0$ and symmetrizing the second term, we have
\begin{align}
\frac{d}{dt}M(Iu(t))&=\frac{i}{2} \Ld_4 \big( m(\xi_1)^2 - m(\xi_2)^2 + m(\xi_3)^2 - m(\xi_4)^2; u(t) \big) = \Ld_4( M_4; u(t) ),
\label{dM}
\end{align}

\noi
where the multiplier $M_4$ is defined by
\begin{align*}
M_4(\xi_1,\ldots,\xi_4) := \frac{i}{2} \big( m(\xi_1)^2 - m(\xi_2)^2 + m(\xi_3)^2 - m(\xi_4)^2   \big). 
\end{align*}

We define
\begin{align}
\s_4 (\xi_1, \dots, \xi_4) := - \frac{i M_4 (\xi_1, \dots, \xi_4)}{|\xi_1|^\al - |\xi_2|^\al + |\xi_3|^\al - |\xi_4|^\al},
\label{sigma4}
\end{align}

\noi
which is always finite in view of Lemma \ref{LEM:si4R} below. Let us now define a new modified mass with a correction term:
\begin{align}
M^4(Iu(t)) := M(Iu(t)) + \Lambda_4 ( \sigma_4; u(t) ) 
\label{mEng}   
\end{align}

\noi
The role of $\sigma_4$ here is to make a cancellation. 
Specifically, by using the differentiation law \eqref{diff}, we obtain
\begin{align*}
\frac{d}{dt}\Lambda_4 (\sigma_4;u(t)) &= - i\Ld_4 \big(\s_4 (|\xi_1|^\al-|\xi_2|^\al+|\xi_3|^\al-|\xi_4|^\al ); u(t) \big) \\
&\quad -i \Ld_6 \big( \mathbf X_1^2(\s_4) -\mathbf X_2^2(\s_4)+\mathbf X_3^2(\s_4)-\mathbf X_4^2(\s_4);u(t) \big),
\end{align*}

\noi 
where
\begin{align*}
\mathbf X_1^2(\s_4) (\xi_1, \dots, \xi_6) &= \s_4(\xi_{123},\xi_4,\xi_5,\xi_{6}), \\
\mathbf X_2^2(\s_4) (\xi_1, \dots, \xi_6) &= \s_4(\xi_1,\xi_{234},\xi_5,\xi_{6}), \\
\mathbf X_3^2(\s_4) (\xi_1, \dots, \xi_6) &= \s_4(\xi_1,\xi_2,\xi_{345},\xi_{6}), \\
\mathbf X_4^2(\s_4)(\xi_1, \dots, \xi_6) &= \s_4(\xi_1,\xi_2,\xi_3,\xi_{456}).
\end{align*}

\noi
By \eqref{sigma4}, the symmetry between $\mathbf X_1^2(\s_4)$ and $\mathbf X_3^2(\s_4)$, the symmetry between $\mathbf X_2^2(\s_4)$ and $\mathbf X_4^2(\s_4)$, and conjugation,
we can write
\begin{align*}
\frac{d}{dt}\Ld_4 (\sigma_4;u(t)) &= - \Ld_4 ( M_4; u(t) ) - 2i \Ld_6 \big( \s_4(\xi_{123},\xi_4,\xi_5,\xi_{6}) -\s_4(\xi_1,\xi_2,\xi_3,\xi_{456});u(t) \big) \\
&= - \Ld_4 ( M_4; u(t) ) + 4 \Re i \Ld_6 \big( \s_4(\xi_1,\xi_2,\xi_3,\xi_{456});u(t) \big).
\end{align*}

\noi
Thus, by differentiating \eqref{mEng} and using \eqref{dM}, we obtain
\begin{align}
\frac{d}{dt} M^4(Iu(t))
= 4 \Re i\Lambda_6 \big( \sigma_4 (\xi_1,\xi_2,\xi_3,\xi_{456} ); u(t) \big) = 4 \Re \Ld_4 (M_6; u(t) ),
\label{dM4}
\end{align}

\noi
where the multiplier $M_6$ is defined by
\begin{align*}
M_6 (\xi_1, \dots, \xi_6) := i \s_4 (\xi_1,\xi_2,\xi_3,\xi_{456}).
\end{align*}

\smallskip
We now present several estimates related to the terms introduced above, starting with the following bound on $\s_4$ defined in \eqref{sigma4}.

\begin{lemma}
\label{LEM:si4R}
Let $\al > 2$ and $m$ be the multiplier defined in \eqref{Iop1}. Let $|\xi_j|\sim N_j$ with $j=1,2,3,4$, where $N_j \geq 1$ are dyadic numbers. Let $N^{(1)}, N^{(2)}, N^{(3)}, N^{(4)}$ be a rearrangement of $N_1, N_2, N_3, N_4$ such that $N^{(1)} \geq N^{(2)} \geq N^{(3)} \geq N^{(4)}$. Then, under the condition $\xi_1 + \cdots + \xi_4 = 0$, we have the bound
\begin{align*}
|\sigma_4 (\xi_1,\dots,\xi_4 ) | \les \frac{m (N^{(4)})^2}{( N + N^{(3)} )( N + N^{(1)} )^{\al-1}}. 
\end{align*}
\end{lemma} 

\begin{proof}	
As mentioned in the proof of Proposition \ref{PROP:triR}, we have
\begin{align*}
\big| |\xi_1|^\al - |\xi_2|^\al + |\xi_3|^\al - |\xi_4|^\al \big| \ges |\xi_1+\xi_2 | | \xi_1+\xi_4 |   | \xi_{\text{max}} |^{\al-2}
\end{align*} 

\noi
under $\xi_1 + \xi_2 + \xi_3 + \xi_4 = 0$, where $\xi_{\text{max}} = \max \{ |\xi_1|, |\xi_2|, |\xi_3|, |\xi_4| \}$. Thus, we have
\begin{align*}
|\sigma_4 (\xi_1,\dots,\xi_4 ) | \les	\frac{ \big| m(\xi_1)^2 - m(\xi_2)^2 + m(\xi_3)^2 - m(\xi_4)^2 \big|}{|\xi_1+\xi_2| |\xi_1+\xi_4| |\xi_{\text{max}}|^{\al-2} }.
\end{align*}

\noi
The rest of the steps follow from minor modifications of \cite[Lemma 4.1]{kihoon}, and so we omit details.
\end{proof}

We now have the following lemma regarding the difference between $M^4 (I u)$ and $M (I u)$.

\begin{lemma}
\label{LEM:Mdiff}
Let $\al>2$, $\frac{2-\al}{4}\leq s<0$, and $u$ be a  smooth solution of the fraction NLS \eqref{fnls1} on an interval $J \subset \R$.
Then, for any $t \in J$, we have
\begin{align*}
\big| M^4(Iu(t)) - M(Iu(t)) \big| \les \| Iu(t) \|_{L^2(\R)}^4.
\end{align*}
\end{lemma}

\begin{proof}
From \eqref{mEng}, we only need to show
\begin{align*}
|\Ld_4 (\s_4; u(t))| \les \| I u (t) \|_{L^2}^4.
\end{align*}

\noi
By letting $v = I u$, it suffices to show
\begin{align*}
|\Ld_4 (\s_4'; v(t))| \les \| v (t) \|_{L^2}^4,
\end{align*}

\noi
where
\begin{align*}
\s_4' (\xi_1, \dots, \xi_4) = \frac{\s_4 (\xi_1, \dots, \xi_4)}{ \prod_{j = 1}^4 m (\xi_j) }.
\end{align*}

For simplicity, for a dyadic number $K \geq 1$, we denote $v_K = \pi_K v (t)$. We dyadically decompose all $v$'s with frequency scales $N_1, N_2, N_3, N_4$, and we let $N^{(1)}, N^{(2)}, N^{(3)}, N^{(4)}$ be a rearrangement of these dyadic frequency scales such that $N^{(1)} \geq N^{(2)} \geq N^{(3)} \geq N^{(4)}$. By Lemma \ref{LEM:si4R}, the fact that $\jb{\xi}^{s} \les m (\xi) \leq 1$, Young's convolution inequalities, and H\"older's inequalities, we obtain
\begin{align*}
|\Ld_4 (\s_4'; v(t))| &\leq \sum_{\substack{N_1, \dots, N_4 \geq 1 \\ \text{dyadic}}} \int_{\G_4} | \s_4' (\xi_1, \dots, \xi_4) | \prod_{j = 1}^4 | \ft{v_{N_j}} (\xi_j)| d\xi_1 \cdots d\xi_4 \\
&\les \sum_{\substack{N_1, \dots, N_4 \geq 1 \\ \text{dyadic}}} \frac{m (N^{(4)})^2}{( N + N^{(3)} )( N + N^{(1)} )^{\al-1} \prod_{j = 1}^4 m(N^{(1)})} \big( |\ft{v_{N_1}}| * \cdots * |\ft{v_{N_4}}| \big) (0) \\
&\les \sum_{\substack{N_1, \dots, N_4 \geq 1 \\ \text{dyadic}}} \frac{1}{(N^{(1)})^{\al - 1 + s} (N^{(2)})^s (N^{(3)})^{1 + s}} \| \ft{v_{N^{(1)}}} \|_{L^2} \| \ft{v_{N^{(2)}}} \|_{L^2} \| \ft{v_{N^{(3)}}} \|_{L^1} \| \ft{v_{N^{(4)}}} \|_{L^1} \\
&\leq \sum_{\substack{N_1, \dots, N_4 \geq 1 \\ \text{dyadic}}} \frac{1}{(N^{(1)})^{\al - 1 + 3s}} \prod_{j = 1}^4 \| \ft{v_{N^{(j)}}} \|_{L^2} \\
&\les \| v (t) \|_{L^2}^4,
\end{align*}

\noi
where the last inequality holds since $\al - 1 + 3s > 0$ given $s \geq \frac{2 - \al}{4}$ and $\al > 2$. This finishes our proof.
\end{proof}

We now show the following almost conservation law for $M^4 (I u)$.

\begin{proposition}
\label{PROP:aclaw}
Let $\al>2$, $\frac{2-\al}{4}\leq s<0$, $\eps > 0$ and $u$ be a smooth solution of the cubic fNLS~\eqref{fnls1} on an interval $J \subset \R$. Then, for any $t, t_0 \in J$, we have
\begin{align*}
\big|  M^4(Iu(t))- M^4(Iu(t_0)) \big| \les  N^{- 2 \al + 2 + \dl} \| Iu \|_{X_J^{0,\frac{1}{2}+\eps}}^6
\end{align*}

\noi
for $\dl > 0$ arbitrarily small.
\end{proposition}

\begin{proof}
By \eqref{dM4}, we have
\begin{align*}
\begin{split}
\big|  M^4(Iu(t))- M^4(Iu(t_0)) \big| &= \bigg| \int_{t_0}^t \frac{d}{dt} M^4 (Iu (t')) dt' \bigg| \\
&= 4 \bigg| \int_{t_0}^t \Ld_4 (M_6; u(t')) dt'  \bigg| \\
&= \int_{t_0}^t \int_{\G_6} |\s_4 (\xi_1, \xi_2, \xi_3, \xi_{456})| \prod_{j = 1}^6 |\ft{u} (t', \xi_j)| d\xi_1 \cdots d\xi_6 dt'. 
\end{split}
\end{align*}

\noi
As in the proof of Lemma \ref{LEM:Mdiff}, we let $v = Iu$, so that we only need to show
\begin{align*}
\int_{t_0}^t \int_{\G_6} |\s_4' (\xi_1, \xi_2, \xi_3, \xi_{456})| \prod_{j = 1}^6 |\ft{v} (t', \xi_j)| d\xi_1 \cdots d\xi_6 dt' \les N^{-2 \al + 2} \| v \|_{X^{0, \frac 12 + \eps}}^6,
\end{align*}

\noi
where
\begin{align*}
\s_4' (\xi_1, \xi_2, \xi_3, \xi_{456}) = \frac{\s_4 (\xi_1, \xi_2, \xi_3, \xi_{456})}{\prod_{j = 1}^6 m (\xi_j)}.
\end{align*}

We now use a dyadic decomposition so that $|\xi_j| \sim N_j$ for some dyadic numbers $N_j \geq 1$ for $j = 1, \dots, 6$, and also $|\xi_{456}| \sim N_{456}$ for some dyadic number $N_{456} \geq 1$. We denote by $N^{(1)}, N^{(2)}, N^{(3)}, N^{(4)}$ a rearrangement of $N_1, N_2, N_3, N_{456}$ such that $N^{(1)} \geq N^{(2)} \geq N^{(3)} \geq N^{(4)}$. We also define $N_{\text{max}} = \max \{N_1, \dots, N_6\}$. Note that if $N_1, N_2, N_3, N_{456} \ll N$, then by \eqref{sigma4} we know that $\s_4 \equiv 0$, and so the estimate follows directly. Thus, since two largest frequencies are comparable, we can assume that two of the frequency scales among $N_1, N_2, N_3, N_{456}$ are $\ges N$. By symmetry, we can also assume that $N_1 \geq N_2 \geq N_3$ and $N_4 \geq N_5 \geq N_6$.

As in \cite[Lemma 4.3]{kihoon}, we consider the following four cases. For simplicity, for a dyadic number $K \geq 1$, we denote $v_K = \pi_K v$.

\medskip \noi
\textbf{Case 1:} $N \gg N_4 \geq N_5 \geq N_6$. In this case, we must have $N_1 \sim N_2 \ges N \gg N_4 \geq N_5 \geq N_6$, $m(N_4) = m(N_5) = m(N_6) = 1$, $m(N_3)^{-1} \leq \max(1, N^s N_3^{-s}) \les N^s N_1^{-s}$, and $N_1 \sim N^{(1)} \sim N_{\text{max}}$. Thus, by Lemma \ref{LEM:si4R}, Young's convolution inequalities, H\"older's inequalities, Plancherel's theorem, the bilinear Strichartz estimate \eqref{bistr}, and Lemma \ref{LEM:emb}, we obtain
\begin{align*}
\int_{t_0}^t &\int_{\G_6} |\s_4' (\xi_1, \xi_2, \xi_3, \xi_{456})| \prod_{j = 1}^6 |\ft{v_{N_j}} (t', \xi_j)| d\xi_1 \cdots d\xi_6 dt' \\
&\les \sum_{\substack{N_1, \dots, N_6, N_{456} \geq 1 \\ \text{dyadic}}} \frac{N^{3s} N_1^{-2s} N_2^{-s} }{(N + N^{(3)}) (N + N^{(1)})^{\al - 1} } \int_{t_0}^t \int_{\G_6} \prod_{j = 1}^6 |\ft{v_{N_j}} (t', \xi_j)| d\xi_1 \cdots d\xi_6 dt' \\
&\les \sum_{\substack{N_1, \dots, N_6, N_{456} \geq 1 \\ \text{dyadic}}} \frac{N^{3s} N_1^{-2s} N_2^{-s}}{(N + N^{(3)}) (N + N^{(1)})^{\al - 1}}  \| v_{N_1} v_{N_4} \|_{L^2_J L^2_x} \| v_{N_2} v_{N_5} \|_{L^2_J L^2_x} \| \ft{v_{N_3}} \|_{L^\infty_J L_\xi^1} \| \ft{v_{N_6}} \|_{L^\infty_J L_\xi^1} \\
&\les \sum_{\substack{N_1, \dots, N_6, N_{456} \geq 1 \\ \text{dyadic}}} \frac{N^{3s} N_1^{\frac{1 - \al}{2} - 2s} N_2^{\frac{1 - \al}{2} - s} N_3^{\frac 12} N_6^{\frac 12}}{(N + N^{(3)}) (N + N^{(1)})^{\al - 1}} \prod_{j = 1}^6 \| v_{N_j} \|_{X_J^{0, \frac 12 + \eps}} \\
&\les \sum_{\substack{N_1, \dots, N_6, N_{456} \geq 1 \\ \text{dyadic}}} \frac{N^{3s}}{(N + N^{(1)})^{2 \al - 2 + 3s}} \| v \|_{X_J^{0, \frac 12 + \eps}}^6 \\
&\les N^{-2 \al + 2 + \dl} \| v \|_{X_J^{0, \frac 12 + \eps}}^6,
\end{align*}

\noi
where we used $N_3 \les N^{(3)}$, $N_6 \ll N$, $N_1 \sim N_2 \sim N^{(1)} \sim N + N^{(1)}$, and $2\al - 2 + 3s > 0$ given $\frac{2 - \al}{4} \leq s < 0$ and $\al > 2$. This is the desired bound.

\medskip \noi
\textbf{Case 2:} $N_4 \ges N \gg N_5 \ges N_6$. In this case, we have $N_1 \ges N_4$ since two largest frequencies among $N_1, \dots, N_6$ are comparable, so that $N_1 \ges N \gg N_5 \ges N_6$, $m(N_5) = m(N_6) = 1$, $m(N_2)^{-1} \les N^s N_1^{-s}$, $m(N_3)^{-1} \les N^s N_1^{-s}$, and $N_1 \sim N^{(1)} \sim N_{\text{max}}$.
Using similar steps as in Case 1, we obtain
\begin{align*}
\int_{t_0}^t &\int_{\G_6} |\s_4' (\xi_1, \xi_2, \xi_3, \xi_{456})| \prod_{j = 1}^6 |\ft{v_{N_j}} (t', \xi_j)| d\xi_1 \cdots d\xi_6 dt' \\
&\les \sum_{\substack{N_1, \dots, N_6, N_{456} \geq 1 \\ \text{dyadic}}} \frac{N^{4s} N_1^{-3s} N_4^{-s}}{(N + N^{(3)}) (N + N^{(1)})^{\al - 1} } \| v_{N_1} v_{N_5} \|_{L_J^2 L_x^2} \| v_{N_4} v_{N_6} \|_{L_J^2 L_x^2} \| \ft{v_{N_2}} \|_{L^\infty_J L^1_\xi} \| \ft{v_{N_3}} \|_{L^\infty_J L^1_\xi} \\
&\les \sum_{\substack{N_1, \dots, N_6, N_{456} \geq 1 \\ \text{dyadic}}} \frac{N^{4s} N_1^{\frac{1 - \al}{2} - 3s} N_2^{\frac 12} N_3^{\frac 12} N_4^{\frac{1 - \al}{2} - s}}{(N + N^{(3)}) (N + N^{(1)})^{\al - 1}} \prod_{j = 1}^6 \| v_{N_j} \|_{X^{0, \frac 12 + \eps}_J} \\
&\les \sum_{\substack{N_1, \dots, N_6, N_{456} \geq 1 \\ \text{dyadic}}} \frac{N^{\frac{1 - \al}{2} + 3s} }{(N + N^{(3)})^{\frac 12} (N + N^{(1)})^{\frac 32 \al - 2 + 3s}} \| v \|_{X_J^{0, \frac 12 + \eps}}^6 \\
&\les N^{-2\al + 2 + \dl} \| v \|_{X_J^{0, \frac 12 + \eps}}^6,
\end{align*}

\noi
where we used $N_3 \les N^{(3)}$, $N_4 \ges N$, $N_1 \sim N^{(1)} \sim N + N^{(1)}$, $\frac{1 - \al}{2} - s < 0$, and $\frac 32 \al - 2 + 3s > 0$ (given $\frac{2 - \al}{4} \leq s < 0$ and $\al > 2$). This is the desired bound.

\medskip \noi
\textbf{Case 3:} $N_4 \geq N_5 \ges N \gg N_6$. In this case, we have $m(N_6) = 1$. 

\medskip \noi
\textbf{Subcase 3.1:} $N_1 \ges N_4$. In this subcase, we have $N_1 \ges N$, $m(N_2)^{-1} \les N^s N_1^{-s}$, $m(N_3)^{-1} \les N^s N_1^{-s}$, and $N_1 \sim N^{(1)} \sim N_{\text{max}}$. Using Young's inequality, Plancherel's theorem, H\"older's inequality, the bilinear Strichartz estimate \eqref{bistr}, Bernstein's inequality, and the Strichartz estimate~\eqref{str1}, we obtain
\begin{align*}
\int_{t_0}^t &\int_{\G_6} |\s_4' (\xi_1, \xi_2, \xi_3, \xi_{456})| \prod_{j = 1}^6 |\ft{v_{N_j}} (t', \xi_j)| d\xi_1 \cdots d\xi_6 dt' \\
&\les \sum_{\substack{N_1, \dots, N_6, N_{456} \geq 1 \\ \text{dyadic}}} \frac{N^{5s} N_1^{-3s} N_4^{-s} N_5^{-s}}{(N + N^{(3)}) (N + N^{(1)})^{\al - 1} } \| v_{N_1} v_{N_6} \|_{L_J^2 L_x^2} \| v_{N_2} v_{N_3} v_{N_4} v_{N_5} \|_{L_J^2 L_x^2} \\
&\les \sum_{\substack{N_1, \dots, N_6, N_{456} \geq 1 \\ \text{dyadic}}} \frac{N^{5s} N_1^{-3s} N_4^{-s} N_5^{-s}}{(N + N^{(3)}) (N + N^{(1)})^{\al - 1} } \\
&\quad \times \| v_{N_1} v_{N_6} \|_{L_J^2 L_x^2} \| v_{N_2} \|_{L_J^\infty L_x^2} \| v_{N_3} \|_{L_J^\infty L_x^\infty} \| v_{N_4} \|_{L_J^4 L_x^\infty} \| v_{N_5} \|_{L_J^4 L_x^\infty} \\
&\les \sum_{\substack{N_1, \dots, N_6, N_{456} \geq 1 \\ \text{dyadic}}} \frac{N^{5s} N_1^{\frac{1 - \al}{2} - 3s} N_3^{\frac 12} N_4^{\frac{2 - \al}{4} - s} N_5^{\frac{2 - \al}{4} - s}}{(N + N^{(3)}) (N + N^{(1)})^{\al - 1} } \prod_{j = 1}^6 \| v_{N_j} \|_{X^{0, \frac 12 + \eps}_J} \\
&\les \sum_{\substack{N_1, \dots, N_6, N_{456} \geq 1 \\ \text{dyadic}}} \frac{N^{\frac{2 - \al}{2} + 3s}}{(N + N^{(3)})^{\frac 12} (N + N^{(1)})^{\frac 32 \al - \frac 32 + 3s}} \| v \|_{X_J^{0, \frac 12 + \eps}}^6 \\
&\les N^{-2 \al + 2 + \dl} \| v \|_{X_J^{0, \frac 12 + \eps}}^6,
\end{align*}

\noi
where we used $N_3 \les N^{(3)}$, $N_4 \ges N$, $N_5 \ges N$, $N_1 \sim N^{(1)} \sim N + N^{(1)}$, $\frac{2 - \al}{4} - s < 0$, and $\frac 32 \al - \frac 32 + 3s > 0$ (given $\frac{2 - \al}{4} \leq s < 0$ and $\al > 2$). This is the desired bound.

\medskip \noi
\textbf{Subcase 3.2:} $N_1 \ll N_4$. In this subcase, since two largest frequencies are comparable, we have $N_4 \sim N_5 \sim N_{\text{max}}$, so that $N_5 \gg N_2$, $m(N_1)^{-1} \les N^s (N + N^{(1)})^{-s}$, $m(N_2)^{-1} \les N^s N_4^{-s}$, $m(N_3)^{-1} \les N^s N_4^{-s}$. Using similar steps as in Case 1, we obtain
\begin{align*}
\int_{t_0}^t &\int_{\G_6} |\s_4' (\xi_1, \xi_2, \xi_3, \xi_{456})| \prod_{j = 1}^6 |\ft{v_{N_j}} (t', \xi_j)| d\xi_1 \cdots d\xi_6 dt' \\
&\les \sum_{\substack{N_1, \dots, N_6, N_{456} \geq 1 \\ \text{dyadic}}} \frac{N^{5s} N_4^{-3s} N_5^{-s}}{(N + N^{(3)}) (N + N^{(1)})^{\al - 1 + s} } \\
&\quad \times \| v_{N_1} v_{N_4} \|_{L_J^2 L_x^2} \| v_{N_2} v_{N_5} \|_{L_J^2 L_x^2} \| \ft{v_{N_3}} \|_{L^\infty_J L^1_\xi} \| \ft{v_{N_6}} \|_{L^\infty_J L^1_\xi} \\
&\les \sum_{\substack{N_1, \dots, N_6, N_{456} \geq 1 \\ \text{dyadic}}} \frac{N^{5s} N_3^{\frac 12} N_4^{\frac{1 - \al}{2} - 3s} N_5^{\frac{1 - \al}{2} - s} N_6^{\frac 12}}{(N + N^{(3)}) (N + N^{(1)})^{\al - 1 + s} } \prod_{j = 1}^6 \| v_{N_j} \|_{X^{0, \frac 12 + \eps}_J} \\
&\les \sum_{\substack{N_1, \dots, N_6, N_{456} \geq 1 \\ \text{dyadic}}} \frac{N^{5s} N_{\text{max}}^{1 - \al - 4s}}{(N + N^{(1)})^{\al - 1 + s}} \| v \|_{X^{0, \frac 12 + \eps}_J}^6 \\
&\les N^{-2 \al + 2 + \dl} \| v \|_{X^{0, \frac 12 + \eps}_J}^6,
\end{align*}

\noi
where we used $N_3 \les N^{(3)}$, $N_6 \ll N$, $N_{\text{max}} \ges N$, $1 - \al - 4s < 0$, and $\al - 1 + s > 0$ (given $\frac{2 - \al}{4} \leq s < 0$ and $\al > 2$). This is the desired bound.

\medskip \noi
\textbf{Case 4:} $N_4 \geq N_5 \geq N_6 \ges N$.

\medskip \noi
\textbf{Subcase 4.1:} $N_1 \ges N_4$. In this subcase, we have $N_1 \ges N$, $m(N_2)^{-1} \les N^s N_1^{-s}$, $m(N_3)^{-1} \les N^s N_1^{-s}$, and $N_1 \sim N^{(1)} \sim N_{\text{max}}$. Using Young's inequality, Plancherel's theorem, H\"older's inequality, and the Strichartz estimate \eqref{str1}, we obtain
\begin{align*}
\int_{t_0}^t &\int_{\G_6} |\s_4' (\xi_1, \xi_2, \xi_3, \xi_{456})| \prod_{j = 1}^6 |\ft{v_{N_j}} (t', \xi_j)| d\xi_1 \cdots d\xi_6 dt' \\
&\les \sum_{\substack{N_1, \dots, N_6, N_{456} \geq 1 \\ \text{dyadic}}} \frac{N^{6s} N_1^{-3s} N_4^{-s} N_5^{-s} N_6^{-s}}{(N + N^{(3)}) (N + N^{(1)})^{\al - 1} } \| v_{N_1} v_{N_2} \|_{L_J^4 L_x^2} \| v_{N_3} v_{N_4} v_{N_5} v_{N_6} \|_{L_J^{\frac 43} L_x^2} \\ 
&\les \sum_{\substack{N_1, \dots, N_6, N_{456} \geq 1 \\ \text{dyadic}}} \frac{N^{6s} N_1^{-3s} N_4^{-s} N_5^{-s} N_6^{-s}}{(N + N^{(3)}) (N + N^{(1)})^{\al - 1} } \\
&\quad \times \| v_{N_1} \|_{L_J^4 L_x^\infty} \| v_{N_2} \|_{L_J^\infty L_x^2} \| v_{N_3} \|_{L_J^\infty L_x^2} \| v_{N_4} \|_{L_J^4 L_x^\infty} \| v_{N_5} \|_{L_J^4 L_x^\infty} \| v_{N_6} \|_{L_J^4 L_x^\infty} \\
&\les \sum_{\substack{N_1, \dots, N_6, N_{456} \geq 1 \\ \text{dyadic}}} \frac{N^{6s} N_1^{\frac{2 - \al}{4} - 3s} N_4^{\frac{2 - \al}{4} - s} N_5^{\frac{2 - \al}{4} - s} N_6^{\frac{2 - \al}{4} - s}}{(N + N^{(3)}) (N + N^{(1)})^{\al - 1} } \prod_{j = 1}^6 \| v_{N_j} \|_{X^{0, \frac 12 + \eps}_J} \\
&\les \sum_{\substack{N_1, \dots, N_6, N_{456} \geq 1 \\ \text{dyadic}}} \frac{N^{\frac{6 - 3\al}{4} + 3s}}{(N + N^{(3)}) (N + N^{(1)})^{\frac 54 \al - \frac 32 + 3s} } \| v \|_{X^{0, \frac 12 + \eps}_J}^6 \\
&\les N^{-2 \al + 2 + \dl} \| v \|_{X^{0, \frac 12 + \eps}_J}^6,
\end{align*}

\noi
where we used $N_4 \ges N$, $N_5 \ges N$, $N_6 \ges N$, $N_1 \sim N^{(1)} \sim N + N^{(1)}$, $\frac{2 - \al}{4} - s$, and $\frac 54 \al - \frac 32 + 3s > 0$ (given $\frac{2 - \al}{4} \leq s < 0$ and $\al > 2$). This is the desired bound.

\medskip \noi
\textbf{Subcase 4.2:} $N_1 \ll N_4$. In this subcase, since two largest frequencies are comparable, we have $N_4 \sim N_5 \sim N_{\text{max}}$, so that $N_5 \gg N_2$, $m(N_1)^{-1} \les N^s (N + N^{(1)})^{-s}$, $m(N_2)^{-1} \les N^s (N + N^{(1)})^{-s}$, $m(N_3)^{-1} \les N^s (N + N^{(1)})^{-s}$. Using similar steps as in Case 1, we obtain
\begin{align*}
\int_{t_0}^t &\int_{\G_6} |\s_4' (\xi_1, \xi_2, \xi_3, \xi_{456})| \prod_{j = 1}^6 |\ft{v_{N_j}} (t', \xi_j)| d\xi_1 \cdots d\xi_6 dt' \\
&\les \sum_{\substack{N_1, \dots, N_6, N_{456} \geq 1 \\ \text{dyadic}}} \frac{N^{6s} N_4^{-s} N_5^{-s} N_6^{-s}}{(N + N^{(3)}) (N + N^{(1)})^{\al - 1 + 3s} } \\
&\quad \times \| v_{N_1} v_{N_4} \|_{L_J^2 L_x^2} \| v_{N_2} v_{N_5} \|_{L_J^2 L_x^2} \| \ft{v_{N_3}} \|_{L^\infty_J L^1_\xi} \| \ft{v_{N_6}} \|_{L^\infty_J L^1_\xi} \\
&\les \sum_{\substack{N_1, \dots, N_6, N_{456} \geq 1 \\ \text{dyadic}}} \frac{N^{6s} N_3^{\frac 12} N_4^{\frac{1 - \al}{2} - s} N_5^{\frac{1 - \al}{2} - s} N_6^{\frac 12 - s}}{(N + N^{(3)}) (N + N^{(1)})^{\al - 1 + 3s} } \prod_{j = 1}^6 \| v_{N_j} \|_{X^{0, \frac 12 + \eps}_J} \\
&\les \sum_{\substack{N_1, \dots, N_6, N_{456} \geq 1 \\ \text{dyadic}}} \frac{N^{6s} N_{\text{max}}^{\frac 32 - \al - 3s}}{(N + N^{(3)})^{\frac 12} (N + N^{(1)})^{\al - 1 + 3s}} \| v \|_{X^{0, \frac 12 + \eps}_J}^6 \\
&\les N^{-2 \al + 2 + \dl} \| v \|_{X^{0, \frac 12 + \eps}_J}^6,
\end{align*}

\noi
where we used $N_3 \les N^{(3)}$, $N_{\text{max}} \ges N$, $\frac 32 - \al - 3s < 0$, and $\al - 1 + 3s > 0$ (given $\frac{2 - \al}{4} \leq s < 0$ and $\al > 2$). This is the desired bound.
\end{proof}


We are now ready to prove Theorem \ref{TM:GWP1}.
\begin{proof}[Proof of Theorem \ref{TM:GWP1}]
For simplicity, we only consider well-posedness on the time interval $[0, T]$, and the argument for $[-T, 0]$ is the same. By the continuous dependence of the solution on initial data, we know that Lemma \ref{LEM:Mdiff} and Proposition \ref{PROP:aclaw} work for a general solution $u$ of \eqref{fnls1} with initial data $u_0 \in H^s (\R)$.

We first perform a scaling argument. Note that for any $\ld > 1$, if $u$ is a solution to \eqref{fnls1} with initial data $u_0$, then $u^\ld (t, x) = \ld^{- \frac{\al}{2}} u (\ld^{- \al} t, \ld^{-1} x)$ is also a solution to \eqref{fnls1} with initial data $u_{0}^\ld (x) = \ld^{- \frac{\al}{2}} u_0 (\ld^{-1} x)$. A straightforward computation yields
\begin{align*}
\| I u_{0}^\ld \|_{L^2} \les \ld^{\frac{1 - \al}{2} - s} N^{-s} \| u_0 \|_{H^s}.
\end{align*}

\noi
Note that $\frac{1 - \al}{2} - s < 0$ given $\frac{2 - \al}{4} \leq s < 0$. We will choose the parameter $N = N(T) \gg 1$ later, but we can take $\ld = \ld (N, \| u_0 \|_{H^s})$ large enough such that
\begin{align*}
\| I u_{0}^\ld \|_{L^2} < \eta_0 \leq \frac{\eta}{2}
\end{align*} 

\noi
with $\eta > 0$ be as given in Proposition \ref{PROP:ILWP1}, which can be satisfied by taking $\ld \sim N^{\frac{2s}{1 - \al - 2s}}$.  

Our goal is to construct the solution $u^\ld$ on $[0, \ld^\al T]$. By Proposition \ref{PROP:ILWP1}, we know that $u^\ld$ exists on $[0, 1]$ and that
\begin{align}
M (I u^\ld (t)) \leq C_0 \| Iu^\ld \|_{X^{0, \frac 12 + \eps}_{[0, 1]}}^2 \leq 4C_0 \eta_0^2
\label{gwpR-0}
\end{align}

\noi
for $t \in [0, 1]$ and some constant $C_0 > 1$. By Lemma \ref{LEM:Mdiff} along with \eqref{gwpR-0}, we have
\begin{align}
M^4 (Iu^\ld (t)) = M(Iu^\ld (t)) + \mathcal{O} (\eta_0^4)
\label{gwpR-1}
\end{align}

\noi
for $t \in [0, 1]$. By Proposition \ref{PROP:aclaw}, we have
\begin{align}
M^4 (Iu^\ld (t)) \leq M^4 (Iu^\ld (0)) + C_1 N^{-2 \al + 2 + \dl} \eta_0^6
\label{gwpR-2} 
\end{align}

\noi
for some $\dl > 0$ arbitrarily small and some constant $C_1 > 0$ independent of $N$. Thus, by \eqref{i-smoothing}, \eqref{gwpR-1}, and \eqref{gwpR-2}, we have
\begin{align}
\begin{split}
\| u^\ld (1) \|_{H^s}^2 &\leq M( I u^\ld (1)) \\ 
&= M^4 ( I u^\ld (1)) + \mathcal{O} ( \eta_0^4 ) \\
&\leq M^4 (Iu^\ld (0)) + C_1 N^{-2 \al + 2 + \dl} \eta_0^6 + \mathcal{O} ( \eta_0^4 ) \\
&\leq \eta_0^2 + C_1 N^{-2 \al + 2 + \dl} \eta_0^6 + \mathcal{O} (\eta_0^4) \\
&\leq 4 \eta_0^2 \leq \eta^2.
\end{split}
\label{gwpR-3}
\end{align}

\noi
Thus, we can use Proposition \ref{PROP:ILWP1} again to extend the solution $u^\ld$ to the time interval $[0, 2]$.

By iterating the above procedure $k$ times, we obtain
\begin{align*}
M^4 (I u^\ld (t)) \leq M^4 (I u^\ld (0)) + (k + 1) C_1 N^{-2 \al + 2 + \dl} \eta_0^6
\end{align*}

\noi
for $t \in [0, k + 1]$. Thus, as long as $(k + 1) N^{-2 \al + 2 + \dl} \sim 1$, we can repeat \eqref{gwpR-3} to obtain
\begin{align*}
\| u^\ld (k + 1) \|_{H^s}^2 &\leq M^4 (I u^\ld (k + 1)) + \mathcal{O} (\eta_0^4) \\
&\leq M^4 (I u^\ld (0)) + (k + 1) C_1 N^{-2 \al + 2 + \dl} \eta_0^6 + \mathcal{O} (\eta_0^4) \\
&\leq \eta_0^2 + (k + 1) C_1 N^{-2 \al + 2 + \dl} \eta_0^6 + \mathcal{O} (\eta_0^4) \\
&\leq 4 \eta_0^2.
\end{align*}

\noi
By taking $k \sim N^{2 \al - 2 -\dl}$, we extend the solution $u^\ld$ to the time interval $[0, N^{2 \al - 2 -\dl}]$. It remains to take
\begin{align*}
N^{2 \al - 2 -\dl} \sim \ld^\al T \sim N^{\frac{2\al s}{1 - \al - 2s}} T,
\end{align*}

\noi
which is equivalent to
\begin{align}
N^{2\al - 2 - \dl + \frac{2 \al s}{\al - 1 + 2s}} \sim T.
\label{NT}
\end{align}

\noi
This can be achieved by choosing $N \gg 1$ large enough, since $\dl > 0$ can be arbitrarily small and a direct computation yields $2 \al - 2 + \frac{2 \al s}{\al - 1 + 2s} > 0$, given $\frac{2 - \al}{4} \leq s < 0$ and $\al > 2$. Thus, we finish the proof of Theorem \ref{TM:GWP1}.
\end{proof}

\begin{remark} \rm
\label{RMK:sob}
From the proof of Theorem \ref{TM:GWP1}, we can establish a growth bound of the Sobolev norm of the solution $u$ to the cubic fNLS \eqref{fnls1}. By rescaling, for $T > 0$, we have
\begin{align}
\sup_{t \in [0, T]} \| u (t) \|_{H^s} \sim \ld^{\frac{\al - 1}{2} + s} \sup_{t \in [0, \ld^\al T]} \| u^\ld (t) \|_{H^s} \leq \ld^{\frac{\al - 1}{2} + s} \sup_{t \in [0, \ld^\al T]} \| I u^\ld (t) \|_{L^2}
\label{sob1}
\end{align}

\noi
and
\begin{align}
\| I u_0^\ld \|_{L^2} \les N^{-s} \| u_0^\ld \|_{H^s} \sim N^{-s} \ld^{\frac{1 - \al}{2} - s} \| u_0 \|_{H^s}.
\label{sob2}
\end{align}

\noi
From the proof of Theorem \ref{TM:GWP1} above, we have
\begin{align}
\sup_{t \in [0, \ld^\al T]} \| I u^\ld (t) \|_{L^2} \les \| I u_0^\ld \|_{L^2}.
\label{sob3}
\end{align}

\noi
Thus, combining \eqref{sob1}, \eqref{sob2}, and \eqref{sob3}, we obtain
\begin{align*}
\sup_{t \in [0, T]} \| u (t) \|_{H^s} \les N^{-s} \| u_0 \|_{H^s}. 
\end{align*}

\noi
By \eqref{NT}, we obtain
\begin{align*}
\sup_{t \in [0, T]} \| u (t) \|_{H^s} \les T^\be \| u_0 \|_{H^s},
\end{align*}

\noi
where
\begin{align*}
\be = \frac{-s (\al - 1 + 2s)}{(2 \al - 2 - \dl) (\al - 1 + 2s) + 2 \al s}
\end{align*}

\noi
for $\dl > 0$ arbitrarily small. In particular, when $\al = 4$ and $s = -\frac 12$, we have $\be = \frac{1}{8 - 2\dl}$, which is better than the growth of the Sobolev norm established in \cite{kihoon}.
\end{remark}

\section{Well-posedness on the circle}
\label{SEC:WPT}
In this section, we study low regularity well-posedness of the cubic fNLS \eqref{fnls1} on the circle (i.e.~$ \M = \T $).
We recall from Subsection \ref{SUBSEC:introT} that we mainly work with well-posedness of the equation \eqref{fnls3} for $w$.

\subsection{Strichartz estimates on the circle}
\label{SEC:lines2}
In this subsection, we establish some Strichartz estimates, which will be useful for proving well-posedness of the cubic fNLS \eqref{fnls3}. We begin by showing the following $L^4$-Strichartz estimates.

\begin{lemma}
\label{LEM:L4str}
Let $\al > 2$ and $b \geq \frac{\al + 1}{4\al}$. Then, we have
\begin{align}
\| u \|_{L^4_t L_x^4 (\R \times \T)} \les \| u \|_{X^{0, b}}.
\label{L4X}
\end{align}

\noi
Furthermore, by letting $u_0 \in H^s (\T)$ with $s > \frac{1 - \al}{2}$, we have
\begin{align}
\| u \|_{L^4_t L_x^4 (\R \times \T)} \les \| u \|_{Y^{0, b}}.
\label{L4Y}
\end{align}
\end{lemma}

\begin{proof}
Let us only consider the case $b \leq \frac 12$, since the case $b > \frac 12$ will then follow directly. We first consider \eqref{L4X}. By following the same steps as in \cite[Lemma 2.1]{TT} or \cite[Lemma 2.2]{kwak}, it suffices to show the following bound
\begin{align}
\sum_{\substack{n_1 \geq 0 \\ n - n_1 \geq 0}} \jb{\tau - |n_1|^\al + |n - n_1|^\al}^{1 - 4b} < C
\label{sum_n1}
\end{align}

\noi
for some constant $C > 0$ uniformly in $\tau \in \R$ and $n \in \Z$. Without loss of generality, we can assume that $n_1 \geq n - n_1$. 

Let us define the function $g_1$ on $\R$ by
\begin{align*}
g_1 (\xi) = \tau - |\xi|^\al + |n - \xi|^\al.
\end{align*}

\noi
Note that we have $|g_1'(\xi)| \ges n^{\al - 1}$ if $\xi \geq n - \xi \geq 0$. By Lemma \ref{LEM:count}, for any $N > 0$, we have
\begin{align}
\# \{ n_1 \in \Z_+ \cup \{0\} : n_1 \geq n - n_1 \geq 0, |g_1 (n_1)| \leq N \} \les \frac{N}{n^{\al - 1}} + 1.
\label{count_n1}
\end{align}

Let us go back to \eqref{sum_n1}. Let $K \in \N$ be chosen later. By \eqref{count_n1} and the fact that $\frac 14 < b \leq \frac 12$, we can compute that
\begin{align*}
&\sum_{n_1 \geq n - n_1 \geq 0} \jb{\tau - |n_1|^\al + |n - n_1|^\al}^{1 - 4b} \\
&= \sum_{\substack{n_1 \geq n - n_1 \geq 0 \\ |g_1 (n_1)| \leq 2}} \jb{g_1 (n_1)}^{1 - 4b} + \sum_{k = 1}^K \sum_{\substack{n_1 \geq n - n_1 \geq 0 \\ 2^k < |g_1 (n_1)| \leq 2^{k + 1}}} \jb{g_1 (n_1)}^{1 - 4b} + \sum_{\substack{n_1 \geq n - n_1 \geq 0 \\ |g_1 (n_1)| > 2^{K + 1}}} \jb{g_1 (n_1)}^{1 - 4b} \\
&\les \sum_{k = 0}^K \Big( \frac{2^{k + 1}}{n^{\al - 1}} + 1 \Big) 2^{k (1 - 4b)} + n 2^{K (1 - 4b)} \\
&\les \frac{2^{2K (1 - 2b)}}{n^{\al - 1}} + 1 + n 2^{K (1 - 4b)}.
\end{align*}

\noi
We now choose $K$ such that
\begin{align*}
\frac{2^{2K (1 - 2b)}}{n^{\al - 1}} \sim 1,
\end{align*}

\noi
which gives $2^K \sim n^{\frac{\al - 1}{2 - 4b}}$. Thus, we obtain
\begin{align*}
\sum_{n_1 \geq n - n_1 \geq 0} \jb{\tau - |n_1|^\al + |n - n_1|^\al}^{1 - 4b} \les 1 + n^{1 + \frac{(\al - 1)(1 - 4b)}{2 - 4b}},
\end{align*}

\noi
which is bounded by an absolute constant if $1 + \frac{(\al - 1)(1 - 4b)}{2 - 4b} \leq 0$ or, equivalently, $b \geq \frac{\al + 1}{4\al}$. This implies \eqref{L4X}.

For \eqref{L4Y}, instead of \eqref{sum_n1}, we need to show
\begin{align*}
\sum_{\substack{n_1 \geq 0 \\ n - n_1 \geq 0}} \jb{\tau - \mu (n_1) + \mu (n - n_1)}^{1 - 4b} < C,
\end{align*}

\noi
where we recall that $\mu (n) = |n|^\al - |\ft{u_0} (n)|^2$. Here, we allow $C$ to depend on $\| u_0 \|_{H^s}$ but independent of $\tau$ and $n$. Thus, we only need to consider the case when $n \gg \| u_0 \|_{H^s}^{\frac{2}{\al - 1 + 2s}}$. Note that
\begin{align*}
|\ft{u_0} (n)|^2 \leq \jb{n}^{-2s} \| u_0 \|_{H^s (\T)}^2,
\end{align*}

\noi
so that we can construct a $C^1$ function $h$ on $\R$ in such a way that $h(n) = |\ft{u_0} (n)|^2$ for all $n \in \Z$ and $h'(\xi) \leq \jb{\xi}^{-2s} \| u_0 \|_{H^s (\T)}^2$ for all $\xi \in \R$. Thus, since $-2s < \al - 1$ and $n \gg \| u_0 \|_{H^s (\T)}^{\frac{2}{\al - 1 + 2s}}$, by defining
\begin{align*}
g_2 (\xi) = \tau - \mu(\xi) + \mu (n - \xi),
\end{align*}

\noi
we have $|g_2' (\xi)| \ges n^{\al - 1}$ if $\xi \geq n - \xi \geq 0$. The rest of the steps follow from exactly the same way as the steps for showing \eqref{sum_n1}.
\end{proof}

We now show the following $L^6$-Strichartz estimate.
\begin{lemma}
\label{LEM:L6str}
Let $\al > 2$, $\eps > 0$, and $I \subset \R$ be an interval of unit length. Then, we have
\begin{align}
\| u \|_{L_I^6 L_x^6} \les \| u \|_{X_I^{\eps, \frac 12 - \eta}}
\label{L6}
\end{align}

\noi
for some $\eta > 0$ sufficiently small.
\end{lemma}

\begin{proof}
We use the following $L^6$-Strichartz estimate from \cite[Proposition 1.1]{Schi}:
\begin{align*}
\big\| e^{-it D^\al} f \big\|_{L_I^6 L_x^6} \les \| f \|_{H_x^\eps},
\end{align*}

\noi
which was proved by the $\l^2$ decoupling theorem by Bourgain-Demeter \cite{BD}.
By using the transference principle in \cite[Lemma 2.9]{Tao}, we obtain
\begin{align}
\| u \|_{L_I^6 L_x^6} \les \| u \|_{X_I^{\eps, \frac 12 + \dl}}
\label{L6-1}
\end{align}

\noi
for any $\dl > 0$. Also, by Sobolev's embedding and Lemma \ref{LEM:emb}, we have
\begin{align}
\| u \|_{L_I^6 L_x^6} \les \| u \|_{X_I^{\frac 13, \frac 13}}.
\label{L6-2}
\end{align}

\noi
Thus, by interpolating \eqref{L6-1} and \eqref{L6-2}, we obtain the desired estimate \eqref{L6}.
\end{proof}

\subsection{Non-resonant and resonant estimates}

In this section, we establish non-resonant and resonant estimates that are necessary for proving Theorem \ref{TM:LWP2}, local well-posedness of the cubic fNLS \eqref{fnls3}. By writing \eqref{fnls3} in terms of Fourier coefficients, we have
\begin{align}
\label{fNLS6}
\begin{split}
\dt\ft{w}(n) + i|n|^{\al} \ft{w}(n) &= -i \sum_{\G(n)}  e^{it\Psi(\cj{n})}  \ft{w}(n_1)\cj{\ft{w}}(n_2) \ft{w}(n_3) + i  \big( |\ft{w}(n)|^2 - |\ft{u}_0(n)|^2 \big) \ft{w}(n)  \\
&= -i\ft{\NN_2(w)}(n) + i\ft{\RR_2(w)}(n),
\end{split}
\end{align}
for all $  n\in \Z$.

We first show the following trilinear estimate for the non-resonant term $\NN_2 (w)$, which is a slight generalization of \cite[Proposition 3.1]{kwak}.

\begin{proposition}
\label{Prop:triR} 
Let $\al > 2$, $\max \{\frac{2 - \al}{4}, \frac{1 - \al}{6}\} < s < 0$, $T > 0$, and $u_0 \in H^s (\T)$.
Then, for $0 < \eps \ll 1$ small enough, we have
\begin{align}
\| {\NN_2(w_1, w_2, w_3)} \|_{X_T^{s,-\frac12  + 2\eps}} \les \prod_{j = 1}^3 \| w_j \|_{X_T^{s,\frac12 +\eps}}.	
\label{triT}
\end{align}
\end{proposition}



\begin{proof}
By working with extensions of $u_1$, $u_2$, and $u_3$ on the whole time line, we can ignore the subscript $T$ for the $X^{s, b}$-norms. By duality, we have
\begin{align*}
\| \NN_2 &(w_1, w_2, w_3) \|_{X^{s, -\frac 12 + 2 \eps}} \\
&= \sup_{\| h \|_{X^{s, \frac 12 - 2 \eps}} \leq 1} \bigg| \sum_{n \in \Z} \jb{n}^{2s} \int_\R \sum_{\G (n)} e^{it \Psi (\cj{n})} \ft{w_1} (t, n_1) \cj{\ft{w_2}} (t, n_2) \ft{w_3} (t, n_3) \cj{\ft{h}} (t, n) dt \bigg|
\end{align*}

\noi
By letting $\ft{v_j} (t, n_j) = e^{it |\ft{u_0} (n_j)|^2} \ft{w_j} (t, n_j)$, $j = 1, 2, 3$, we have
\begin{align*}
\sum_{\G(n)}  e^{it\Psi(\bar n)}  \ft{w_1}(t, n_1)\cj{\ft{w_2}}(t, n_2) \ft{w_3}(t, n_3)=
\sum_{\G(n)}  e^{-it|\ft{u_0}(n)|^2}  \ft{v_1}(t, n_1)\cj{\ft{v_2}}(t, n_2) \ft{v_3}(t, n_3).
\end{align*}

\noi
Hence, to prove \eqref{triT}, we only need to show
\begin{equation}
\label{eq:trilinear}
\bigg| \sum_{n \in \Z} \jb{n}^{2s} \int_\R  \sum_{\G(n)}  \ft{v_1}(t, n_1) \cj{\ft{v_2}}(t, n_2) \ft{v_3}(t, n_3)
\cj{\ft{v_4}} (t, n) dt  \bigg|
\les  \prod_{j = 1}^3 \| v_j \|_{Y^{s,\frac12 + \eps}} \cdot \| v_4 \|_{Y^{s,\frac12 - 2\eps}}.
\end{equation}

\noi
where $\ft{v_4} (t, n) := e^{it|\ft{u_0}(n)|^2} \ft{h}(t, n) $. By Plancherel's theorem in time,  \eqref{eq:trilinear} becomes
\begin{align}
\label{tri_goal}
\begin{split}
\bigg| \sum_{n \in \Z} &\jb{n}^{2s} \int_{\tau_1 - \tau_2 + \tau_3 = \tau} \sum_{\G(n)}  \ft{v_1}(\tau_1, n_1) \cj{\ft{v_2}}(\tau_2, n_2) \ft{v_3}(\tau_3, n_3)
\cj{\ft{v_4}} (\tau, n) d \tau_1 d \tau_2 d \tau_3 d \tau \bigg| \\
&\les \prod_{j = 1}^3 \| v_j \|_{Y^{s,\frac12 + \eps}} \cdot \| v_4 \|_{Y^{s,\frac12 - 2\eps}}.
\end{split}
\end{align}

We define $G$ to be
\begin{align*}
G &:= (\tau_1 + |n_1|^\al ) - (\tau_2 + |n_2|^\al) +(\tau_3 + |n_3|^\al)-(\tau + |n|^\al) \\
&\hphantom{:}\quad +( |\ft{u_0}(n_1)|^2 - |\ft{u_0}(n_2)|^2 + |\ft{u_0}(n_3)|^2 - |\ft{u_0}(n)|^2 ).
\end{align*}

\noi
Since $u_{0} \in H^{s}$ for $\max \{ \frac{2 - \al}{4}, \frac{1 - \al}{6} \} < s < 0$, we know that $|\ft{u_{0}}(n)|^{2} \les \jb{n}^{-s} \ll \jb{\nm}^{\al - 2}$, where $\nm = \max\{|n_1|,|n_2|,|n_3|,|n|\}$ is large enough. Thus, by Lemma \ref{LEM:factor}, we have
\[
 |G| \ges |n_1-n_2||n_2-n_3||\nm|^{\al-2}
\]

\noi
under the conditions $\tau_1 - \tau_2 + \tau_3 = \tau$, $n_1 - n_2 + n_3 = n$, $n \neq n_1$, and $n \neq n_3$.
Recalling that $\mu (n) = |n|^\al - |\ft{u_0} (n)|^2$, we deduce that
\begin{align*}
\max\{|\tau + \mu(n)|, |\tau_1 + \mu(n_1)|, |\tau_2 + \mu(n_2)|, |\tau_3 + \mu (n_3)| \} \ges |n_1-n_2||n_2-n_3||\nm|^{\al-2}.
\end{align*}

\noi
We can assume that $|\tau + \mu (n)|$ is the maximum modulation since other cases are similar (and slightly easier), so that we have
\begin{align}
|\tau + \mu (n)| \ges |n_1-n_2||n_2-n_3||\nm|^{\al-2}.
\label{mod_bdd}
\end{align}

\noi
We define
\begin{align*}
\ft{f_1} (\tau_1,n_1) &= 
\jb{n_1}^s|\ft{v_1}(\tau_1,n_1)|, \\
\ft{f_2} (\tau_2, n_2) &= \jb{n_2}^s|\ft{v_2}(\tau_2, n_2)|, \\
\ft{f_3} (\tau_3,n_3) &= 
\jb{n_3}^s|\ft{v_3}(\tau_3,n_3)|, \\
\ft{g}(\tau,n) &= \jb{\tau +
 \mu(n)  }^{\frac 12 - 2\eps}\jb{n}^s|\ft{v_4}(\tau,n)|,
\end{align*}

\noi
so that by Plancherel's theorem in time and \eqref{mod_bdd}, the left-hand-side of \eqref{tri_goal} is bounded by
\begin{equation}
\label{eq:tri1}
\int_\R \sum_{n,\G(n)} m(\cj{n})  \big| \ft{f_1} (t, n_1) \cj{\ft{f_2}} (t, n_2) \ft{f_3} (t, n_3) \cj{\ft{g}}(t, n) \big| dt,
\end{equation}

\noi
where the multiplier $m (\cj n)$ is given by
\[
m(\cj n) = m(n_1,n_2,n_3,n):=
\frac{\jb{n}^{s}\jb{n_1}^{-s}\jb{n_2}^{-s}\jb{n_3}^{-s}}{(|n_1-n_2||n_2-n_3||\nm|^{\al-2})^{\frac 12 - 2 \eps}}.
\]

By symmetry between $n_1$ and $n_3$ and the fact that two largest frequencies are comparable, it suffices to consider the following 4 cases.

\medskip
\noindent
\textbf{Case 1:} $\jb{n} \sim |\nm|$.

In this case, we have $\jb{n}^s \jb{n_j}^{-s} \les 1$ for some $j \in \{1, 2, 3\}$. Given $\eps > 0$ sufficiently small and $s > \frac{2 - \al}{4}$, we have
\begin{align*}
m(\cj n) &\les \frac{|\nm|^{-2s}}{(|n_1-n_2||n_2-n_3|)^{\frac{1}{2} - 2 \eps}  |\nm|^{\frac{\al -2}{2} -}  } \les \frac{1}{|n_1-n_2|^{\frac{\al - 1}{2} + 2s  -}}.
\end{align*}

\noi
Then, from the Cauchy–Schwarz inequality, H\"older's inequality, Plancherel's theorem, Hausdorff-Young's inequality, and the $L^4$-Strichartz estimate (Lemma \ref{LEM:L4str}), we have
\begin{align*}
|\eqref{eq:tri1}| &\les \int_\R \sum_{n,\G(n)} \frac{1}{|n_1-n_2|^{\frac{\al-1}{2}+2s-}} \big| \ft{f_1} (t, n_1) \cj{\ft{f_2}} (t, n_2) \ft{f_3} (t, n_3) \cj{\ft{g}} (t, n) \big| dt  \\
&= \int_\R \sum_{n' \neq 0}  \frac{1}{|n'|^{\frac{\al-1}{2}+2s-}} \Big| \ft{f_1 \cj{f_2} }(t, n') \ft{ f_3 \cj{g} }(t, -n') \Big| dt \\
& \les \int_\R \bigg( \sum_{n' \neq 0}  \Big| \ft{f_1 \cj{f_2} }(t, n') \ft{ f_3 \cj{g} }(t, -n') \Big|^{2} \bigg)^{1/2} dt \\
& \leq \int_\R \Big\| \ft{f_1 \cj{f_2} } (t, n) \Big\|_{\l_{n}^{2}} \big\| \ft{ f_3 \cj{g} } (t, n) \big\|_{\l_{n}^{\infty}} dt \\
& \leq \int_\R \|f_1 \cj{f_2} \|_{L_{x}^{2}} \| f_3 \cj{g} \|_{L_{x}^{1}} dt \\
&\leq \norm{f_1}_{L_{t,x}^4}\norm{f_2}_{L_{t,x}^4}\norm{f_3}_{L_t^{\infty}L_x^2}\norm{g}_{L_{t,x}^2}  \notag\\
&\les \| f_1 \|_{Y^{0, \frac12 + \eps}} \| f_2 \|_{Y^{0, \frac12 + \eps}} 
\| f_3 \|_{Y^{0, \frac12 + \eps}}
\| g \|_{L_{t,x}^2} \\
&= \| v_1 \|_{Y^{s, \frac 12 + \eps}} \| v_2 \|_{Y^{s, \frac 12 + \eps}} \| v_3 \|_{Y^{s, \frac 12 + \eps}} \| v_4 \|_{Y^{s, \frac 12 - 2 \eps}},
\end{align*}

\noi
where we used $\al - 1 + 4s > 1$ given $s>\frac{2-\al}{4}$. This is the desired estimate.

\medskip
\noindent
\textbf{Case 2:} $\jb{n_1} \sim \jb{n_2} \sim \jb{n_3} \gg \jb{n}$.

In this case, we have $|n_2 - n_3| \sim |n_1| \sim |\nm|$. Thus, given $\eps > 0$ sufficiently small and $s > \frac{1 - \al}{6}$, we have
\begin{align*}
m(\cj{n}) \les \frac{|\nm|^{-3s}}{|n_1 - n_2|^{\frac 12 - 2 \eps} |\nm|^{\frac{\al - 1}{2} -}} \les \frac{1}{|n_1 - n_2|^{\frac{\al}{2} + 3s -}}.
\end{align*}

\noi
We can then proceed as in Case 1 to obtain the desired estimate, where we need to use $\al + 6s > 1$ given $s > \frac{1 - \al}{6}$.

\medskip
\noindent
\textbf{Case 3:} $\jb{n_1} \sim \jb{n_2} \gg \jb{n}, \jb{n_3}$.

In this case, we can use exactly the same steps as in Case 2 to obtain the desired estimate.

\medskip
\noindent
\textbf{Case 4:} $\jb{n_1} \sim \jb{n_3} \gg \jb{n}, \jb{n_2}$.

In this case, we can use exactly the same steps as in Case 2 to obtain the desired estimate.

\end{proof}

We now consider the resonant term $\RR_2(w)$.
Note that for any $T > 0$ and $s \in \R$, we have
\begin{equation}
\label{eqI}
\norm{\RR_2(w)}_{L_T^2 H_x^s} \les \Big( \sup_{\substack{|t| \leq T \\ n \in \Z}}\big| |\ft{w}(t, n)|^2 - |\ft{u_0} (n)|^2 \big| \Big) \norm{w}_{L^2_T H_x^s}.
\end{equation}

\noi
If $w$ is a smooth solution to \eqref{fnls3}, then by \eqref{fNLS6}, we have
\begin{align}
\begin{split}
|\ft {w}(t, n)|^2 - |\ft {u_0}(n)|^2 & = \int_0^t \frac{d}{dt'} |\ft{w}(t', n)|^2 dt' \\
& = 2\Im \int_0^t \sum_{\G(n)} e^{it'\Psi(\bar n)} \ft{w}(t', n_1) \cj{\ft {w}} (t', n_2) \ft {w} (t', n_3) 
\cj{\ft {w}} (t', n) dt'
\end{split}
\label{EnN}
\end{align}

\noi
for any $t \in \R$ and $n \in \Z$. We now focus on the right-hand-side of \eqref{EnN} and establish the following estimate.

\begin{proposition}
\label{Prop:quinc}
Let $\al>2$, $\frac{2 - \al}{6} \leq s < 0$, $0 < T \leq 1$, and $u_0 \in C^{\infty}(\T)$. 
Suppose that $w$ is a global smooth solution to \eqref{fnls3} with $w(0)=u_0$. Then,
for $0 < \eps \ll 1$ small enough, we have
\begin{align}
\label{quin1}
\begin{split}
\sup_{\substack{|t| \leq T \\ n \in \Z}} \bigg| &\int_0^t  \sum_{\G(n)} 
e^{i\Psi(\cj n)}
\ft{w}(t', n_1) \cj{\ft{w}} (t', n_2) \ft{w}(t', n_3)\cj{\ft{w}}(t', n) dt' \bigg| \\
&\les 
\|u_0\|^4_{H^s} + \norm{w}_{X_T^{s,\frac12+\eps}}^4 + \norm{w}_{X_T^{s,\frac12+\eps}}^6
\end{split}
\end{align}

\end{proposition}

\begin{proof}
We fix $t$ such that $|t| \leq T$. 
Recalling the definition of $\Psi$ in \eqref{Psi}. Given $n_j, n_{j1}, n_{j2}, n_{j3} \in \Z$ for $j \in \{1, 2, 3\}$, we define
\begin{align*}
\Psi(\cj n_j) := \Psi (n_{j, 1}, n_{j, 2}, n_{j, 3}, n_j) = |\ft{u_0} (n_{j, 1})|^2 - |\ft{u_0} (n_{j, 2})|^2 + |\ft{u_0} (n_{j, 3})|^2 - |\ft{u_0} (n_j)|^2.
\end{align*}

\noi
We also define
$\mu^{+}(n) := |n|^\al + |\ft{u_0} (n)|^2$. Furthermore, for $n_1, n_2, n_3, n \in \Z$, we define
\begin{align*}
\phi(\cj n) := \phi (n_1, n_2, n_3, n) = |n_1|^\al - |n_2|^\al + |n_3|^\al - |n|^\al.
\end{align*}

By Fubini's theorem and the integration by parts, we obtain
\begin{align}
\label{fub_ibp}
\begin{split}
\bigg |\int_0^t  &\sum_{\G(n)}
e^{it'\Psi(\cj n)}
 \ft{w}(t', n_1) \cj{\ft{w}}(t', n_2) \ft{w}(t', n_3)\cj{\ft{w}}(t', n) dt' \bigg| \\
&= \bigg | \sum_{\G(n)} \int_0^t
e^{it'\Psi(\cj n)}
 \ft{w}(t', n_1) \cj{\ft{w}}(t', n_2) \ft{w}(t', n_3)\cj{\ft{w}}(t', n) dt' \bigg| \\
&= \bigg| \sum_{\G(n)} \int_0^t e^{- it' \phi(\cj n) } 
\big(e^{it'\mu^+(n_1)} \ft{w}(t', n_1)\big) \\
&\quad \times \big(\cj{e^{it'\mu^+(n_2)} \ft{w}(t', n_2)} \big) 
  \big(e^{it'\mu^+(n_3)} \ft{w}(t', n_3) \big)   \big(\cj{e^{it'\mu^+(n) } \ft{w}(t', n)} \big) dt'   \bigg| \\
&\leq \bigg| \sum_{\G(n)} \frac{e^{it\Psi(\cj n)}}{\phi(\cj n)} \big( \ft{w}(t,n_1) \cj{\ft{w}}(t,n_2) \ft{w}(t,n_3) \cj{\ft{w}}(t,n) - \ft{u_0} (n_1)\cj{\ft{u_0}} (n_2) \ft{u_0}(n_3) \cj{\ft{u_0}}(n)  \big) \bigg| \\
&\quad + \bigg| \sum_{\G(n)} \int_0^t \frac{e^{- it' \phi(\cj n)}}{\phi(\cj n)}
 \cdot \frac{d}{dt'}\bigg( \big(e^{it' \mu^+ (n_1)} \ft{w}(t', n_1)\big)
\big(\cj{e^{it' \mu^+ (n_2)}  \ft{w}(t', n_2)} \big)\\
&\quad \times \big(e^{it' \mu^+ (n_3) } \ft{w}(t', n_3) \big)
\big(\cj{e^{it' \mu^+ (n)} \ft{w}(t', n)} \big) \bigg) dt' \bigg| \\
&=: \1+\II
\end{split}
\end{align}

We first consider $\1$. 
For convenience, we define $\ft f (m) = \jb{m}^s | \ft{w} (t, m)|$ for any $m \in \Z$. We also let $\nm = \max \{ |n_1|, |n_2|, |n_3|, |n| \}$. 
Thus, by Lemma \ref{LEM:factor}, H\"older's inequality, Hausdorff-Young's inequalities, and the Cauchy-Schwarz inequality, we obtain
\begin{align}
\label{4f}
\begin{split}
\bigg| &\sum_{\G(n)} \frac{e^{it\Psi(\cj n)}}{\phi(\cj n)} \ft{w}(t,n_1) \cj{\ft{w}}(t,n_2) \ft{w}(t,n_3) \cj{\ft{w}}(t,n) \bigg| \\
&\les \sum_{n, \G(n)} \frac{\jb{n_1}^{-s} \jb{n_2}^{-s} \jb{n_3}^{-s} \jb{n}^{-s} }{|n_1 - n_2| |\nm|^{\al - 2}} \ft{f} (n_1) \cj{\ft{f}} (n_2) \ft{f} (n_3) \cj{\ft{f}} (n) \\
&\les \sum_{n, \G(n)} \frac{1}{|n_1 - n_2| |\nm|^{\al - 2 + 4s}} \ft{f} (n_1) \cj{\ft{f}} (n_2) \ft{f} (n_3) \cj{\ft{f}} (n) \\
&\les \sum_{n' \neq 0} \frac{1}{|n'|^{\al - 1 + 4s}} \ft{f \cj{f}} (n') \ft{f \cj{f}} (-n') \\
&\les \Big\| \ft{f \cj{f}} \Big\|_{\l^\infty} \big\| \ft{f \cj{f}} \big\|_{\l^\infty} \\
&\les \| f \|_{L^2}^4 \les \| w(t) \|_{H^s}^4,
\end{split}
\end{align}

\noi
where we used $\al - 2 + 4s > 0$ and $\al - 1 + 4s > 1$ given $s \geq \frac{2 - \al}{6}$. Thus, using similar steps, we obtain
\begin{align*}
\1 \les \| u_0 \|_{H^s}^4 + \| w (t) \|_{H^s}^4 \les \| u_0 \|_{H^s}^4 + \| w \|_{X^{s, \frac 12 + \eps}}^4
\end{align*}

\noi
for any $\eps > 0$.

We now consider $\II$.
By using \eqref{fNLS6}, we have
\begin{align}
\label{ob1} 
\begin{split}
\frac{d}{dt} ( e^{it \mu^{+}(n)}\ft{w}(t, n)  ) 
 &= e^{it \mu^{+}(n) } ( \dt  \ft{w}(t, n) + i\mu^{+}(n) \ft{w}(t, n) )  \\
&= e^{it\mu^{+}(n)} \bigg( i|\ft{w}(t, n)|^2\ft{w}(t, n) - i\sum_{\G(n)} e^{it\Psi(\cj n)}  \ft{w}(t, n_1)\cj{\ft{w} (t, n_2)} \ft{w}(t, n_3) \bigg),
\end{split}
\end{align}

\noi
By \eqref{ob1} and Plancherel's theorem, we have
\begin{align*}
\II &\les  \bigg| \sum_{\G (n)} \int_0^t  \frac{1}{\phi (\cj{n})} \bigg[ |\ft{w} (t', n_1)|^2 \big( e^{it' |\ft{u_0} (n_1)|^2} \ft{w} (t', n_1) \big) \\
&\quad - \bigg( \sum_{\G (n_1)} e^{i t' \Psi (\cj{n_1}) + it' |\ft{u_0} (n_1)|^2} \ft{w} (t', n_{1, 1}) \cj{\ft{w}} (t', n_{1, 2}) \ft{w} (t', n_{1, 3}) \bigg) \bigg] \\
&\quad \times \big(\cj{e^{it' |\ft{u_0} (n_2)|^2}  \ft{w}(t', n_2)} \big) \big(e^{it' |\ft{u_0} (n_3)|^2 } \ft{w}(t', n_3) \big)
\big(\cj{e^{it' |\ft{u_0} (n)|^2} \ft{w}(t', n)} \big)  dt' \bigg| \\
&=  \bigg| \sum_{\G (n)} \int_{\G' (\tau)} \frac{1}{\phi (\cj{n})} \bigg( \ft{w} (\tau_{1, 1}, n_1) \cj{\ft{w}} (\tau_{1, 2}, n_1) \ft{u} (\tau_{1, 3}, n_1) \\
&\quad - \sum_{\G (n_1)} \ft u (\tau_{1, 1}, n_{1, 1}) \cj{\ft{u}} (\tau_{1, 2}, n_{1, 2}) \ft u (\tau_{1, 3}, n_{1, 3}) \bigg) \cj{\ft{u}} (\tau_2, n_2) \ft{u} (\tau_3, n_3) \cj{\ft{u \ind_{[0,t]}}} (\tau, n) d\tau_{11} \cdots d\tau \bigg| \\
&=: \int_{\G' (\tau)} \big| \II_1 + \II_2 \big| d\tau_{1, 1} \cdots d\tau 
\end{align*}

\noi
where $\G' (\tau)$ refers to the condition $\tau = \tau_{1, 1} - \tau_{1, 2} + \tau_{1, 3} - \tau_2 + \tau_3$.

For the term $\II_1$, we define
\begin{align*}
\ft h (\tau, n) &= \jb{n}^s |\ft w (\tau, n)|, \\
\ft g (\tau, n) &= \jb{n}^s |\ft u (\tau, n)|, \\
\ft{g'} (\tau, n) &= \jb{n}^s \big| \ft{u \ind_{[0, t]}} (\tau, n) \big|,
\end{align*}

\noi
for any $\tau \in \R$ and $n \in \Z$. Given $s \geq  \frac{2 - \al}{6}$, by Lemma \ref{LEM:factor}, H\"older's inequality, and similar steps as in \eqref{4f}, we obtain
\begin{align*}
|\II_1| &\les \sum_{n, \G(n)} \frac{\jb{n_1}^{-3 s} \jb{n_2}^{-s} \jb{n_3}^{-s} \jb{n}^{-s}}{|n_1 - n_2| |\nm|^{\al - 2}} \ft{h} (\tau_{1, 1}, n_1) \ft{h} (\tau_{1, 2}, n_1) \ft{g} (\tau_{1, 3}, n_1) \ft{g} (\tau_2, n_2) \ft{g} (\tau_3, n_3) \ft{g'} (\tau, n) \\
&\les \big( \sup_{n_1 \in \Z} \ft{h} (\tau_{1, 1}, n_1) \ft{h} (\tau_{1, 2}, n_1) \big) \sum_{n, \G(n)} \frac{1}{|n_1 - n_2|} \ft{g} (\tau_{1, 3}, n_1) \ft{g} (\tau_2, n_2) \ft{g} (\tau_3, n_3) \ft{g'} (\tau, n) \\
&\les \| \ft h (\tau_{1, 1}, n_1) \|_{\l_{n_1}^2} \| \ft h (\tau_{1, 2}, n_1) \|_{\l_{n_1}^2} \| \ft g (\tau_{1, 3}, n_1) \|_{\l^2_{n_1}} \| \ft g (\tau_{2}, n_2) \|_{\l^2_{n_2}} \| \ft g (\tau_{3}, n_3) \|_{\l^2_{n_3}} \big\| \ft{g'} (\tau, n) \big\|_{\l^2_{n}}
\end{align*}

\noi
Thus, by the Cauchy-Schwarz inequality in $\tau$, Young's convolution inequalities, H\"older's inequality in $\tau_3$, and Lemma \ref{LEM:time}, we get
\begin{align*}
&\int_{\G' (\tau)} |\II_1| d\tau_{1, 1} \cdots d\tau \\
&\leq \bigg\| \int_{\G' (\tau)} \| \ft h (\tau_{1, 1}) \|_{\l_{n_{1}}^2} \| \ft h (\tau_{1, 2}) \|_{\l_{n_{1}}^2} \| \ft g (\tau_{1, 3}) \|_{\l_{n_{1}}^2} \| \ft g (\tau_{2}) \|_{\l_{n_2}^2} \| \ft g (\tau_{3}) \|_{\l_{n_3}^2} d\tau_{1, 1} \cdots d\tau_3 \bigg\|_{L^2_\tau} \big\| \ft{g'} \big\|_{L^2_\tau \l_n^2} \\
&\les \| \ft h \|_{L^2_{\tau_{1, 1}} \l_{n_{1}}^2} \| \ft h \|_{L^2_{\tau_{1, 2}} \l_{n_{1}}^2} \| \ft g \|_{L^2_{\tau_{1, 3}} \l_{n_{1}}^2} \| \ft g \|_{L^2_{\tau_2} \l_{n_2}^2} \| \ft g \|_{\l_{n_3}^2 L^1_{\tau_3}} \big\| \ft{g'} \big\|_{L^2_\tau \l_n^2} \\
&\les \| w \|_{X^{s, \frac 12 + \eps}}^6,
\end{align*}

\noi
as desired.

For the term $\II_2$, we consider the following cases.

\medskip \noi
\textbf{Case 1:} $|n_{1, 1}|, |n_{1, 2}|, |n_{1, 3}| \les \nm$.

In this case, for $j \in \{(1,1), (1,2), (1,3), 2, 3\}$, we define
\begin{align*}
\ft{g} (\tau_j, n_j) &= \jb{n_j}^s | \ft{u} (\tau_j, n_j) |, \\
\ft{g'} (\tau, n) &= \jb{n}^s \big| \ft{u \ind_{[0,t]}} (\tau, n) \big|.
\end{align*}

\noi
Given $s \geq  \frac{2 - \al}{6}$, we can bound $|n|^{-s} \les |\nm|^{-s}$ and $|n_j|^{-s} \les |\nm|^{-s}$ for all $j$. Thus, by Lemma~\ref{LEM:factor}, the change of variables $n' = n_{1, 2} - n_{1, 3}$ and $n'' = n_3 - n = n_2 - n_1$, the Cauchy-Schwarz inequality in $n_{1, 1}$ and then in $n'$, H\"older's inequality and Hausdorff-Young's inequality in $n''$, Plancherel's theorem, and H\"older's inequalities again, we obtain
\begin{align*}
|\II_2| &\les \sum_{n, \G (n)} \sum_{\G (n_1)} \frac{\ft{g} (n_{1, 1}) \ft{g} (n_{1, 2}) \ft{g} (n_{1, 3}) \ft{g} (n_{2}) \ft{g} (n_{3}) \ft{g'} (n)}{|n_1 - n_2| |n_2 - n_3|}  \\
&\les \sum_{\substack{n_{1, 1}, n_{1, 2}, n_3, n', n'' \in \Z \\ |n''|, |n_{1, 1} - n' + n'' - n_3| \neq 0}} \frac{\ft g (n_{1, 1}) \ft g (n_{1, 2}) \ft g (n_{1, 2} - n') \ft g (n_{1, 1} - n' + n'') \ft g (n_3) \ft{g'} (n_3 - n'')}{ |n''| |n_{1, 1} - n' + n'' - n_3| } \\
&\les \sum_{\substack{n_3, n'' \in \Z \\ |n''| \neq 0}} \frac{1}{|n''|} \ft{g} (n_3) \ft{g'} (n_3 - n'') \sum_{n' \in \Z} \bigg( \sum_{\substack{n_{1, 1} \in \Z \\ |n_{1, 1} - n' + n'' - n_3| \neq 0}} \frac{\ft g (n_{1, 1})^2}{|n_{1, 1} - n' + n'' - n_3|^2} \bigg)^{1/2} \\
&\quad \times \bigg( \sum_{n_{1, 1} \in \Z} \big| \ft{\cj{g} g} (- n') \ft g (n_{1, 1} - n' + n'') \big|^2 \bigg)^{1/2} \\
&\leq \bigg( \sum_{n'' \in \Z} \bigg| \sum_{n_3 \in \Z} \ft{g} (n_3) \ft{g'} (n_3 - n'') \bigg|^4 \bigg)^{1/4} \bigg( \sum_{n_{1, 1} \in \Z} \ft g (n_{1, 1})^2 \bigg)^{1/2} \\
&\quad \times \bigg( \sum_{n' \in \Z} \big| \ft{\cj{g} g} (- n') \big|^2 \sum_{n_{1, 1} \in \Z} | \ft g (n_{1, 1} - n' + n'') |^2 \bigg)^{1/2} \\
&\les \| \cj{g} g' \|_{L_x^{\frac43}} \| g \|_{L^2_x}^2 \| \cj{g} g \|_{L^2_x} \les \| g \|_{L_x^2}^2 \| g \|_{L_x^4}^3 \| g' \|_{L_x^2}
\end{align*}

\noi
Thus, by Plancherel's theorem in time, H\"older's inequality, the $L^4$-Strichartz estimate (Lemma~\ref{LEM:L4str}), Lemma \ref{LEM:emb}, and Lemma \ref{LEM:time}, we obtain
\begin{align*}
\int_{\G' (\tau)} |\II_2| d\tau_{1, 1} \cdots d\tau &\leq \| g \|_{L_t^\infty L_x^2}^2 \| g \|_{L_t^4 L_x^4}^3 \| g' \|_{L_t^4 L_x^2} \\
&\les \| g \|_{Y^{0, \frac 12 + \eps}}^2 \| g \|_{Y^{0, \frac 12 + \eps}}^3 \| g' \|_{Y^{0, \frac 12 - \eps}} \\
&\les \| w \|_{X^{s, \frac 12 + \eps}}^6,
\end{align*}

\noi
which is the desired estimate.

\medskip \noi
\textbf{Case 2:} $|n_{1, 1}|, |n_{1, 3}| \gg \nm$.

In this case, we note that
\begin{align*}
\max \{ &|\tau_{1, 1} + \mu (n_{1, 1})|, |\tau_{1, 2} + \mu (n_{1, 2})|, |\tau_{1, 3} + \mu (n_{1, 3})|, |\tau_{2} + \mu (n_{2})|, |\tau_3 + \mu (n_3)|, |\tau + \mu (n)| \} \\
&\ges \big| \big( |n_{1, 1}|^\al - |n_{1, 2}|^\al + |n_{1, 3}|^\al - |n_1|^\al \big) + \big( |n_1|^\al - |n_2|^\al - |n_3|^\al - |n|^\al \big) \\
&\quad - |\ft{u_0} (n_{1, 1})|^2 + |\ft{u_0} (n_{1, 2})|^2 + |\ft{u_0} (n_{1, 3})|^2 + |\ft{u_0} (n_2)|^2 - |\ft{u_0} (n_3)|^2 + |\ft{u_0} (n)|^2 \big|.
\end{align*}

\noi
We can assume that $|\tau + \mu(n)|$ is the maximum modulation, since all other cases are similar and simpler. As in the proof of Proposition \ref{Prop:triR}, we can ignore the $|\ft{u_0}|^2$ terms above in our later analysis. By Lemma \ref{LEM:factor}, we have
\begin{align}
\big| |n_{1, 1}|^\al - |n_{1, 2}|^\al + |n_{1, 3}|^\al - |n_1|^\al \big| &\ges |n_{1, 1} - n_{1, 2}| |n_{1, 2} - n_{1, 3}| |\nm'|^{\al - 2}, \label{factor1} \\
\big| |n_1|^\al - |n_2|^\al - |n_3|^\al - |n|^\al \big| &\ges |n_1 - n_2| |n_2 - n_3| |\nm|^{\al - 2}. \label{factor2}
\end{align}

\noi
where $\nm' = \max \{ |n_{1, 1}|, |n_{1, 2}|, |n_{1, 3}|, |n_1| \}$. In this case, we have
\begin{align*}
|n_{1, 1} - n_{1, 2}| &= |n_{1, 3} - n_1| \sim \nm' \gg \nm, \\
|n_{1, 2} - n_{1, 3}| &= |n_{1, 1} - n_1| \sim \nm' \gg \nm,
\end{align*}

\noi
so that $\eqref{factor1} \gg \eqref{factor2}$. Thus, we can use the maximum modulation to gain extra smoothing, which means that
\begin{align}
|\tau + \mu(n)| \ges |n_{1, 1} - n_{1, 2}| |n_{1, 2} - n_{1, 3}| |\nm'|^{\al - 2} \ges |n_{1, 2} - n_{1, 3}|^{1 + 2 \eps} |\nm'|^{\al - 1 - 2\eps}.
\label{modu_bdd}
\end{align}

We now define
\begin{align*}
\ft{g} (\tau_j, n_j) &= \jb{n_j}^s | \ft{u} (\tau_j, n_j) |, \\
\ft{g'} (\tau, n) &= \jb{n}^s \jb{\tau + |n|^\al}^{\frac 12 - \eps} \big| \ft{u \ind_{[0,t]}} (\tau, n) \big|, \\
\end{align*}

\noi
where $j \in \{(1,1), (1,2), (1,3), 2, 3\}$ and $\eps > 0$ sufficiently small. Given $s \geq \frac{2 - \al}{6}$, we can use the bounds 
\begin{align*}
&|n_{1, 1}|^{-s} \les |\nm'|^{-s}, \quad |n_{1, 2}|^{-s} \les |\nm'|^{-s}, \quad |n_{1, 3}|^{-s} \les |\nm'|^{-s}, \\
&|n_2|^{-s} \les |\nm|^{-s}, \quad |n_3|^{-s} \les |\nm|^{-s}, \quad |n|^{-s} \les |\nm|^{-s}.
\end{align*}

\noi
By \eqref{modu_bdd}, the change of variables $n' = n_{1, 2} - n_{1, 3}$ and $n'' = n_3 - n = n_2 - n_1$, the Cauchy-Schwarz inequality in $n_{1, 1}$ and then in $n'$, H\"older's inequalities in $n''$ and $n'$, and Young's convolution inequality, we obtain
\begin{align*}
|\II_2| &\les \sum_{\substack{n_{1, 1}, n_{1, 2}, n_3, n', n'' \in \Z \\ |n'|, |n''|, |n_{1, 1} - n' + n'' - n_3| \neq 0}} \frac{\ft g (n_{1, 1}) \ft g (n_{1, 2}) \ft g (n_{1, 2} - n') \ft g (n_{1, 1} - n' + n'') \ft g (n_3) \ft{g'} (n_3 - n'')}{|n'|^{\frac 12 + \eps} |n''| |n_{1, 1} - n' + n'' - n_3| } \\
&\les \sum_{\substack{n_3, n'' \in \Z \\ |n''| \neq 0}} \frac{1}{|n''|} \ft{g} (n_3) \ft{g'} (n_3 - n'') \sum_{\substack{n' \in \Z \\ |n'| \neq 0}} \bigg( \sum_{\substack{n_{1, 1} \in \Z \\ |n_{1, 1} - n' + n'' - n_3| \neq 0}} \frac{\ft g (n_{1, 1})^2}{|n_{1, 1} - n' + n'' - n_3|^2} \bigg)^{1/2} \\
&\quad \times \frac{1}{|n'|^{\frac 12 + \eps}} \bigg( \sum_{n_{1, 1} \in \Z} \big| \ft{\cj{g} g} (- n') \ft g (n_{1, 1} - n' + n'') \big|^2 \bigg)^{1/2} \\
&\leq \sum_{\substack{n_3, n'' \in \Z \\ |n''| \neq 0}} \frac{1}{|n''|} \ft{g} (n_3) \ft{g'} (n_3 - n'') \bigg( \sum_{n_{1, 1} \in \Z} \ft g (n_{1, 1})^2 \bigg)^{1/2} \\
&\quad \times \bigg( \sum_{\substack{n' \in \Z \\ |n'| \neq 0}} \frac{1}{|n'|^{1 + 2 \eps}} \big| \ft{\cj{g} g} (- n') \big|^2 \sum_{n_{1, 1} \in \Z} | \ft g (n_{1, 1} - n' + n'') |^2 \bigg)^{1/2} \\
&\les \| \ft g \|_{\l^2} \| \ft g \|_{\l^2}^2 \big\| \ft{\cj{g} g} \big\|_{\l^\infty} \big\| \ft{g'} \big\|_{\l^2} \les \| \ft g \|_{\l^2} \| \ft g \|_{\l^2}^4 \big\| \ft{g'} \big\|_{\l^2}.
\end{align*}

\noi
Thus, by considering time frequencies using the Cauchy-Schwarz inequality in $\tau$, Young's convolution inequalities, H\"older's inequality in $\tau_{1, 1}, \tau_{1, 2}, \tau_{1, 3}, \tau_2$, and Lemma \ref{LEM:time}, we get
\begin{align*}
&\int_{\G' (\tau)} |\II_2| d\tau_{1, 1} \cdots d\tau \\
&\leq \bigg\| \int_{\G' (\tau)} \| \ft g (\tau_{1, 1}) \|_{\l_{n_{1, 1}}^2} \| \ft g (\tau_{1, 2}) \|_{\l_{n_{1, 2}}^2} \| \ft g (\tau_{1, 3}) \|_{\l_{n_{1, 3}}^2} \| \ft g (\tau_{2}) \|_{\l_{n_2}^2} \| \ft g (\tau_{3}) \|_{\l_{n_3}^2} d\tau_{1, 1} \cdots d\tau_3 \bigg\|_{L^2_\tau} \big\| \ft{g'} \big\|_{L^2_\tau \l_n^2} \\
&\les \| \ft g \|_{\l_{n_{1, 1}}^2 L^1_{\tau_{1, 1}}} \| \ft g \|_{\l_{n_{1, 2}}^2 L^1_{\tau_{1, 2}}} \| \ft g \|_{\l_{n_{1, 3}}^2 L^1_{\tau_{1, 3}}} \| \ft g \|_{\l_{n_2}^2 L^1_{\tau_2}} \| \ft g \|_{L^2_{\tau_3} \l_{n_3}^2} \big\| \ft{g'} \big\|_{L^2_\tau \l_n^2} \\
&\les \| w \|_{X^{s, \frac 12 + \eps}}^6,
\end{align*}

\noi
as desired.

\medskip \noi
\textbf{Case 3:} $|n_{1, 1}|, |n_{1, 2}| \gg \nm$ and $|n_{1, 3}| \les \nm$, or $|n_{1, 2}|, |n_{1, 3}| \gg \nm$ and $|n_{1, 1}| \les \nm$.\footnote{In the proof of Proposition 3.4 in \cite{kwak}, the Case B-3 contains a mistake since the bound for the largest modulation does not necessarily hold. This issue can be fixed by proceeding as in our Case 3.}

By symmetry between $n_{1, 1}$ and $n_{1, 3}$, we can assume that $|n_{1, 2}|, |n_{1, 3}| \gg \nm$ and $|n_{1, 1}| \les \nm$. Note that
\begin{align*}
\big| |n_{1, 1}|^\al - |n_{1, 2}|^\al + |n_{1, 3}|^\al - |n_1|^\al \big| &\ges |n_{1, 1} - n_{1, 2}| |n_{1, 2} - n_{1, 3}| |\nm'|^{\al - 2} \\
&\ges |n_{1, 2} - n_{1, 3}||\nm'|^{\al - 1}.
\end{align*}

\noi
If we have
\begin{align*}
\big| |n_1|^\al - |n_2|^\al - |n_3|^\al - |n|^\al \big| \ll |n_{1, 2} - n_{1, 3}| |\nm'|^{\al - 1},
\end{align*}

\noi
then we can use the same steps as in Case 2 to obtain the desired estimate. If we have
\begin{align*}
\big| |n_1|^\al - |n_2|^\al - |n_3|^\al - |n|^\al \big| \ges |n_{1, 2} - n_{1, 3}| |\nm'|^{\al - 1},
\end{align*}

\noi
then we can proceed as in Case 1 with $|n_{1, 1} - n' + n'' - n_3|$ being replaced by $|n'|$ and with $|\nm'| \ges |n_1 - n_2| = |n''|$. 
\end{proof}


\begin{corollary}
\label{Cor:qtc}
Let $\al>2$, $\frac{2 - \al}{6} \leq s < 0$, $0 < T \leq 1$, and $u_0 \in C^{\infty}(\T)$. 
Suppose that $w$ is a global smooth solution to \eqref{fnls3} with $w(0)=u_0$. Then,
for $0 < \eps \ll 1$ small enough, we have
\begin{align*}
 \norm{\RR_2(w)}_{L_T^2 H_x^s} \les \Big( \|u_0\|^4_{H^s}+  \| w \|_{X_T^{s, \frac 12 + \eps}}^4 +  \norm{w}_{X_T^{s,\frac 12 + \eps}}^6 \Big) \norm{w}_{X^{s,\frac 12 + \eps}_{T}} 
\end{align*}
\end{corollary}

\begin{proof}
The estimate follows directly from \eqref{eqI}, \eqref{EnN}, and Proposition \ref{Prop:quinc}.
\end{proof}

\begin{remark} \rm
\label{RMK:sum_n}
From the proof of Proposition \ref{Prop:quinc}, we can replace the supremum in $n \in \Z$ in~\eqref{quin1} by a summation in $n \in \Z$. 
\end{remark}

\begin{remark} \rm
\label{RMK:fub_ibp}
In Proposition \ref{Prop:quinc} and Corollary \ref{Cor:qtc}, the solution $w$ does not need to be smooth. The proof also works if we only require $w \in X_T^{s, \frac 12 + \eps}$, provided that we justify the use of Fubini's theorem and the integration by parts in \eqref{fub_ibp}.

To justify Fubini's theorem in \eqref{fub_ibp}, since $|t| \leq 1$, by Hausdorff-Young's inequality and (the $X^{s, b}$-version) of \eqref{tri_goal}, we have
\begin{align*}
\int_0^t &\sum_{\G(n)} \big| \ft w (t', n_1) \cj{\ft{w}} (t', n_2) \ft w (t', n_3) \cj{\ft{w}} (t', n) \big| dt' \\
&\leq \jb{n}^{-2s} \sum_{\G(n)} \jb{n}^{2s}  \int_{\tau_1 - \tau_2 + \tau_3 = \tau} \big| \ft w (\tau_1, n_1) \cj{\ft{w}} (\tau_2, n_2) \ft w (\tau_3, n_3) \cj{\ft{w}} (\tau, n) \big| d\tau_1 d\tau_2 d\tau_3 d\tau \\
&\les \| w \|_{X^{s, \frac 12 + \eps}}^3 \| w \|_{X^{s, \frac 12 - 2 \eps}} < \infty,
\end{align*}

\noi
which is a valid condition for applying Fubini's theorem.

To justify the integration by parts in \eqref{fub_ibp}, since we only worked with fixed spatial frequencies, we only need to show that for $f, g \in H^{\frac 12 + \eps} (\R)$ and $t \in \R$, we have
\begin{align*}
\int_{0}^{t} \frac{d}{dt'} f(t') \cdot g(t') dt' = f(t) g(t) - f(0) g(0) - \int_{0}^{t} f(t') \cdot \frac{d}{dt'} g(t') dt'.
\end{align*}

\noi
We first note that for any $h \in H^{\frac 12 + \eps} (\R)$ and $t \in \R$, by Plancherel's theorem, we obtain
\begin{align*}
\int_{0}^{t} \frac{d}{dt'} h(t') dt' = \int_\R i \tau \ft h (\tau) \ft{\ind_{[0, t]}} (-\tau) d\tau = \int_\R \ft h (\tau) \big( e^{i t \tau} - 1 \big) d\tau = h(t) - h(0),
\end{align*}

\noi
so that it suffices to show the product rule
\begin{align}
\label{prod}
\frac{d}{dt} (fg) = \frac{d}{dt} f \cdot g + f \cdot \frac{d}{dt} g.
\end{align}

\noi
For convenience, since we only care about the support of $f$ and $g$ on $[0, t]$ and $[0, t] \cong \T$, we view $f$ and $g$ as functions on $\T$. Let $\{f_n\}_{n \in \N}$ and $\{g_n\}_{n \in \N}$ be two sequences of smooth functions such that $f_n \to f$ and $g_n \to g$ in $H^{\frac 12 + \eps} (\T)$ as $n \to \infty$. Since $H^{\frac 12 + \eps} (\T)$ is an algebra, it is easy to see that $f_n g_n \to fg$ in $H^{\frac 12 + \eps} (\T)$, so that $\frac{d}{dt} (f_n g_n) \to \frac{d}{dt} (fg)$ in $H^{-\frac 12 + \eps} (\T)$. It remains to show that $\frac{d}{dt} f_n \cdot g_n \to \frac{d}{dt} f \cdot g$ in $H^{-\frac 12 + \eps} (\T)$ (the convergence $f_n \cdot \frac{d}{dt} g_n \to f \cdot \frac{d}{dt} g$ is similar). This follows directly from Lemma \ref{LEM:prod} and Sobolev's inequality, since
\begin{align*}
\Big\| \frac{d}{dt} f \cdot g \Big\|_{H^{-\frac 12 + \eps}} &\les \Big\| \jb{\nb}^{-\frac  12 + \eps} \frac{d}{dt} f \Big\|_{L^2} \big\| \jb{\nb}^{\frac 12 - \eps} g \big\|_{L^{\frac{2}{1 - 2\eps}}} \\
&\les \Big\| \frac{d}{dt} f \Big\|_{H^{-\frac 12 + \eps}} \| g \|_{H^{\frac 12 + \eps}}.
\end{align*}

\noi
Thus, we have \eqref{prod}, and so the integration by parts is justified. Note that \eqref{prod} also justifies the product rules applied in \eqref{fub_ibp} and \eqref{EnN}.
\end{remark}

\subsection{Local well-posedness on the circle}
\label{SEC:LWP2}
In this subsection, we prove Theorem \ref{TM:LWP2}, local well-posedness of the cubic fNLS \eqref{fnls3}. To prove local well-posedness, we show existence, uniqueness, and continuous dependence on the initial data.

\medskip \noi
\textbf{Existence of a solution:}

\smallskip
In this part, we follow \cite{MT} and show the existence of a solution to \eqref{fnls3} in $X_T^{s, \frac 12 + \eps}$ with initial data $u_0 \in  H^s (\T)$ for any $\frac{2 - \al}{6} \leq s < 0$, some $0 < T \leq 1$, and some $\eps > 0$ sufficiently small. Let $\{ u_{0,j} \}_{j \in \N}$ be a sequence in $C^{\infty}(\T)$
such that $u_{0,j} \to u_0$ in $H^s (\T)$ as $j \to \infty$. Let $w_j$ be the smooth global solutions to \eqref{fnls3} with $w_j|_{t = 0} = u_{0, j}$.
Then, $w_j$ satisfies
\begin{align}
\label{slow2}
\begin{split}
\ft{w_j}(t,n) &= e^{-it |n|^\al}\ft{u_{0,j}}(n) -i \int_0^te^{-i(t-t')|n|^\al}\sum_{\G(\cj n)}e^{it\Psi(\bar n)}  \ft{w_j}(t', n_1)\cj{\ft{w_j}} (t', n_2) \ft{w_j}(t', n_3) dt'  \\
&\quad + i\int_0^te^{-i(t-t') |n|^\al}(|\ft{w_j}(t', n)|^2-|\ft{u_{0,j}}(n)|^2) \ft{w_j}(t', n) dt'.
\end{split}
\end{align}

\noi
Let $K > 0$ be such that $\| u_{0, j} \|_{H^s} \leq K$ for all $j \in \N$. Let $0 < t \leq T \leq 1$. By Lemma \ref{LEM:lin}, Proposition \ref{Prop:triR}, and Corollary \ref{Cor:qtc}, we have
\begin{align*}
\| w_j \|_{X_t^{s, \frac 12 + \eps}} \leq CK + C T^\eps \| w_j \|_{X_t^{s, \frac 12 + \eps}}^3 + C T^\eps \Big( K^4 +  \| w \|_{X_t^{s, \frac 12 + \eps}}^4 +  \norm{w}_{X_t^{s,\frac 12 + \eps}}^6 \Big) \norm{w}_{X^{s,\frac 12 + \eps}_{t}}
\end{align*}

\noi
for $\eps > 0$ sufficiently small and some constant $C > 0$. By Lemma \ref{LEM:Xsb_cty} and choosing $T = T(K) > 0$ sufficiently small, a standard continuity argument yields
\begin{align}
\| w_j \|_{X_T^{s, \frac 12 + \eps}} \leq 2CK
\label{wj_bdd}
\end{align}

\noi
uniformly in $j \in \N$.

We now show that the sequence $\{w_j\}_{j \in \N}$ converges to a limit $w$, which satisfies \eqref{fnls3} with $w|_{t = 0} = u_0$. By \eqref{wj_bdd}, Lemma \ref{LEM:emb}, and the Banach-Alaoglu theorem, we can find a subsequence $\{ w_{j_k} \}_{k \in \N}$ such that
\begin{align*}
w_{j_k} \longrightarrow w \quad \text{weakly-* in } L^\infty ([-T, T]; H^s(\T)).
\end{align*}

\noi
Also, by using the Arzel\`a-Ascoli theorem and further extracting a subsequence, we obtain that $\ft{w_{j_k}} (n) \to \ft{w} (n)$ in $C([-T, T])$ for each $n \in \Z$. Indeed, for each $k \in \N$, $n \in \Z$, and $t_1, t_2 \in \R$, we have
\begin{align*}
| \ft{w_{j_k}} (t_1, n) - \ft{w_{j_k}} (t_2, n) | &\leq \jb{n}^{-s'} \bigg\| \int_\R \jb{n}^{s'} \ft{w_{j_k}} (\tau, n) (e^{i t_1 \tau} - e^{i t_2 \tau}) d\tau  \bigg\|_{\l_n^2} \\
&\les |t_1 - t_2|^\frac{\eps}{2} \jb{n}^{-s'} \big\| \jb{n}^{s'} \jb{\tau + |n|^\al}^{\frac 12 + \frac{\eps}{2}} |\tau|^{\frac{\eps}{2}} \ft{w_{j_k}} (\tau, n) \big\|_{\l_n^2 L_\tau^2} \\
&\les |t_1 - t_2|^\frac{\eps}{2} \jb{n}^{-s'} \| w_{j_k} \|_{X_T^{s, \frac 12 + \eps}}
\end{align*}

\noi
for some $s' < s$, which allows us to apply the Arzel\`a-Ascoli theorem due to \eqref{wj_bdd}. Thus, for any $N \in \N$, we have
\begin{align}
\pi_{\leq N} w_{j_k} \longrightarrow \pi_{\leq N} w \quad \text{in } C([-T, T]; H^s (\T)).
\label{convN}
\end{align}

\noi
Furthermore, using \eqref{wj_bdd} again, we can further extract a subsequence such that
\begin{align}
\jb{n}^s e^{it |n|^\al} \ft{w_{j_k}} (t, n) \longrightarrow \jb{n}^s e^{it |n|^\al} \ft{w} (t, n) \quad \text{weakly in } H^{\frac 12 + \eps} ([-T, T]; L^2 (\Z)).
\label{conv_weak}
\end{align}

\noi
This implies that
\begin{align*}
\| w \|_{X_T^{s, \frac 12 + \eps}} \leq \liminf_{k \to \infty} \| w_{j_k} \|_{X_T^{s, \frac 12 + \eps}} \leq 2CK,
\end{align*}

\noi
and so $w \in X_T^{s, b}$.

It remains to show that $w$ indeed satisfies \eqref{fnls3}. 
Let us rewrite \eqref{slow2} as
\begin{align}
\label{slow3}
\begin{split}
&\ft{w_{j_k}} (t, n) -e^{-it |n|^\al }\ft{u}_{0,j_k}(n)    \\
&\quad + i\int_0^t e^{-i(t-t') |n|^\al }\sum_{\G(n)}   e^{it\Psi(\cj n)}   \ft{\pi_{\leq N} w_{j_k}} (t', n_1) \cj{\ft{\pi_{\leq N} w_{j_k}}} (t', n_2) \ft{\pi_{\leq N} w_{j_k}}(t', n_3) dt'       \\
&\quad -i\int_0^t e^{-i(t-t') |n|^\al }(|\ft{w_{j_k}} (t', n)|^2-|\ft{u}_{0,j_k}(n)|^2)\ft{w}_{j_k}(t', n)  dt'     \\
&=-i\int_0^t e^{-i(t-t') |n|^\al }  \sum_{\G(n)}  e^{it\Psi(\cj n)}   \bigg(  \ft{w_{j_k}} (t', n_1) \cj{\ft{w_{j_k}}} (t', n_2) \ft{w_{j_k}}(t', n_3)     \\
&\quad -\ft{\pi_{\leq N} w_{j_k}}  (t', n_1) \cj{\ft{\pi_{\leq N} w_{j_k}}} (t', n_2) \ft{\pi_{\leq N} w_{j_k}} (t', n_3)    \bigg) dt'
\end{split}
\end{align}

\noi
We define the right-hand-side of \eqref{slow3} as $\ft{R_{j_k, N}} (t, n)$.
Note that for $s' < s$, we have
\begin{align*}
\| \pi_{\leq N} w_{j_k}  \|_{X^{s, \frac 12 + \eps}_T}&\leq \| w_{j_k}  \|_{X^{s, \frac 12 + \eps}_T},\\
\|w_{j_k} - \pi_{\leq N} w_{j_k}  \|_{X^{s', \frac 12 + \eps}_T}  &\leq     CN^{-(s-s')} \|  w_{j_k}  \|_{X^{s, \frac 12 + \eps}_T}
\end{align*}

\noi
Thus, by Proposition \ref{Prop:triR} and \eqref{wj_bdd}, we have that for $\max \{ \frac{2 - \al}{4}, \frac{1 - \al}{6} \} < s' < s$,
\begin{align}
\label{RN_bdd}
\| R_{j_k, N}  \|_{X^{s', \frac 12 + \eps}_T} &\les
\| w_{j_k}  \|_{X^{s', \frac 12 + \eps}_T}^2    \|w_{j_k} - \pi_{\leq N} w_{j_k}  \|_{X^{s', \frac 12 + \eps}_T} \les K^3 N^{- (s - s')}.
\end{align}



\noi
Going back to \eqref{slow3} and using \eqref{RN_bdd}, we obtain that for any $M \in \N$,
\begin{align*}
\bigg\|
\sum_{\substack{n \in \Z \\ |n| \leq M}} &\jb{n}^{s'} \bigg( 
\ft{w_{j_k}} (t,n) - e^{-it |n|^\al }\ft{u_{0,j_k}}(n) \\
&\quad +i\int_0^t e^{-i(t-t') |n|^\al }\sum_{\G(n)}  e^{it\Psi(\cj n)}    \ft{\pi_{\leq N} w_{j_k}} (t', n_1) \cj{\ft{\pi_{\leq N} w_{j_k}}} (t', n_2) \ft{\pi_{\leq N} w_{j_k}} (t', n_3) dt'  \\
&\quad -i\int_0^t e^{-i(t-t') |n|^\al }(|\ft{w_{j_k}}(t', n)|^2-|\ft{u_{0,j_k}}(n)|^2) \ft{w_{j_k}} (t', n) dt' \bigg) e^{inx}
\bigg\|_{L^2_{T} L^2_x} \notag \\
&\leq C K^3 N^{-(s-s')},
\end{align*}

\noi 
for $\max \{ \frac{2 - \al}{4}, \frac{1 - \al}{6} \} < s' < s$ and some constant $C > 0$. Since we have a finite sum, by \eqref{convN} and the dominated convergence theorem, we can let $j_k \to \infty$ and obtain
\begin{align}
\label{sum_nM}
\begin{split}
\bigg\|
\sum_{\substack{n \in \Z \\ |n| \leq M}} &\jb{n}^{s'} \bigg( 
\ft{w} (t,n) - e^{-it |n|^\al }\ft{u_{0}}(n) \\
&\quad +i\int_0^t e^{-i(t-t') |n|^\al }\sum_{\G(n)}  e^{it\Psi(\cj n)}    \ft{\pi_{\leq N} w} (t', n_1) \cj{\ft{\pi_{\leq N} w}} (t', n_2) \ft{\pi_{\leq N} w} (t', n_3) dt'  \\
&\quad -i\int_0^t e^{-i(t-t') |n|^\al }(|\ft{w}(t', n)|^2-|\ft{u_{0}}(n)|^2) \ft{w} (t', n) dt' \bigg) e^{inx}
\bigg\|_{L^2_{T} L^2_x}  \\
&\leq C K^3 N^{-(s-s')},
\end{split}
\end{align}
By the monotone convergence theorem, we can take $M \rightarrow \infty$ to get rid of the restriction ``$|n| \leq M$'' in \eqref{sum_nM}. To show that $w$ satisfies \eqref{fnls3}, it then remains to show
\begin{align*}
    \bigg\|
\sum_{n \in \Z} &\jb{n}^{s'} \bigg( i\int_0^t e^{-i(t-t') |n|^\al }\sum_{\G(n)}  e^{it\Psi(\cj n)}    \ft{w}(t', n_1) \cj{\ft{w}}(t', n_2) \ft{w}(t', n_3) dt' \\
& -i\int_0^t e^{-i(t-t') |n|^\al }\sum_{\G(n)}  e^{it\Psi(\cj n)}    \ft{\pi_{\leq N} w}(t', n_1) \cj{\ft{\pi_{\leq N} w}} (t', n_2)\ft{\pi_{\leq N} w}(t', n_3)  dt' \bigg) e^{inx}
\bigg\|_{L^2_{T} L^2_x} \longrightarrow 0
\end{align*}

\noi
as $N \to \infty$, which follows from a similar way as \eqref{RN_bdd}. This finishes the proof of existence of a solution to \eqref{fnls3} in $X_T^{s, b}$.

\medskip \noi
\textbf{Uniqueness of the solution:}

\smallskip
We now show the uniqueness of the solution to the cubic fNLS \eqref{fnls3}. 
Let $w_1$ and $w_2$ be solutions to \eqref{fnls3} in $X_T^{s, \frac 12 + \eps}$ with initial data $u_0 \in  H^s (\T)$ for $\frac{2 - \al}{6} \leq s < 0$, some $0 < T \leq 1$, and some $\eps > 0$ sufficiently small. Let $z = w_1 - w_2$, which satisfies
\begin{align*}
\begin{split}
\dt \ft{z}(n) &+ i|n|^\al \ft{z}(n) \\
&= i\big( |\ft{w_1}(n)|^{2}-|\ft{u_{0}}(n)|^{2} \big) \ft{z}(n)+ i \big( |\ft{w_1}(n)|^{2}-|\ft{w_2}(n)|^{2} \big)\ft{w_2}(n) \\
    &\quad -i\sum_{\G (n)} e^{it \Psi (\cj{n})} \big( \ft{z}(n_1)\cj{\ft{w_1}}(n_2)\ft{w_1}(n_3) + \ft{w_2}(n_1)\cj{\ft{z}}(n_2)\ft{w_1}(n_3) + \ft{w_2}(n_1) \cj{\ft{w_2}}(n_2)\ft{z}(n_3) \big)  \\
    & =: \ft{\1}(n)+\ft{\II}(n)+\ft{\III}(n)
\end{split}
\end{align*}

\noi
Thus, by using the Duhamel formulation and Lemma \ref{LEM:lin}, we have
\begin{align}
\|w\|_{X_{T}^{s,\frac{1}{2} + \eps}} \les T^\eps \Big( \| \1 \|_{L_{T}^{2} H^{s}} + \| \II \|_{L_{T}^{2} H^{s}} + \| \III \|_{X_{T}^{s,-\frac{1}{2} + 2\eps}} \Big)
\label{uni1}
\end{align}

\noi
For $\1$, by Corollary \ref{Cor:qtc} (in particular, Remark \ref{RMK:fub_ibp}), we have
\begin{align}
\| \1 \|_{L_T^2 H^s} \les C \Big( \| u_0 \|_{H^s}, \| w_1 \|_{X_T^{s, \frac 12 + \eps}} \Big) \| z \|_{X_T^{s, \frac 12 + \eps}}.
\label{uni2}
\end{align}

\noi
For $\III$, by Proposition \ref{Prop:triR}, we have
\begin{align}
\| \III \|_{X_T^{-\frac 12 + 2\eps}} \les C \Big(\| w_1 \|_{X_T^{s, \frac 12 + \eps}}, \| w_2 \|_{X_T^{s, \frac 12 + \eps}} \Big) \| z \|_{X_T^{s, \frac 12 + \eps}}.
\label{uni3}
\end{align}

We now deal with $\II$. For $|t| \leq T$, we note from \eqref{fNLS6} that
\begin{align*}
|\ft{w_1} (t, n)|^{2}-|\ft{w_2}(t, n)|^{2} &= \int_0^t \frac{d}{dt'} |\ft{w_1}(t', n)|^{2} - \frac{d}{dt'} |\ft{w_2}(t', n)|^{2} dt' \\
&= 2 \Re \int_0^t \dt \ft{w_1} (t', n) \cj{\ft{z}} (t', n) + \dt \ft{z} (t', n) \cj{\ft{w_2}} (t', n) dt' \\
&= 2 \Im \int_0^t \sum_{\G (n)} e^{it' \Psi (\cj{n})} \big( \ft{w_1} (t', n_1) \cj{\ft{w_1}} (t', n_2) \ft{w_1} (t', n_3) \cj{\ft{z}} (t', n) \\
&\quad + \ft{z} (t', n_1) \cj{\ft{w_1}} (t', n_2) \ft{w_1} (t', n_3) \cj{\ft{w_2}} (t', n) \\
&\quad + \ft{w_2} (t', n_1) \cj{\ft{z}} (t', n_2) \ft{w_1} (t', n_3) \cj{\ft{w_2}} (t', n) \\
&\quad + \ft{w_2} (t', n_1) \cj{\ft{w_2}} (t', n_2) \ft{z} (t', n_3) \cj{\ft{w_2}} (t', n)\big) dt'
\end{align*}

\noi
Thus, we can now use similar steps as in Proposition \ref{Prop:quinc} (justified in Remark \ref{RMK:fub_ibp}) to obtain
\begin{align*}
\sup_{\substack{|t| \leq T \\ n \in \Z}} |\ft{w_1} (t, n)|^{2}-|\ft{w_2}(t, n)|^{2} \les C \Big( \| u_0 \|_{H^s}, \| w_1 \|_{X_T^{s, \frac 12 + \eps}}, \| w_2 \|_{X_T^{s, \frac 12 + \eps}} \Big) \| z \|_{X_T^{s, \frac 12 + \eps}},
\end{align*}

\noi
so that
\begin{align}
\| \II \|_{L_T^2 H^s} \les C \Big( \| u_0 \|_{H^s}, \| w_1 \|_{X_T^{s, \frac 12 + \eps}}, \| w_2 \|_{X_T^{s, \frac 12 + \eps}} \Big) \| z \|_{X_T^{s, \frac 12 + \eps}}.
\label{uni4}
\end{align}

\noi
Combining \eqref{uni1}, \eqref{uni2}, \eqref{uni3}, and \eqref{uni4} and taking $T > 0$ to be sufficiently small, we obtain $\| z \|_{X_T^{s, \frac 12 + \eps}} = 0$, and so we finish the proof of uniqueness.

\medskip \noi
\textbf{Continuous dependence of solutions on initial data:}

\smallskip
Lastly, we prove the continuous dependence of solutions of \eqref{fnls3} on initial data. Again, we consider the solutions in $X_T^{s, \frac 12 + \eps}$ for $\frac{2 - \al}{6} \leq s < 0$, some $0 < T \leq 1$, and some $\eps > 0$ sufficiently small.
Let $u_0 \in H^s(\T)$ and $\{u_{0,j}\}_{j \in \N}$ be a sequence of initial data in $H^s(\T)$ such that 
\begin{align*}
u_{0,j} \to  u_0      \quad 
\text{in } H^s (\T).
\end{align*}

\noi
Let $w_j$ and $w$ to be the solutions of \eqref{fnls3} with $w_j|_{t = 0} = u_{0, j}$ and $w|_{t = 0} = u_0$. 
Note that the local existence time depends only on the $H^s$-norm of the initial data, so that we can assume that $w$ and $w_j$ exist on $[-T, T]$. By the existence part and the uniqueness part above, we have the convergence \eqref{convN} and the convergence \eqref{conv_weak} for the whole sequence $\{w_j\}_{j \in \N}$. We want to show that $w_j$ converges to $w$ strongly in $C([-T, T]; H^s (\T))$.

Let $\eta > 0$. By the convergence \eqref{conv_weak}, $\{w_j\}_{j \in \N}$ and $w$ are bounded in $X_T^{s, \frac 12 + \eps}$ uniformly in $j \in \N$. By Proposition \ref{Prop:quinc}, Remark \ref{RMK:sum_n}, and Remark \ref{RMK:fub_ibp}, for any $N \in \N$, we have
\begin{align}
\label{cont1}
\begin{split}
\| \pi_{> N} w_j \|_{C_T H^s}^2 &\leq \sum_{|n| > N} \jb{n}^{2s} |\ft{u_{0, j}} (n)|^2 + \sup_{|t| \leq T} \sum_{|n| > N} \jb{n}^{2s} \big| |\ft{w_j} (t, n)|^2 - |\ft{u_{0, j}} (n)|^2 \big| \\
&\les \| \pi_{> N} u_{0, j} \|_{H^s}^2 + N^{2s} C\Big( \| w_j \|_{X_T^{s, \frac 12 + \eps}} \Big)
\end{split}
\end{align}

\noi
and
\begin{align}
\label{cont2}
\begin{split}
\| \pi_{> N} w \|_{C_T H^s}^2 &\leq \sum_{|n| > N} \jb{n}^{2s} |\ft{u_{0}} (n)|^2 + \sup_{|t| \leq T} \sum_{|n| > N} \jb{n}^{2s} \big| |\ft{w} (t, n)|^2 - |\ft{u_{0}} (n)|^2 \big| \\
&\les \| \pi_{> N} u_{0} \|_{H^s}^2 + N^{2s} C\Big( \| w \|_{X_T^{s, \frac 12 + \eps}} \Big).
\end{split}
\end{align}

\noi
Also, for any $j \in \N$, we have
\begin{align}
\| \pi_{> N} u_{0, j} \|_{H^s} \leq \| u_{0, j} - u_0 \|_{H^s} + \| \pi_{> N} u_0 \|_{H^s}.
\label{cont3}
\end{align}

\noi
Thus, by \eqref{cont1}, \eqref{cont2}, and \eqref{cont3}, for $j \in \N$ and $N \in \N$ large enough, we have
\begin{align*}
\| \pi_{> N} w_j \|_{C_T H^s} < \frac{\eta}{3} \quad \text{and} \quad \| \pi_{> N} w \|_{C_T H^s} < \frac{\eta}{3}.
\end{align*}

\noi
Therefore, by the convergence \eqref{convN}, we have
\begin{align*}
\| w_j - w \|_{C_T H^s} &\leq \| \pi_{\leq N} w_j - \pi_{\leq N} w \|_{C_T H^s} + \| \pi_{> N} w_j \|_{C_T H^s} + \| \pi_{> N} w \|_{C_T H^s} \\
&< \frac{\eta}{3} + \frac{\eta}{3} + \frac{\eta}{3} = \eta,
\end{align*}

\noi
and so we finish the proof of continuous dependence on initial data.

\subsection{Global well-posedness on the circle}
\label{SEC:AClaw2}
In this subsection, we prove Theorem \ref{TM:GWP2}, global well-posedness of the cubic fNLS \eqref{fnls3} on $\T$ by employing the $I$-method. Our goal is to establish an a priori bound of the $H^s$-norm of the solution $w$ to \eqref{fnls3}. The proof is similar to the real line setting, but the main difference is that we do not have scaling symmetry on $\T$. Instead, in order to exploit the subcritical nature of the equation, we use the $H^s_M$-norm as defined in~\eqref{SobM}, where $M \in \N$ is to be chosen later.

In the following, we restrict our attention on $\frac{2-\al}{6} \leq s < 0$.
Given $N \geq 1$, we define a 
smooth even monotone function $m_M: \R \to [0, 1]$ such that
\begin{equation}
\label{iop2}
m_M(\xi) :=
\begin{cases}
1 & |\xi| < N \\
\Big( \frac{M + |\xi|}{N} \Big)^s & |\xi| > 2N,
\end{cases}
\end{equation}

\noi
where $N \in \N$ is a large parameter to be determined later. We then define the Fourier multiplier operator $I_M = I_{M, N} : H_M^s (\T) \to L^2 (\T)$ such that $\ft{I_M u} (n) = m_M(n) \ft u (n)$ for any $n \in \Z$.
As in the real line setting, 
$I_M$ acts as an identity operator on low frequencies and as a smoothing operator on high frequencies.
Note that we have the following bounds:
\begin{align}
\| u \|_{ H^s_M} \les  \|I_M u\|_{L^2}
\les  N^{-s}\| u \|_{ H^s_M}.
\label{I2}
\end{align}

We first apply the $I_M$-operator to our cubic fNLS \eqref{fnls3} and obtain
\begin{align}
\label{Isys2}
\begin{cases}
i \dt  I_M w = I_M D^\al w + I_M \NN_2(w) - I_M \RR_2 (w) \\
I_M w|_{t= 0} = I_M u_0,
\end{cases} 
\end{align}

\noi
where $\NN_2$ and $\RR_2$ are defined in \eqref{NN1O} and \eqref{NN2O}, respectively.
Below, we show a local well-posedness result for the $I$-system \eqref{Isys2} as a variant of local well-posedness of \eqref{fnls3}.

\begin{proposition}
\label{PROP:ILWP2}
Let $\al>2$, $\frac{2-\al}{6} \leq s <0$, and $1 \leq M \les N^2$. 
Then, given any $u_0 \in H^s (\T)$, there exists $T = T (\| I_M u_0 \|_{L^2}) > 0$
and a unique solution $I_M w\in C([-T, T] ; L^2(\T) ) $ to the $I$-system \eqref{Isys2} such that $ I_M w\in C( [0,\dl];L^2(\T) ) $.  
Furthermore, there exists $\eta> 0$ independent of $M$ and $N$ such that for any $t_0 \in \R$ and $w(t_0) \in H^s (\T)$ satisfying
\begin{align*}
\| I_M w (t_0) \|_{L^2} \leq \eta,
\end{align*}

\noi
the solution $I_M w$ to the $I$-system \eqref{Isys2} exists on the interval $J = [t_0 - 1, t_0 + 1]$ with the bound
\begin{align*}
\| I_M w \|_{X_J^{0, \frac 12 + \eps}} \leq 2 \eta
\end{align*}

\noi
for some $\eps > 0$ sufficiently small.
\end{proposition}

\begin{proof}
The proof of local well-posedness of \eqref{Isys2} follows essentially from the compactness argument in Subsection \ref{SEC:LWP2}, once we show the following two estimates:
\begin{align}
\big\| I_M \big( \NN_2 (w_1, w_2, w_3) \big) \big\|_{X_T^{0, - \frac 12 + 2\eps}} \les \prod_{j = 1}^3 \| I_M w_j \|_{X_T^{0, \frac 12 + \eps}},
\label{IN2}
\end{align}
\begin{align}
\| I_M \RR_2 (w) \|_{L_T^2 L_x^2} \les \Big( \| I_M u_0 \|_{H^s}^4 + \| I_M w \|_{X_T^{0, \frac 12 + \eps}}^4 + \| I_M w \|_{X_T^{0, \frac 12 + \eps}}^6 \Big) \| I_M w \|_{X_T^{0, \frac 12 + \eps}}.
\label{IR2}
\end{align}

The non-resonant estimate \eqref{IN2} follows from Proposition \ref{PROP:triR}, as long as the following estimate holds
\begin{align}
\frac{m_M (n_4) \jb{n_4}^{-s}}{m_M (n_1) \jb{n_1}^{-s} m_M (n_2) \jb{n_2}^{-s} m_M (n_3) \jb{n_3}^{-s}} \les 1
\label{IN2-1}
\end{align}

\noi
under the condition $n_1 - n_2 + n_3 = n_4$. Since the largest two frequencies are comparable, we can assume without loss of generality that $|n_1| \ges |n_4|$, so that $m_M (n_4) \jb{n_4}^{-s} \les m_M (n_1) \jb{n_1}^{-s}$. Also, since $M \les N^2$, we have $m_M (n_2) \jb{n_2}^{-s} \ges 1$ and $m_M (n_3) \jb{n_3}^{-s} \ges 1$. Thus, \eqref{IN2-1} holds.

For \eqref{IR2}, we note that
\begin{align*}
\| I_M \RR_2 (w) \|_{L_T^2 L_x^2} \leq \Big( \sup_{\substack{|t| \leq T \\ n \in \Z}} \big| |\ft w (t, n)|^2 - |\ft{u_0} (n)|^2 \big| \Big) \| I_M w \|_{L_T^2 L_x^2}.
\end{align*}

\noi
The supremum can then be estimated using \eqref{EnN} and Proposition \ref{Prop:quinc}, given that $m_M (m) \jb{m}^{-s} \ges 1$ for any $m \in \Z$. This proves \eqref{IR2}, and so local well-posedness of \eqref{Isys2}.
\end{proof}

We now proceed in a similar way as in Subsection \ref{SEC:Imd}. Recall the notations in Subsection \ref{SUBSEC:notations}. Let $w$ be a smooth solution of the cubic fNLS \eqref{fnls3}. We define the modified mass by
\begin{align*}
M (I_M w (t)) = \| I_M w (t) \|_{L_x^2}^2 = \sum_{n \in \Z} m_M (n)^2 |\ft{w} (t, n)|^2 .
\end{align*}

\noi
By differentiating $M (I_M w)$ in time, using \eqref{fNLS6}, and symmetrizing the frequencies, we obtain
\begin{align}
\frac{d}{dt} M(I_M w (t)) &= \Ld_4 \big(  e^{it \Psi (n_1, \dots, n_4)}  M_{4, M} (n_1, \dots, n_4) \ind_{\{ n_1 \neq n_2, n_2 \neq n_3 \}}; w(t) \big),
\label{dMM}
\end{align}

\noi
where
\begin{align*}
\Psi (n_1, \dots, n_4) &= |\ft{u_0} (n_1)|^2 - |\ft{u_0} (n_2)|^2 + |\ft{u_0} (n_3)|^2 - |\ft{u_0} (n_4)|^2, \\
M_{4, M} (n_1, \dots, n_4) &= \frac{i}{2} \big( m_M (n_1)^2 - m_M (n_2)^2 + m_M (n_3)^2 - m_M (n_4)^2 \big).
\end{align*}

We define the multiplier
\begin{align}
\s_{4, M} (n_1, \dots, n_4) := -\frac{i M_{4, M} (n_1, \dots, n_4) \ind_{\{ n_1 \neq n_2, n_2 \neq n_3 \}}}{|n_1|^\al - |n_2|^\al + |n_3|^\al - |n_4|^\al}.
\label{sigma4M}
\end{align}

\noi
We now define the following new modified mass with a correction term:
\begin{align}
M^4 (I_M w (t)) := M(I_M w (t)) + \Ld_4 (e^{it \Psi} \s_{4, M} ; w(t)).
\label{M4M}
\end{align}

\noi
By using \eqref{fNLS6} and the symmetry of frequencies, we can compute that
\begin{align*}
\frac{d}{dt} \Ld_4 (\s_4; w(t)) &= - \Ld_4 \big(  e^{it \Psi (n_1, \dots, n_4)}  M_{4, M} (n_1, \dots, n_4) \ind_{\{ n_1 \neq n_2, n_2 \neq n_3 \}}; w(t) \big) \\
&\quad + 4 \Re i \Ld_6 \big( e^{it \Psi (n_1, n_2, n_3, n_4 - n_5 + n_6)} e^{it \Psi (n_4, n_5, n_6, n_4 - n_5 + n_6)} \\
&\qquad \times \s_{4, M} (n_1, n_2, n_3, n_4 - n_5 + n_6); w (t) \big) \\
&\quad + 4 \Re i \sum_{\substack{n_1 - n_2 + n_3 - n_4 = 0 \\ n_1 \neq n_2, n_2 \neq n_3}} e^{it \Psi (n_1, n_2, n_3, n_4)} \s_{4, M} (n_1, n_2, n_3, n_4) \\
&\qquad \times |\ft{w} (t, n_1)|^2 \ft{w} (t, n_1) \cj{\ft{w}} (t, n_2) \ft{w} (t, n_3) \cj{\ft{w}} (t, n_4)
\end{align*}

\noi
Thus, by differentiating \eqref{M4M} and using \eqref{dMM}, we obtain
\begin{align}
\begin{split}
\frac{d}{dt} M^4 (I_M u (t)) &= 4 \Re i \Ld_6 \big( e^{it \Psi (n_1, n_2, n_3, n_4 - n_5 + n_6)} e^{it \Psi (n_4, n_5, n_6, n_4 - n_5 + n_6)} \\
&\qquad \times \s_{4, M} (n_1, n_2, n_3, n_4 - n_5 + n_6); w (t) \big) \\
&\quad + 4 \Re i \sum_{\substack{n_1 - n_2 + n_3 - n_4 = 0 \\ n_1 \neq n_2, n_2 \neq n_3}} e^{it \Psi (n_1, n_2, n_3, n_4)} \s_{4, M} (n_1, n_2, n_3, n_4) \\
&\qquad \times |\ft{w} (t, n_1)|^2 \ft{w} (t, n_1) \cj{\ft{w}} (t, n_2) \ft{w} (t, n_3) \cj{\ft{w}} (t, n_4) \\
&=: \1 + \II.
\end{split}
\label{dM4M}
\end{align}

Our next step is to establish the following estimate.
\begin{lemma}
\label{LEM:Mdiff2}
Let $\al > 2$, $\frac{2 - \al}{6} \leq s < 0$, $1 \leq M \les N^2$, and $w$ be a smooth solution of the cubic fNLS \eqref{fnls3} on an interval $J \subset \R$. Then, for any $t \in J$, we have
\begin{align}
\big| M^4 (I_M w (t)) - M(I_M w (t)) \big| \les \| I_M w (t) \|_{L^2 (\T)}^4
\label{Mdiff2_goal}
\end{align}
\end{lemma}

\begin{proof}
Recall that $\s_{4, M}$ is as defined in \eqref{sigma4M}. By using similar steps as in Lemma \ref{LEM:si4R} (more precisely, \cite[Lemma 4.1]{kihoon}), we know that if $|n_j| \sim N_j$ ($j = 1, 2, 3, 4$) for $N_j \geq 1$ dyadic numbers and $N^{(1)}, N^{(2)}, N^{(3)}, N^{(4)}$ is a rearrangement of $N_1, N_2, N_3, N_4$ such that $N^{(1)} \geq N^{(2)} \geq N^{(3)} \geq N^{(4)}$, we have
\begin{align}
|\s_{4, M} (n_1, \dots, n_4)| \les \frac{m_M (N^{(4)})^2}{(N + N^{(3)}) (N + N^{(1)})^{\al - 1}}
\label{s4M_bdd}
\end{align}

\noi
under $n_1 - n_2 + n_3 - n_4 = 0$, $n_1 \neq n_2$, and $n_2 \neq n_3$. Thanks to \eqref{s4M_bdd}, the estimate \eqref{Mdiff2_goal} then follows from almost the same steps as in Lemma \ref{LEM:Mdiff}, given that we have $\jb{n}^s \les m_M (n) \leq 1$ for all $n \in \Z$ since $M \les N^2$.
\end{proof}

We now show the following almost conservation law for the modified mass $M^4 (I_M w)$. 

\begin{proposition}
\label{PROP:aclaw2}
Let $\al > 2$, $\frac{2-\al}{6} \leq s < 0$, $\eps > 0$, $1 \leq M \les N^2$, and $w$ be a smooth solution of the cubic fNLS \eqref{fnls3} on an interval $J \subset \R$. Then, for any $t, t_0 \in J$, we have
\begin{align*}
\big| M^4 (I_M w (t)) - M^4 (I_M w (t_0))  \big| \les N^{\max\{ -\frac 32 \al + 2 - 6s, -\al - 6 s \} + \dl} \| I_M w \|_{X_J^{0, \frac 12 + \eps}}^6
\end{align*}

\noi
for $\dl > 0$ arbitrarily small.
\end{proposition}

\begin{proof}
By \eqref{dM4M}, we have
\begin{align*}
\big| M^4 (I_M w (t)) - M^4 (I_M w (t_0))  \big| &= \bigg| \int_{t_0}^t \frac{d}{dt} M^4 (I_M w (t')) dt' \bigg| \\
&\leq 4 \bigg| \int_{t_0}^t \1 dt' \bigg| + 4 \bigg| \int_{t_0}^t \II dt' \bigg|.
\end{align*}
For simplicity below, for a dyadic number $K \geq 1$, we denote $w_K = \pi_K w$.


For the $\II$ term, we use a dyadic decomposition so that $|n_j| \sim N_j$ for some dyadic numbers $N_j \geq 1$ for $j = 1, 2, 3, 4$. We denote by $N^{(1)}, N^{(2)}, N^{(3)}, N^{(4)}$ a rearrangement of $N_1, N_2, N_3, N_{4}$ such that $N^{(1)} \geq N^{(2)} \geq N^{(3)} \geq N^{(4)}$. By \eqref{s4M_bdd}, the fact that $m(N_j)^{-1} \les N_j^{-s}$, Plancherel's theorem, Young's convolution inequalities, H\"older's inequalities, Lemma \ref{LEM:emb}, and the $L^4$-Strichartz estimate (Lemma \ref{LEM:L4str}), we obtain
\begin{align*}
\bigg| \int_{t_0}^t \II dt' \bigg| &\les \sum_{\substack{N_1, \dots, N_4 \geq 1 \\ \text{dyadic}}} \frac{1}{(N + N^{(3)}) (N + N^{(1)})^{\al - 1} m_M (N_1)^3 \prod_{j = 2}^4 m_M (N_j)} \\
&\quad \times \big\| \ft{I_M w_{N_1}} \big\|_{L_\tau^1 \l_n^3}^3 \big\| \ft{I_M w_{N_2}} * \ft{I_M w_{N_3}} \big\|_{L_\tau^2 \l_n^2} \big\| \ft{I_M w_{N_4} \ind_{[t_0, t]}} \big\|_{L_\tau^2 \l_n^2} \\
&\les \sum_{\substack{N_1, \dots, N_4 \geq 1 \\ \text{dyadic}}} \frac{(N^{(1)})^{-6s}}{(N + N^{(3)}) (N + N^{(1)})^{\al - 1}} \| I_M w_{N_1} \|_{X_J^{0, \frac 12 + \eps}}^3 \\
&\quad \times \| I_M w_{N_2} \|_{L_J^4 L_x^4} \| I_M w_{N_3} \|_{L_J^4 L_x^4} \| I_M w_{N_4} \|_{X^{0, \frac 12 + \eps}} \\
&\les N^{-\al - 6s + \dl} \| I_M w \|_{X_J^{0, \frac 12 + \eps}}^6,
\end{align*}

\noi
where we used $\al - 1 + 6s > 0$ given $\frac{2 - \al}{6} \leq s < 0$. This is the desired bound.

For the $\1$ term, we use a dyadic decomposition so that $|n_j| \sim N_j$ for some dyadic numbers $N_j \geq 1$ for $j = 1, \dots, 6$ and also $|n_4 - n_5 + n_6| \sim N_{456}$. We denote by $N^{(1)}, N^{(2)}, N^{(3)}, N^{(4)}$ a rearrangement of $N_1, N_2, N_3, N_{456}$ such that $N^{(1)} \geq N^{(2)} \geq N^{(3)} \geq N^{(4)}$. We also define $N_{\text{max}} = \max \{N_1, \dots, N_6\}$. We can assume that $N_{\text{max}} \ges N$, since otherwise we have $\s_{4, M} \equiv 0$ and so the estimate follows directly. By symmetry, we can assume that $N_1 \geq N_2 \geq N_3$ and $N_4 \geq N_5 \geq N_6$. We consider the following two cases. 

\medskip \noi
\textbf{Case 1:} $N_1 \ges N_4$. 

In this case, we have $N_1 \sim N^{(1)} \sim N_{\text{max}}$. Thus, by 
\eqref{s4M_bdd}, the fact that $m_M (N_j)^{-1} \les N^s(M+N_{\text{max}})^{-s}$, Plancherel's theorem, H\"older's inequality, the $L^6$-Strichartz estimate (Lemma~\ref{LEM:L6str}), and Lemma \ref{LEM:emb}, we obtain
\begin{align*}
\bigg| \int_{t_0}^t \1 dt' \bigg| &\les \sum_{\substack{N_1, \dots, N_6, N_{456} \geq 1 \\ \text{dyadic}}} \frac{1}{(N + N^{(3)}) (N + N^{(1)})^{\al - 1} \prod_{j = 1}^6 m_M (N_j)} \\
&\quad \times \int_{t_0}^t \int_{\T} \prod_{j = 1}^5 \Big| \F_{t, x}^{-1} \big( \big| \ft{I_M w_{N_j}} \big| \big) \Big| \cdot \Big| \F_{t, x}^{-1} \big( \big| \ft{I_M w_{N_6} \ind_{[t_0, t]}} \big| \big) \Big| dx dt' \\
&\les \sum_{\substack{N_1, \dots, N_6, N_{456} \geq 1 \\ \text{dyadic}}} \frac{N^{6s} ( M^{-6s} + N_{\text{max}}^{-6s} )}{(N + N^{(3)}) (N + N^{(1)})^{\al - 1}}  \\
&\quad \times \prod_{j = 1}^5 \Big\| \F_{t, x}^{-1} \big( \big| \ft{I_M w_{N_j}} \big| \big) \Big\|_{L_J^6 L_x^6} \cdot \Big\| \F_{t, x}^{-1} \big( \big| \ft{I_M w_{N_6} \ind_{[t_0, t]}} \big| \big) \Big\|_{L_J^6 L_x^6} \\
&\les \sum_{\substack{N_1, \dots, N_6, N_{456} \geq 1 \\ \text{dyadic}}} \frac{N^{6s - 12s} + N^{6s} N_{\text{max}}^{-6s}}{(N + N^{(3)}) (N + N^{(1)})^{\al - 1}} \prod_{j = 1}^6 \| I_M w_{N_j} \|_{X_J^{0, \frac 12 + \eps}} \\
&\les N^{-\al - 6s + \dl} \| I_M w \|_{X_J^{0, \frac 12 + \eps}}^6,
\end{align*}

\noi
where we used $N_3 \les N^{(3)}$, $N_6 \les N^{(1)}$, $M \les N^2$, and $\al - \frac 32 + 6s > 0$ given $\frac{2 - \al}{6} \leq s < 0$ and $\al > 2$. This is the desired bound.

\medskip \noi
\textbf{Case 2:} $N_1 \ll N_4$. 

In this case, since the two largest frequencies are comparable, we have $N_4 \sim N_5 \sim N_{\text{max}} \ges N^{(1)}$, and so we need to rely on the smoothing from the maximum modulation as in Case 2 of Proposition \ref{Prop:quinc}. For convenience, we denote $n_{456} = n_4 - n_5 + n_6$. Let $\nm = \max\{ n_1, n_2, n_3, n_{456} \}$ and $\nm' = \max\{ n_4, n_5, n_6, n_{456} \}$. Note that we have 
\begin{align}
\big| |n_1|^\al - |n_2|^\al + |n_3|^\al - |n_{456}|^\al \big| \les |\nm|^\al \sim N_1^\al.
\label{mod1}
\end{align}

\noi
By Lemma \ref{LEM:factor}, we have
\begin{align}
\big| |n_4|^\al - |n_5|^\al + |n_6|^\al - |n_{456}|^\al \big| \ges |n_4 - n_5| |n_4 - n_{456}| |\nm'|^{\al - 2} \ges N_4^{\al - 1}.
\label{mod2}
\end{align}

\medskip \noi
\textbf{Subcase 2.1:} $N_1^\al \ll N_4^{\al - 1}$.

By \eqref{mod1} and \eqref{mod2}, we can assume without loss of generality that (the other cases are similar and easier)
\begin{align*}
|\tau_6 + |n_6|^\al| = \max_{j = 1, \dots, 6} \{ |\tau_j + |n_j|^\al| \} \ges N_4^{\al - 1}.
\end{align*}

\noi
Thus, by \eqref{s4M_bdd}, Plancherel's theorem, Young's convolution inequalities, H\"older's inequalities, the $L^4$-Strichartz estimate (Lemma \ref{LEM:L4str}), and Lemma \ref{LEM:time}, we obtain
\begin{align*}
\bigg| \int_{t_0}^t \1 dt' \bigg| &\les \sum_{\substack{N_1, \dots, N_6, N_{456} \geq 1 \\ \text{dyadic}}} \frac{N^{6s} (M^{-3s} + N_1^{-3s}) (M^{-3s} + N_4^{-3s})}{(N + N^{(3)}) (N + N^{(1)})^{\al - 1}} \big\| \ft{I_M w_{N_4}} * \ft{I_M w_{N_5}} \big\|_{L_\tau^2 \l_n^2} \\
&\quad \times \big\| \ft{I_M w_{N_1}} \big\|_{L_\tau^{1} \l_n^1} \big\| \ft{I_M w_{N_2}} \big\|_{L_\tau^{1} \l_n^1} \big\| \ft{I_M w_{N_3}} \big\|_{L_\tau^1 \l_n^1} \big\| \ft{I_M w_{N_6} \ind_{[t_0, t]}} \big\|_{L_\tau^{2} \l_n^2} \\
&\les \sum_{\substack{N_1, \dots, N_6, N_{456} \geq 1 \\ \text{dyadic}}} \frac{N^{6s} (M^{-3s} + N_1^{-3s}) (M^{-3s} + N_4^{-3s}) N_1^{\frac 12} N_2^{\frac 12} N_3^{\frac 12}}{(N + N^{(3)}) (N + N^{(1)})^{\al - 1} N_4^{\frac{\al - 1}{2} -}}  \| I_M w_{N_4} \|_{L_J^4 L_x^4}  \\
&\quad \times \| I_M w_{N_5} \|_{L_J^4 L_x^4} \| I_M w_{N_1} \|_{X_J^{0, \frac 12 + \eps}} \| I_M w_{N_2} \|_{X_J^{0, \frac 12 + \eps}} \| I_M w_{N_3} \|_{X_J^{0, \frac 12 + \eps}} \| I_M w_{N_6} \ind_{[t_0, t]} \|_{X^{0, \frac 12 -}} \\
&\les \sum_{\substack{N_1, \dots, N_6, N_{456} \geq 1 \\ \text{dyadic}}} \frac{N^{6s} (N^{-6s} + N_1^{-3s}) (N^{-6s} + N_4^{-3s}) }{(N + N^{(3)})^{\frac 12} (N + N^{(1)})^{\al - 2} N_4^{\frac{\al - 1}{2} -}} \prod_{j = 1}^6 \| I_M w_{N_j} \|_{X_J^{0, \frac 12 + \eps}} \\
&\les N^{-\frac 32 \al + 2 - 6s + \dl} \| I_M w \|_{X_J^{0, \frac 12 + \eps}}^6,
\end{align*}

\noi
where we used $N_1 \les N^{(1)}$, $N_2 \les N^{(1)}$, $N_3 \les N^{(3)}$, $M \les N^2$, $N_4 = N_{\text{max}} \ges N$, and $\frac{\al - 1}{2} + 3s > 0$ and $\al - 2 + 3s > 0$ given $\frac{2 - \al}{6} \leq s < 0$ and $\al > 2$. This is the desired bound.

\medskip \noi
\textbf{Subcase 2.2:} $N_1^\al \ges N_4^{\al - 1}$.

In this case, we proceed as in Case 1. Since $N_1 \ll N_4$ and $N^{(1)} \sim N_1 \ges N_4^{\frac{\al - 1}{\al}}$, the dyadic coefficient part becomes
\begin{align*}
&\sum_{\substack{N_1, \dots, N_6, N_{456} \geq 1 \\ \text{dyadic}}} \frac{1}{(N + N^{(3)}) (N + N^{(1)})^{\al - 1} \prod_{j = 1}^6 m_M (N_j)} \\
&\les \sum_{\substack{N_1, \dots, N_6, N_{456} \geq 1 \\ \text{dyadic}}} \frac{N^{6s} (M^{-3s} + N_1^{-3s}) (M^{-3s} + N_4^{-3s})}{(N + N^{(3)}) (N + N^{(1)})^{\al - 1}} \\
&\les N^{-\al + (6 - 6a)s + \dl} + N^{6s} \sum_{\substack{N_1, \dots, N_6, N_{456} \geq 1 \\ \text{dyadic}}} \frac{M^{-3s} + N_1^{-3s}}{(N + N^{(3)}) (N + N^{(1)})^{\al - 1 + \frac{3s \al}{\al - 1}}} \\
&\les N^{-\al - 6s + \dl} + N^{-\al  - \frac{3\al}{\al - 1}  s + \dl},
\end{align*}

\noi
where we used $N_1 \les N^{(1)}$, $N_4 = N_{\text{max}} \ges N$, $M \les N^2$, and $\al - 1 + \frac{3s \al}{\al - 1} + 3s > 0$ given $\frac{2 - \al}{6} \leq s < 0$ and $\al > 2$. Note that $-\frac{3\al}{\al - 1} s \leq -6s$, so that we have the desired estimate.
\end{proof}

We are now ready to prove Theorem \ref{TM:GWP2}.

\begin{proof}[Proof of Theorem \ref{TM:GWP2}]
For simplicity, we only consider well-posedness on the time interval $[0, T]$, and the argument for $[-T, 0]$ is the same.

From \eqref{I2} and the definition of the $H^s_M$-norm in \eqref{SobM}, we have
\begin{align*}
\| I_M u_0 \|_{L^2} \les N^{-2s} \| u_0 \|_{H^{2s}_M} \les N^{-2s} M^s \| u_0 \|_{H^s}.
\end{align*}

\noi
By letting $M \sim N^2$, we get
\begin{align*}
\| I_M u_0 \|_{L^2} < \eta_0 \leq \frac{\eta}{2},
\end{align*}

\noi
where $\eta > 0$ is as given in Proposition \ref{PROP:ILWP2}. 

We can now apply the same iteration scheme as in the proof of Theorem \ref{TM:GWP1} in Subsection~\ref{SEC:Imd}, using Lemma \ref{LEM:Mdiff2} and Proposition \ref{PROP:aclaw2}. This allows us to obtain that, for $k \in \N$ satisfying
\begin{align*}
k N^{\max\{ -\frac 32 \al + 2 - 6s, - \al - 6s \} + \dl} \sim 1
\end{align*}

\noi
we have
\begin{align*}
\sup_{t \in [0, k]} \| w (t) \|_{H_M^s} \les \sup_{t \in [0, k]} \| I_M w(t) \|_{L^2} \leq \eta.
\end{align*}

\noi
Thus, we have
\begin{align}
\sup_{t \in [0, k]} \| w (t) \|_{H^s} \les M^{-s} \sup_{t \in [0, k]} \| w (t) \|_{H_M^s} \les N^{-2s} \eta < \infty.
\label{sob_bdd1}
\end{align}

\noi
It remains to choose
\begin{align}
N^{\min \{ \frac 32 \al - 2 + 6s, \al + 6s \} - \dl} \sim T,
\label{sob_bdd2}
\end{align}

\noi
which can be achieve by choosing $N \gg 1$ large enough since $\frac 32 \al - 2 + 6s - \dl > 0$ and $\al + 6s - \dl > 0$ given $\frac{2 - \al}{6} \leq s < 0$, $\al > 2$, and $\dl > 0$ sufficiently small. This finishes the proof of Theorem~\ref{TM:GWP2}.
\end{proof}

\begin{remark} \rm
Note that from \eqref{sob_bdd1} and \eqref{sob_bdd2}, we obtain the following growth bound of the Sobolev norm of the solution $w$ to the cubic fNLS \eqref{fnls3}:
\begin{align*}
\sup_{t \in [0, T]} \| w (t) \|_{H^s} \les T^\be \| u_0 \|_{H^s},
\end{align*}

\noi
where
\begin{align*}
\be = \frac{-2s}{\min \{ \frac 32 \al - 2 + 6s, \al + 6s \} - \dl} =
\begin{cases}
\frac{-2s}{\frac 32 \al - 2 + 6s - \delta} & \text{if } \al \leq 4 \\
\frac{-2s}{\al + 6s - \dl} & \text{if } \al > 4.
\end{cases}
\end{align*}
\end{remark}

\begin{ackno} \rm 
The authors would like to thank Tadahiro Oh for proposing this problem and his continuing support. The authors are also grateful to Yuzhao Wang and Chenmin Sun for helpful
discussions. 
E.B. would like to thank the School of Mathematics at the University of Edinburgh for its hospitality during his visit, where part of this manuscript was prepared. E.B. was supported by Unit\'e de Math\'ematiques Pures et Appliqu\'ees UMR 5669 ENS de Lyon / CNRS. G.L., R.L., and Y.Z. were supported by the European Research Council (grant no.~864138 ``SingStochDispDyn''). G.L. was also supported by the EPSRC New Investigator Award (grant no.~EP/S033157/1). Y.Z. was also partially funded by the chair of probability and PDEs at EPFL.
\end{ackno}

\end{document}